\documentclass[reqno,11pt]{amsart}
\usepackage{amsmath, latexsym, amsfonts, amssymb, amsthm, amscd}
\usepackage{graphics,epsf,psfrag}
\usepackage{graphics,epsf,psfrag}
\usepackage[utf8]{inputenc}
\usepackage{stmaryrd}
\usepackage{mathabx}
\usepackage{hyperref}
\usepackage{xcolor}
\usepackage[normalem]{ulem}

\newcommand{\lint}{\llbracket}
\newcommand{\rint}{\rrbracket}

\newcommand{\bt}{\mathbf{t}}
\newcommand{\bx}{\mathbf{x}}

\newcommand{\bc}{\mathbf{c}}
\newcommand{\bu}{\mathbf{u}}
\numberwithin{equation}{section}

\newtheorem{theorema}{Theorem}

\newtheorem{theorem}{Theorem}[section]
\newtheorem{lemma}[theorem]{Lemma}

\newtheorem{proposition}[theorem]{Proposition}
\newtheorem{cor}[theorem]{Corollary}
\newtheorem{rem}[theorem]{Remark}

\newcommand{\dd}{\mathrm{d}}
\newcommand{\ind}{\mathbf{1}}
\newcommand{\restrict}[1]{\raise-.3ex\hbox{\big|}_{#1}}

\newcommand{\R}{\mathbb{R}}

\renewcommand{\tilde}{\widetilde}
\renewcommand{\hat}{\widehat}
\newcommand{\cc}{\complement}

\newcommand{\cX}{{\ensuremath{\mathcal X}} }
\newcommand{\cA}{{\ensuremath{\mathcal A}} }
\newcommand{\cB}{{\ensuremath{\mathcal B}} }
\newcommand{\cF}{{\ensuremath{\mathcal F}} }

\newcommand{\cP}{{\ensuremath{\mathcal P}} }
\newcommand{\cE}{{\ensuremath{\mathcal E}} }

\newcommand{\cC}{{\ensuremath{\mathcal C}} }
\newcommand{\cN}{{\ensuremath{\mathcal N}} }
\newcommand{\cL}{{\ensuremath{\mathcal L}} }
\newcommand{\cT}{{\ensuremath{\mathcal T}} }
\newcommand{\cD}{{\ensuremath{\mathcal D}} }

\newcommand{\cZ}{{\ensuremath{\mathcal Z}} }

\newcommand{\cK}{{\ensuremath{\mathcal K}} }
\newcommand{\cM}{{\ensuremath{\mathcal M}} }

\newcommand{\bP}{{\ensuremath{\mathbf P}} }
\newcommand{\bQ}{{\ensuremath{\mathbf Q}} }
\newcommand{\bE}{{\ensuremath{\mathbf E}} }

\newcommand{\bC}{{\ensuremath{\mathbf C}} }

\newcommand{\cG}{{\ensuremath{\mathcal G}} }

\DeclareMathSymbol{\leqslant}{\mathalpha}{AMSa}{"36} 
\DeclareMathSymbol{\geqslant}{\mathalpha}{AMSa}{"3E} 
\DeclareMathSymbol{\eset}{\mathalpha}{AMSb}{"3F}     
\renewcommand{\leq}{\, \leqslant\,}                   

\newcommand{\Var}{\mathrm{Var}}        
\newcommand{\suptwo}[2]{\sup_{\substack{#1 \\ #2}}} 

\DeclareMathOperator*{\supp}{\mathrm{Support}}

\newcommand{\bbE}{{\ensuremath{\mathbb E}} }

\newcommand{\bbL}{{\ensuremath{\mathbb L}} }

\newcommand{\bbN}{{\ensuremath{\mathbb N}} }

\newcommand{\bbP}{{\ensuremath{\mathbb P}} }

\newcommand{\bbR}{{\ensuremath{\mathbb R}} }

\newcommand{\bbX}{{\ensuremath{\mathbb X}} }

\newcommand{\bbZ}{{\ensuremath{\mathbb Z}} }
\newcommand{\bZ}{{\ensuremath{\mathbf Z}} }

\newcommand{\ga}{\alpha}
\newcommand{\gb}{\beta}

\newcommand{\gep}{\varepsilon}       

\newcommand{\gG}{\Gamma}

\newcommand{\gD}{\Delta}

\newcommand{\go}{\omega}

\newcommand{\gl}{\lambda}

\newcommand{\ups}{\upsilon}

\makeatletter
\def\captionfont@{\footnotesize}
\def\captionheadfont@{\scshape}

\long\def\@makecaption#1#2{%
  \vspace{2mm}
  \setbox\@tempboxa\vbox{\color@setgroup
    \advance\hsize-6pc\noindent
    \captionfont@\captionheadfont@#1\@xp\@ifnotempty\@xp
        {\@cdr#2\@nil}{.\captionfont@\upshape\enspace#2}%
    \unskip\kern-6pc\par
    \global\setbox\@ne\lastbox\color@endgroup}%
  \ifhbox\@ne 
    \setbox\@ne\hbox{\unhbox\@ne\unskip\unskip\unpenalty\unkern}%
  \fi
  \ifdim\wd\@tempboxa=\z@ 
    \setbox\@ne\hbox to\columnwidth{\hss\kern-6pc\box\@ne\hss}%
  \else 
    \setbox\@ne\vbox{\unvbox\@tempboxa\parskip\z@skip
        \noindent\unhbox\@ne\advance\hsize-6pc\par}%
\fi
  \ifnum\@tempcnta<64 
    \addvspace\abovecaptionskip
    \moveright 3pc\box\@ne
  \else 
    \moveright 3pc\box\@ne
    \nobreak
    \vskip\belowcaptionskip
  \fi
\relax
}
\makeatother
\def\writefig#1 #2 #3 {\rlap{\kern #1 truecm
\raise #2 truecm \hbox{#3}}}

\title{The continuum directed polymer in L\'evy Noise}

\author{Quentin Berger}
\address{
LPSM, Sorbonne Universit\'e,  UMR 8001\\
Campus Pierre et Marie Curie, Bo\^ite courrier 158, 4 Place Jussieu, 75252 Paris Cedex 05, France.
}
\email{quentin.berger@sorbonne-universite.fr}

\author{Hubert Lacoin}
\address{
  IMPA, Institudo de Matem\'atica Pura e Aplicada, Estrada Dona Castorina 110
Rio de Janeiro, CEP-22460-320, Brasil. 
}
\email{lacoin@impa.br}


\begin{document}

\begin{abstract}
We present in this paper the construction of a continuum  directed polymer model in an environment given by  space-time L\'evy noise. 
One of the main objectives of this construction is to describe the scaling limit of discrete directed polymer in an heavy-tail environment and for this reason we put special emphasis on the case of $\alpha$-stable noises with $\alpha \in (1,2)$.
Our construction can be performed in arbitrary dimension, provided that the L\'evy measure satisfies specific (and dimension dependent) conditions.
We also discuss a few basic properties of the continuum polymer  and the relation between 
this model and the Stochastic Heat Equation with multiplicative
Lévy noise.
\end{abstract}

\maketitle

\tableofcontents


\section{Introduction}

The aim of this paper is to build a continuum model which describes the scaling limit of directed polymers in $\bbZ^d$ with an environment which has infinite second moment: 
the continuum directed polymer in a  space-time L\'evy noise.
Our construction can be thought as an extension to arbitrary noise and dimension of that presented in \cite{AKQ14b} of a continuum polymer in dimension $1$ with Gaussian white noise.
In a companion paper \cite{BL20_disc}, we prove that the scaling limit of the directed polymer in $\bbZ^d$ with heavy tailed environment is indeed the continuum model constructed in the present paper.

Whereas the construction in \cite{AKQ14b}  is directly based on the solution of the Stochastic Heat Equation (SHE) with multiplicative noise, our approach here needs to be slightly different since the solution of SHE with a general L\'evy noise (see \cite{Chong17} for recent developpements) does not display sufficient regularity. Our continuum model is thus defined via a martingale approximation of the noise obtained by truncating the ``small jumps'' part of the noise. This construction is not specific to directed polymers  and can possibly be applied to describe the scaling limit of a wide variety of disordered models with heavy tailed noise, including the disordered pinning model (see  \cite{CSZ14,CSZ13} for the construction of the corresponding Gaussian scaling limits).

In order to motivate our construction, we provide a brief introduction to the directed polymer model, the notion of its scaling limit and review some literature on the subject.

\subsection{Directed polymer in a random environment (the discrete model)}
\label{sec:discrete}

Let us consider $\eta=(\eta_{n,x})_{n\in \bbN, x\in \bbZ^d}$ a \emph{discrete} $(1+d)$-dimensional field of i.i.d.\ random variables, with law denoted by $\bbP$.
 We assume that 
 \begin{equation} \label{standard}
  \bbP[\eta \ge -1]=1 \quad \text{ and } \quad \bbE[\eta]=0 \, .
 \end{equation}
With some harmless abuse of notation, we let $\eta$ denote a generic random variable with the same law as $\eta_{n,x}$.
We consider the following $(1+d)$-dimensional (discrete) directed polymer model, in environment $(\eta_{n,x})_{n\in \bbN, x\in \bbZ^d}$. 
Let $S=(S_i)_{i\geq 0}$ be the simple symmetric random walk on $\bbZ^d$, with law denoted by $\bP$.
Given a parameter $\beta\in (0,1)$ (which allows to tune the disorder's intensity)  we define the  partition function $Z^{\eta}_{N,\beta}$ by
\begin{equation}
\label{def:Zn}
Z^{\eta}_{N,\beta}:= \bE\Big[\prod_{n=1}^N \big( 1+\beta \eta_{n,S_n} \big) \Big] \, ,
\end{equation}
and the associated polymer (Gibbs) measure $\bP^{\eta}_{N,\beta}$  by
\begin{equation}\label{polymers}
 \frac{ \dd \bP^{\eta}_{N,\beta}}{\dd \bP}(S):=\frac{1}{Z^{\eta}_{N,\beta}} \prod_{n=1}^N \big( 1+\beta \eta_{n,S_n} \big).
\end{equation}
The environment $\eta$ can be thought as a field of impurities, and under 
$\bP^{\eta}_{N,\beta}$ the law of the random walk is modified so that it favors visits to (space-time) sites where $\eta$ assumes a larger value. Assumptions \eqref{standard} are merely practical:
they ensure that $1+\beta \eta_{n,S_n}$  is always positive and imply that $\bbE [ Z^{\eta}_{N,\beta}]=1$.

The directed polymer model has a long history, dating back to~\cite{HH85}, see \cite{C17} for an extensive review. In many directed polymer references (including \cite{C17}) the setup is slightly different and the Gibbs weights  are rather written in an exponential form $\exp( \beta \sum_{n=1}^{N}  \tilde\eta_{n,S_n})$ instead of  $\prod_{n=1}^N \big( 1+\beta  \eta_{n,S_n} \big)$ used here.
For most purposes the two formalisms are equivalent, but the latter turns out to be the adequate one for the specific problem we wish to study (we discuss this point later in the introduction, see Remark~\ref{rem:weights}).

\subsubsection*{Localization transition}

A major point of focus in the directed polymer model has been the localization transition from a high temperature diffusive phase (small $\beta$) to a  low temperature localized phase (large~$\beta$).
This phase transition can be studied via the free-energy $p(\beta):= -\lim_{N\to \infty}\frac{1}{N} \bbE [\log Z^{\eta}_{N,\beta}]$;
we refer to \cite[Prop.~2.5]{CSY03} for a proof of its existence.
The free-energy is a non-negative, non-decreasing and continuous function of $\beta\in (0,1)$ 
(see \cite[Thm.~3.2]{CSY06} for a proof, \cite[Thm.~A.1]{Vi19} for its adaptation to the setup presented here). In particular there exists a critical value $\beta_c\in [0,1]$ which is such that  $p(\gb)=0$ if and only if $\beta \leq \gb_c$.

\smallskip
This phase transition has been mostly studied in the case where the environment has a finite second moment
$\bbE[\eta^2]<\infty$. In the exponential setup, this corresponds to having $\bbE[e^{2\beta\tilde \eta}]<\infty$ (the standard assumption considered in the literature is that $\eta$ has exponential moment of all orders see e.g.~\cite{AKQ10})
and under this assumption it has been show that $\beta_c>0$ when $d\ge 3$, in \cite{Bol89,IS88}, while $\beta_c=0$ when $d=1$ \cite{CV06} and $d=2$ \cite{Lac10pol}. In particular this implies that there is no observable transition in dimension $d=1$ and $d=2$.

\subsubsection*{Intermediate disorder regime and scaling limit}

When $d \le 2$, under a finite second moment assumption (assuming that $\beta>0$ and $\Var(\eta)>0$), we have 
\begin{equation}\label{legitinter}
\lim_{N \to \infty} Z^{\eta}_{N,\beta}= 0 \quad \text{ and } \quad
\lim_{\beta\to 0} Z^{\eta}_{N,\beta}=1 \, .
\end{equation}
A legitimate question is therefore to know how
to scale $\beta$ with $N$ (or $N$ with $\beta$) in order to observe a non-trivial random behavior for  $Z^{\eta}_{N,\beta_N}$  and $\bP^{\eta}_{N,\beta}$ in the limit $N\to\infty$.

\smallskip

This problem has been the object of a large number of works \cite{AKQ10,  CSZ14, CSZ13} (see the review~\cite{CSZreview} and references therein). When $d=1$, the correct scaling is to take $\beta$ proportional to $N^{-1/4}$
---~note that in this case, $N$ is proportional to the correlation length of the system which is given by $|p(\beta)|^{-1}\asymp \beta^{-4}$ see \cite{AY15, Nak16}.
The limit is formally obtained by replacing the random walk path and its environment by their scaling limit, which are respectively given by Brownian Motion and space-time White Noise.
In particular, the scaling limit of the partition function $\lim_{N\to \infty} Z^{\eta}_{N, \hat \beta N^{-1/4}}$ is intimately related to the solution of the Stochastic Heat Equation (SHE) with multiplicative noise \cite{CB95}.

\smallskip

The case of the dimension $d=2$  presents additional difficulty as the SHE with multiplicative noise is ill-defined, so that the heuristic picture we had in dimension $d=1$ cannot be valid.
 For a hierarchical version of the model
the scaling limit of the polymer measure when $\beta_N$ is sent to zero at the appropriate rate is identified in  \cite{clark2019}. The original problem for the model on $\bbZ^2$ (and its continuum counterpart on $\bbR^2$) is still partially open but it has witnessed substantial progress in the recent years \cite{CSZ15, CSZ18scaling, gu2019}. A decisive step toward the identification of the scaling limit, \textit{i.e.}\ the convergence of the partition functions for the right value of $\beta_N$, has be made in a recent breakthrough paper \cite{CSZ21}.
More precisely it is shown that taking
 $\beta_N = \frac{\sqrt{\pi}}{\sqrt{\log N}} (1+  \frac{\vartheta}{\log N})$ for some $\vartheta\in \bbR$,
the rescaled
field of point-to-point partition functions
  $N Z^{\eta}_{N, \beta_N}(\sqrt{\lceil Nx\rceil},\sqrt{\lceil Ny \rceil}))_{x\in \bbR^d, y\in \bbR^d}$ converges in distribution   when $N\to \infty$ to a non-trivial limit, called the \emph{critical $2D$ stochastic heat flow}.
  Note that in this case also, the choice for $\beta_N$ is such that the corresponding correlation length $|p(\beta_N)|^{-1}$ is proportional to $N^{1+o(1)}$,
 since $\log |p(\beta)|  \sim - \frac{\pi}{\gb^2}$ as $\gb\downarrow0$ (see \cite{BL17}).

\subsubsection*{Heavy tailed disorder}

Our main motivation is to investigate intermediate disorder limits beyond the case $\bbE[ \eta^2]<\infty$. 
Our interest lies in the case where 
 $\eta$ is in the domain of attraction of an $\alpha$-stable law for $\alpha\in (1,2)$ and that \eqref{standard} still holds (we can also consider the case $\alpha\in(0,1]$ if one drops the assumption that $\eta$ has zero average). To be more specific, let us assume the tail distribution has a pure power-law decay, \textit{i.e.} that in the large $z$ limit we have
\begin{equation}
\label{def:eta}
\bbP(\eta> z)  = z^{-\alpha}(1+o(1))  \, .
\end{equation}

This kind of heavy tail environment has been studied in \cite{Vi19}. In this case, the existence of a non-trivial weak disorder phase  depends on $\alpha$ and the dimension $d$. We have $\gb_c=0$
if and only if $d\leq \frac{2}{\ga-1}$,
see \cite[Thm.~1.1]{Vi19}.
Moreover, when $d< \frac{2}{\alpha-1}$,
the behavior of the free energy near criticality (that is, for $\beta$ small) is given by   
 $p(\beta)=\beta^{\nu+o(1)}$ with $\nu =\frac{2\alpha}{2-d(\alpha-1)}$.

\smallskip
One of our main goal is to identify the intermediate disorder scaling limit of this model under the assumption~\eqref{def:eta}, when  $\alpha<1+\frac{2}{d}$, \textit{i.e.} $\alpha\in (0,1]$ or $\alpha \in (1,2)$ and $d<\frac{2}{\alpha-1}$.
We present in this paper the construction of the continuum measure that appears as the limit of $\bP_{N,\gb_N}^{\eta}$ in the intermediate disorder regime. The convergence of the discrete model to the continuum one, when $\beta_N$ goes to $0$
at some adequate rate, is the object of a separate work~\cite{BL20_disc},
see Theorem~\ref{thm:conv} below.

\begin{rem}
\label{rem:weights}
Let us  stress that
directed polymers in heavy-tail random environment are also considered in \cite{AL11,BT19,DZ16}:
the main difference is that in these papers the Gibbs weights are written under the exponential form $\exp(\gb \sum_{n=1}^N \tilde \eta_{n,S_n})$.
When the second moment of $\tilde \eta_{n,S_n}$ is infinite,
such a model exhibits very strong localization properties:
polymer trajectories remain in the neighborhood of a single favorite trajectory which visits the high enery sites (see~\cite{AL11,BT19}). 
Also, the intermediate disorder regime is somehow trivial
in this case. Indeed, in~\cite{BT19} the authors show that there is a specific scaling at which 
 a sharp \emph{weak-to-strong disorder} transition occurs.
Under this scaling, there is a (random) threshold below which
the partition function  goes to $1$ and above which it goes to $+\infty$ (see \cite[Thms~2.7-2.8]{BT19}
for a more precise statement).
 For a fixed value of $\beta$, the two setups, exponential $(e^{\gb \tilde \eta})$ and additive $(1+\beta\eta)$ are equivalent, since one can be obtained from the other via a simple transformation. Of course such a correspondence between the two setup disappears when studying the limit $\beta\to 0$, with a fixed distribution for $\eta$ or $\tilde \eta$. In the heavy-tail environment case, the additive setup $(1+\beta \eta)$ is the one which keeps a balance between the randomness of the random walk and the rewards of the environment and allows the existence of a scaling limit in which both are retained.
This comes from the fact the field $\gb \eta_{n,x}$ converges after scaling (as a distribution) to a non trivial limit ---~this is never the case for $\exp(\gb \tilde \eta_{n,x})$, even after centering, because large values of $\tilde \eta$ create too wild fluctuations.


\end{rem}

\subsection{An informal definition of a continuum polymer with L\'evy noise}

Before stating our main result concerning the intermediate disorder regime in $\alpha$-stable environment, we need to provide a description of the scaling limit. The object we construct is formally obtained by 
considering a Feynman--Kac formula where the random walk and the environment are replaced by their respective scaling limits.

\smallskip

The scaling limit of our random walk is a Brownian Motion with covariance matrix~$\frac{1}{d} I_d$ where $I_d$ is the identity matrix in $\bbR^d$.
To define the continuum polymer,
we rather consider a \emph{standard} $d$-dimensional Brownian motion $(B_t)_{t\in  [0,T]}$ (for practical reason it is convenient to define $B$ only until a fixed finite time horizon $T$).
 We let $\bQ$ denote the associated distribution (we omit the dependence in $T$ to lighten notation) on the Wiener space 
\begin{equation}\label{defwiener}
C_0([0,T]):= \big\{ \varphi  \colon  [0,T]\to \bbR^d \ : \ \varphi \text{ is continuous and } \varphi(0)=0 \big\} \,,
\end{equation}
endowed with the topology of uniform convergence and the associated Borel $\sigma$-algebra.
\smallskip

When $\eta$ has a finite second moment, the scaling limit for the environment is given by a  space-time Gaussian white noise. In that case a Brownian polymer model  in dimension $d=1$ can be (and has been) constructed based on the solution of the Stochastic Heat Equation, see \cite{AKQ14b}.
On the other hand, in the case where \eqref{def:eta} is satisfied for some $\alpha\in (0,2)$,
we have to consider a different object, namely
the space-time $(1+d)$-dimensional  $\alpha$-stable noise with L\'evy measure supported on $\bbR_+$. This is the multidimensional analog of the derivative of the $\alpha$-stable process with only  positive jumps.
This is a well studied object, see \cite{FFU17} and references therein, but we try to offer here a short and self-contained introduction for the sake of completeness.
For simplicity, we focus our exposition on the
case $\ga\in (1,2)$, which displays the most interesting phenomenology. However we also treat below a much more general class of noise which includes the case  
$\alpha \in (0,1]$.

\subsubsection*{One-sided $\alpha$-stable noise in $\bbR\times \bbR^d$}
Given $\alpha\in(1,2)$, we start with a Poisson
point process~$\go$ on $\bbR\times \bbR^d\times \bbR_+$ (time, space, and value of disorder) with intensity
\begin{equation}
\label{Poissondens}
\dd t \otimes \dd x \otimes \alpha \ups^{-(1+\alpha)} \dd \ups \, ,
\end{equation}
which is obtained as the scaling limit of the extremal process associated with $(\eta_{n,x})_{n \in \bbN, x\in \bbZ^d}$ satisfying~\eqref{def:eta}.
As it shall draw no confusion the distribution of $\go$ is also denoted by~$\bbP$.
Our $\alpha$-stable L\'evy noise $\xi_{\go}$ is the random distribution which is formally obtained by summing weighted Dirac masses $\ups\,\delta_{(t,x)}$  corresponding to all the points $(t,x,\ups)\in \go$ and subtracting a non-random quantity so that the obtained distribution is centered in expectation. The delicate part is that, as in the definition of $\alpha$-stable processes, the counter term that has to be substracted is infinite.

Let us thus explain how $\xi_{\go}$ can be obtained using a limiting procedure.
We consider~$\go$ as a set of points, and for any $a\in(0,1]$ we define $\go^{(a)}:=\{ (t,x,\ups)\in \go \colon \ups\ge a\}$  the \emph{truncated} environment,  \textit{i.e.}\ removing  atoms (jumps) of size less than $a$.
We then let  $\xi_{\go}^{(a)}$ be the random measure on $\bbR\times \bbR^d$ defined by 
\begin{equation}
\label{def:omegabar}
 \xi_{\go}^{(a)}:= \bigg(\sum_{(t,x,\ups)\in \go} \ups\ind_{\{\ups \ge a\}} \delta_{(t,x)} \bigg)- \frac{ \alpha (a^{1-\alpha}-1)}{\alpha-1} \, \cL \, ,
\end{equation}
where $\cL$ denotes Lebesgue measure on $\bbR\times \bbR^d$ 
(note that our centering only compensates the jumps of intensity smaller than one, so that $\xi_{\go}^{(a)}$ is 
not centered).
We define $\xi_{\go}$ as the distributional limit of~$ \xi_{\go}^{(a)}$ when~$a$ tends to zero.

For the sake of fixing ideas, let us specify a functional space in which  this convergence holds.
Given $s\in \bbR$, the Sobolev space $H^s(\bbR^{1+d})$ is defined as the closure of the space of smooth compactly supported function with respect to the norm
\begin{equation}
\label{Hsnorm}
 \|f\|_{H^s}:=  \left( \int_{\bbR^{1+d}}(1+|z|^2)^{s} |\hat f(z)|^2 \dd z \right)^{1/2},
\end{equation}
where $\hat f(z)= \int_{\bbR^{d+1}} f(x) e^{-i z\cdot x} \dd x$ is the Fourier transform of $f$.
We also consider the local Sobolev space 
\[
H^{s}_{\mathrm{loc}}(\bbR^{1+d}):= \big\{ f \ : \    f \psi \in H^{s} \text{ for every $C^{\infty}$ compactly supported $\psi$} \big\} \,,
\]
considered with the topology induced by the family of semi-norms
$\| \psi f\|_{H^{s}}$ indexed by $\psi$.
We then have the following (standard) result:
When $\alpha\in (1,2)$, then
 $\xi_{\go}^{(a)}$ converges almost surely in  $H^{-s}_{\mathrm{loc}}(\bbR^{1+d})$ with $s>(1+d)/2$ towards a limit $\xi_{\go} \in H^{-s}_{\mathrm{loc}}(\bbR^{1+d})$, see Proposition~\ref{convergence2}.
In particular, this means that~$\xi_{\go}$ can be integrated against any function in $H^s(\bbR^{1+d})$ which has compact support.  

\subsubsection*{Informal description of the scaling limit}

In order to describe the candidate scaling limit of the model \eqref{polymers} 
we must make sense of a Feynman--Kac formula analogous to \eqref{def:Zn}
in which the random walk $S$ is replaced by a Brownian motion $B$
and $\eta$ replaced by the $\ga$-stable noise $\xi_{\go}$.
Similarly to \eqref{polymers} we wish to define a polymer model which is a modification of the Wiener  Brownian measure $\bQ$ obtained via tilting by an energy functional. For $T>0$ and $\beta>0$ we would like to define $\bQ^{\go}_{T,\beta}$ as   
\begin{equation}\label{wewant}
 \frac{\dd \bQ^{\go}_{T,\beta}}{\dd \bQ}(B)= \frac{1}{\cZ^{\go}_{T,\beta}} \, \mathbf{:}e^{\beta H_\go(B)}\mathbf{:}  \, ,
\end{equation}
where the energy functional is given by $\xi_{\go}$ integrated against the Brownian trajectory,
 in the following sense ($\delta_{(s,y)}$ denotes the dirac mass at $(s,y)$)
\begin{equation}
      H_{\go}(B)= \xi_{\go}\left( \int_{0}^T \delta_{(t,B_t)} \, \dd t\right).
       \end{equation}
At this stage, we only consider this expression at a formal level, as it is quite clear that the fact that
$\xi_{\go}\in H_{\rm loc}^{-s}$ is not sufficient to provide a mathematical interpretation of this expression.
       
 \smallskip      
       
The exponential $\mathbf{:}e^{\beta H_\go(B)}\mathbf{:}$ is to be interpreted as an analogous of the time-ordered Wick exponential which is considered for the continuum directed polymer in white noise, see~\cite{AKQ14b}.
Informally, $\mathbf{:}e^{\beta H_\go(B)}\mathbf{:} $ is defined via the following expansion 
\begin{equation}\label{giancarlo}
\mathbf{:}e^{\beta H_\go(B)}\mathbf{:}\;   = \sum_{k=0}^\infty \beta^k\int_{0<t_1<\dots<t_k<T} \prod_{i=1}^k\xi_{\go} ( \delta_{(t_i, B_{t_i}) } \dd t_i)  \, .
\end{equation}
While it is challenging to make sense of the above formula, things become simpler if one looks at the partition function $\cZ^{\go}_{T,\beta}=\bQ\left[ \mathbf{:}e^{\beta H_\go(B)}\mathbf{:}   \right]$, because 
a formal integration with respect to Brownian trajectories makes the integrand more regular.
Let us denote
\begin{equation}\label{defrhotx}
 \rho_t(x):=\frac{1}{(2\pi t)^{d/2}}e^{-\frac{\|x\|^2}{2t}}
\end{equation}
the heat-kernel on $\bbR^d$ associated with the Brownian motion ($\|x\|$ stands for the Euclidean norm of $x$). 
For $0<t_1<\cdots < t_k$ and $x_1, \ldots, x_k \in \bbR^d$, we also use the short-hand notation 
\begin{equation}
\label{def:rho}
\varrho(\bt,\bx) := \prod_{i=1}^{k} \rho_{t_i-t_{i-1}} (x_i-x_{i-1})\, ,
\end{equation}
with by convention $t_0=0$ and $x_0=0$
(in the following, if a different choice is made it will be duly notified).
We will also use the notation $\dd \bt$ and $\dd \bx$ for Lebesgue measure on $\bbR^k$ and $(\bbR^d)^k$ respectively.
With these notation, the expectation of 
\eqref{giancarlo} with respect to the Wiener measure can be formally defined by 
\begin{equation}\label{lexpress}
\cZ^{\go}_{T,\beta}=1+\sum_{k=1}^\infty \beta^k\int_{0<t_1<\dots<t_k<T}\int_{(\bbR^d)^k} \varrho(\bt,\bx)  \prod_{i=1}^k \xi_{\go} ( \dd t_i, \dd x_i).
\end{equation}
In order to give a meaning to the above expression, we will approximate $\xi_{\go}$ by $\xi_{\go}^{(a)}$ and investigate the limiting behavior when $a$ goes to zero.
As it will be seen later, giving a meaning to $\cZ^{\go}_{T,\beta}$ is the most important step in order to give a rigorous interpretation to~\eqref{wewant}.

\section{Model and results}

We can now introduce our results. We present in Sections \ref{sec:mainres}-\ref{sec:generalnoise} our construction of the continuum measure $\bQ_{T,\gb}^{\go}$,
thus defining the continuum directed polymer in L\'evy noise.
For pedagogical reason, we first present in Section~\ref{sec:mainres} the case of the $\alpha$-stable noise with $\ga\in (1,2)$, since it corresponds 
to the scaling limit of the model introduced in Section~\ref{sec:discrete} above;
we turn afterwards in Section~\ref{sec:generalnoise}  to the case of a general heavy-tail noise. 
In Section~\ref{sec:mainprop}, 
we present finer properties of the
measure constructed
and
in Section~\ref{sec:SHE} we discuss the 
relation between our model
and the Stochastic Heat Equation with multiplicative Lévy noise.
Further comments on the results are
made in Section~\ref{sec:comments}.

\subsection{The construction of the continuum polymer in L\'evy $\alpha$-stable noise}
\label{sec:mainres}



Our main result is the construction of a measure on the Wiener space $C_0([0,T])$,
corresponding to the definition~\eqref{wewant}.
To ease the exposition, we single out the most important step of this construction which is the construction of the partition function,
that is giving a mathematical interpretation for the formal integral \eqref{lexpress}.
As mentioned above, we treat the case of an $\alpha$-stable noise first,
before we turn to more general noises.
%

Recall the definition~\eqref{def:omegabar}
of the \textit{truncated} noise~$\xi_{\go}^{(a)}$.
We define, for $a>0$,
\begin{equation}\label{lafirstdef}
\cZ^{\go,a}_{T,\beta}:=1+\sum_{k=1}^\infty \beta^k\int_{0<t_1<\dots<t_k<T}\int_{(\bbR^d)^k} \varrho(\bt,\bx) \prod_{i=1}^k \xi_{\go}^{(a)} ( \dd t_i, \dd x_i).
\end{equation}
Since $\xi_{\go}^{(a)} ( \dd t_i, \dd x_i)$ is a locally finite signed measure, the only possible issue with the above definition is the integrability over $t_i$'s and $x_i$'s and summability over $k$. These conditions are not difficult to check and this done in Proposition~\ref{prop:contiZfini}.
It is also not immediate from~\eqref{lafirstdef} that 
$\cZ^{\go,a}_{T,\beta}$ is positive (which is a required property for being a partition function), but this is ensured by Lemma~\ref{eazy}.

\smallskip

We prove that considering the limit of   $\cZ^{\go,a}_{T,\beta}$ when $a\downarrow0$, we obtain a \emph{non-trivial} (\textit{i.e.}\ disordered) quantity, provided that $\alpha$ is smaller than a critical threshold.
Let us define
\begin{equation}
\label{def:alphac}
\alpha_c=\alpha_c(d)  =
\begin{cases}
2 & \quad \text{if } d=1,2 \, ,\\
1+ \frac{2}{d}& \quad \text{if } d\ge 3 \, .
\end{cases}
\end{equation}

\begin{theorem}
\label{thm:continuous}
If $\alpha\in (1,\alpha_c)$ with $\alpha_c$ defined in \eqref{def:alphac}, 
there exists an almost surely positive random variable $\cZ^{\go }_{T, \beta}$
such that the following convergence 
\begin{equation*}
 \lim_{a \to 0}\cZ^{\go,a }_{T, \beta}=\cZ^{\go }_{T, \beta}
\end{equation*}
holds almost surely and in $\bbL_1$. 
When  $d\ge 3$ and $\alpha\in [\alpha_c,2)$ then for all $\beta>0$ we have 
$\lim_{a \to 0}\cZ^{\go,a }_{T, \beta}=0$ almost surely.
\end{theorem}

\begin{rem} 
Note that the definitions \eqref{def:omegabar} and \eqref{lafirstdef} also make sense when $\alpha\ge 2$. In that case $\xi^{(a)}_{\go}$ does not converge to a limiting distribution but this does a priori prevent $\cZ^{\go,a}_{T,\beta}$ from having a non-trivial limit.
 Proposition~\ref{prop:tozero} below shows that we have in fact $\lim_{a\downarrow 0}\cZ^{\go,a }_{T, \beta} = 0$ for every  $\alpha\in [\alpha_c,\infty)$ in any dimension $d\ge 1$.
\end{rem}

Let us now present the construction of the polymer measure described in \eqref{wewant}.
Recall that our objective is to define a probability on
the Wiener space $C_0([0,T])$ 
which corresponds to the formal definition \eqref{wewant}. 
We proceed in a similar manner as with the partition function: we first consider a measure on $C_0([0,T])$ built with the truncated noise~$\xi_{\go}^{(a)}$. 
Let us introduce the following families of functions on  $C_0([0,T])$:
\begin{equation}\label{thespaces}
\begin{split}
 \cB&:= \left\{ \, f\colon C_0([0,T]) \to \bbR \ : \ f \text{ measurable and bounded} \, \right\},\\
 \cC&:= \left\{ \, f \colon C_0([0,T]) \to \bbR \ : \ f \text{ continuous and bounded}\, \right\},\\
 \cB_b&:= \left\{ \, f\in \cB \ : \ \supp(f) \text{ is  bounded}\,  \right\},\\
\cC_b&:= \left\{  \, f\in \cC \ : \ \supp(f) \text{ is  bounded}\, \right\}\,.
\end{split}
\end{equation}
 Recall that $C_0([0,T])$ is equipped with the topology of the supremum norm: hence we say that $f\colon C_0([0,T]) \to \bbR$ has bounded support if there exists $M>0$ such that   $f(\varphi) =0$ for any $\varphi \in C_0([0,T])$ with $\|\varphi\|_{\infty} >M$.

Given  a bounded Borel-measurable function $f\in \cB$, we define 
\begin{equation}\label{lafirstdeff}
 \cZ^{\go,a}_{T,\beta}(f)=\bQ(f)+\sum_{k=1}^\infty \beta^k\int_{0<t_1<\dots<t_k<T}\int_{(\bbR^d)^k} \varrho( \bt,\bx , f) \prod_{i=1}^k \xi_{\go}^{(a)} ( \dd t_i, \dd x_i)\,,
\end{equation}
 where we use the notation $\bQ(f) := \bQ( f( (B_t)_{t\in [0,T}) )$,
and where $\varrho(\bt,\bx , f)$ is defined by 
(recall~\eqref{def:rho})
\begin{equation}\label{defvarro}
 \varrho( \bt,\bx, f)=  \varrho(\bt,\bx) \bQ\Big[ f\left((B_{t})_{t\in [0,T]}\right) \, \Big| \, \forall i\in \lint 1,k \rint,  \, B_{t_i}=x_i\  \Big].
\end{equation}
With some abuse of notation,
the conditional measure $\bQ\left( \, \cdot \, |  \, \forall i\in \lint 1,k \rint,  \, B_{t_i}=x_i\ \right)$ denotes  the distribution of the process obtained by concatenating independent Brownian bridges linking $(t_{i-1},x_{i-1})$ to $(t_{i},x_{i})$ for $ i\in \lint 1,k \rint$.
 The fact that~\eqref{lafirstdeff} is well-defined for $f\in \cB_b$ is ensured by Proposition~\ref{prop:contiZfini} below; the extension to non-negative $f\in \cB$ is given in Lemma~\ref{eazy}.
Note that  $f\mapsto \varrho( \bt, \bx , f)$ is linear  and thus so is $\cZ^{\go,a}_{T,\beta}(\cdot)$.
 From Lemma \ref{eazy} below, $\cZ^{\go,a}_{T,\beta}(f)\ge 0$ when $f\ge 0$  and $\cZ^{\go,a}_{T,\beta}({\bf 1})=\cZ^{\go,a}_{T,\beta}> 0$. As a consequence, 
for any $a>0$, we can define a probability measure
 $\bQ^{\go,a}_{T,\beta}$  on $C_0([0,T])$ by setting
\begin{equation}\label{QTA}
\bQ^{\go,a}_{T,\beta}\left( A \right):= \frac{\cZ^{\go,a}_{T,\beta}(\ind_{A})}{\cZ^{\go,a}_{T,\beta}} \, ,
\end{equation}
for any Borel set $A$.
We also write $\bQ^{\go,a}_{T,\beta}\left( f \right)$ for the expectation, with respect to $\bQ^{\go,a}_{T,\beta}$,
of a function $f \colon  C_0([0,T])\to \bbR$.
In the same way as for the partition function, we define the measure $\bQ^{\go}_{T,\beta}$ as the limit of $\bQ^{\go,a}_{T,\beta}$ when $a$ goes to zero: this requires $\alpha<\alpha_c$, and the convergence holds for the weak topology.
Let $\cM_T$ denote the space of probability measures on $C_0([0,T])$ equipped with the topology of weak convergence.

\begin{theorem}\label{thm:contmeasure} 
If $\ga\in (1,\ga_c)$, there exists a probability measure 
$\bQ^{\go}_{T,\beta}$ on $C_0([0,T])$ such that 
the following convergence holds almost surely in $\cM_T$
\[
 \lim_{a\to 0} \bQ^{\go,a}_{T,\beta}= \bQ^{\go}_{T,\beta} \,.
\]
In other words, we have almost surely  for every $f\in \cC$
\begin{equation}
 \lim_{a\to 0} \bQ^{\go,a}_{T,\beta}(f)= \bQ^{\go}_{T,\beta}(f)\, .
\end{equation}
\end{theorem}
Since 
$\cZ^{\go,a }_{T, \beta}(\cdot)$ induces a positive measure on $C_0([0,T])$ the above statement turns out to be equivalent to the existence of a positive measure $\cZ^{\go}_{T, \beta}$ such that for every $f\in \cC$
\begin{equation}
  \lim_{a \to 0}\cZ^{\go,a }_{T, \beta}(f)=\cZ^{\go}_{T, \beta}(f) \, .
\end{equation}

\subsubsection{Scaling limit of the discrete model}

In order to justify the fact that  $\bQ^{\go,a}_{T,\beta}$ is the natural model for a continuum polymer based on $\alpha$-stable noise, let mention here the scaling limit result which we prove in \cite{BL20_disc},
namely that the discrete polymer model defined in~\eqref{polymers},
when properly rescaled, converges to the continuum
polymer in L\'evy stable noise.  
We present the convergence with time horizon $T=1$ (which yields no loss of generality by scaling) and set  $\bQ^{\go }_{ \hat \beta}:= \bQ^{\go }_{ \hat \beta,1}$.

Let $S_t^{(N)}$ be the linear interpolation 
of a random walk trajectory, rescaled diffusively:
\begin{equation}\label{defsn}
S^{(N)}_t:= \sqrt{\frac{d}{N}} \Big( (1-u_t) S_{\lfloor Nt \rfloor}  +u_t S_{\lfloor Nt \rfloor+1}    \Big)\,,
 \quad \text{ with } \  u_t= Nt-\lceil N t \rceil.
\end{equation}
We then have the following convergence result. 

\begin{theorema}[cf.~\cite{BL20_disc}]
\label{thm:conv}
Assume that the distribution of the environment $\eta$ satisfies~\eqref{def:eta} for some $\alpha\in (1,\alpha_c)$, with $\alpha_c$ defined in \eqref{def:alphac}.
Setting 
\begin{equation}\label{scalechoice}
 \beta_N:= \hat \beta   2^{\frac{1-\ga}{\alpha}} d^{\frac{d(1-\alpha)}{2}} N^{- \frac{d}{2\ga} ( 1+ \frac{2}{d}-\alpha ) },
\end{equation}
then we have the following convergence in distribution in $\cM_1$,
\begin{equation}
\bP^{\eta}_{N ,\beta_N}\Big( (S^{(N)}_t)_{t\in [0,1]} \in \cdot \Big)  \stackrel{N\to \infty}\Longrightarrow     \bQ^{\go }_{ \hat \beta} \, .
\end{equation}
\end{theorema}

\begin{rem}
The prefactor in $\beta_N$ comes from various factors, including the normalization of the Brownian motion and the periodicity of the random walk.
The above theorem remains valid slightly beyond the assumption \eqref{def:eta}, one can allow for a slowly varying function in the tail distribution provided an appropriate correction in the scaling of $\beta_N$ is made. The analogous result is valid also for $\ga \in (0,1]$.
We refer to \cite{BL20_disc} for details.
\end{rem}

\subsection{The case of a general noise}
\label{sec:generalnoise}

We have focused until now on the case of an $\alpha$-stable noise with $\alpha\in (1,2)$, both because our motivation is to describe the scaling limit for the discrete polymer model with  heavy tailed environment and to make the exposition lighter.
Our result can nonetheless be applied to a much larger variety of noise. 
Let us consider in this section a Poisson process $\go$ on $\bbR\times \bbR^d \times \bbR_+$ with density 
 \begin{equation}
\label{Poissondensgen}
\dd t \otimes \dd x \otimes \gl( \dd \ups),
\end{equation}
where $\gl$ is a measure on $(0,\infty)$  such that $\lambda([a,\infty))<\infty$ for every $a>0$. 
One may keep in mind the case $\lambda(\dd \ups) = \alpha \ups^{-(1+\ga)} \dd \ups$ with $\ga \in (0,2)$, referred to as $\ga$-stable.
We define the truncated noise  $\xi^{(a)}_{\go}$ for $a >0 $ similarly to~\eqref{def:omegabar} (recall that $\cL$ denotes the Lebesgue measure on $\bbR\times \bbR^d$)
\begin{equation}\label{petitnoise}
 \xi^{(a)}_{\go}= \Big(\sum_{(t,x,\ups)\in \go} \ups \ind_{\{\ups \geq a\}} \delta_{(t,x)} \Big)-  \kappa_a  \cL 
\end{equation}
where 
\begin{equation}
\label{def:kappa}
  \kappa_a = \int_{[a,1)} \ups \gl(\dd \ups),
\end{equation}
note that we have in particular $\kappa_a=0$ for $a\geq 1$.
The truncated noise $\xi^{(a)}_{\go}$ converges to a limit  $\xi_{\go}  \in H_{\rm loc}^{-s}(\bbR^{1+d})$
with $s>(1+d)/2$, if and only if  $\int_{(0,1)}  \ups^{2} \gl(\dd \ups)<\infty$ ( the result is definitely standard, but we could not find a reference where it is displayed in this form form, hence we include a proof in the Appendix for completeness, see Proposition~\ref{convergence2}).
Note also that if $\int_{(0,1)} \ups \lambda(\dd \ups) <+\infty$, then  $\kappa_0<\infty$ and
the definition~\eqref{petitnoise}  directly makes sense with $a=0$ so this approximation procedure is not required.
We define, similarly to~\eqref{lafirstdeff}, for any  $f\in \cB_b$
\begin{equation}\label{lafirstdeffpti}
 \cZ^{\go,a}_{T,\beta}(f):=\bQ(f)+\sum_{k=1}^\infty \beta^k\int_{0<t_1<\dots<t_k<T}\int_{(\bbR^d)^k} \varrho( \bt , \bx , f) \prod_{i=1}^k \xi^{(a)}_{\go} ( \dd t_i, \dd x_i) \, .
\end{equation}
The condition that $f$ has a bounded support ensures
that all the integrals are well-defined since the integration is
only over a bounded space-time region;  the summability is shown in Proposition~\ref{prop:contiZfini} below.
Let us stress that Lemma \ref{eazy} below ensures that  $\cZ^{\go,a}_{T,\beta}(f)\ge0$ when $f$ is non-negative.
Given an increasing sequence of positive functions $f_n \in \cB_b$ converging to $1$, one sets 
\begin{equation}
\label{def:Zmonotone}
\cZ^{\go,a}_{T,\beta}:= \lim_{n\to \infty}\cZ^{\go,a}_{T,\beta}(f_n).
\end{equation}
Lemma \ref{eazy} also ensures that the above definition does not depend on the choice of $f_n$.
 Note that the above definition makes it possible to have  $\cZ^{\go,a}_{T,\beta}=\infty$, but this
does not occur provided the following condition is satisfied:
\begin{equation}\label{hypolarge}
 \int_{[1,\infty)} (\log  \ups)^{d/2} \gl(\dd \ups)<\infty.
 \end{equation}

\begin{proposition}
\label{prop:Zfinite}
 Under the assumption \eqref{hypolarge}, 
we have $ \cZ^{\go,a}_{T,\beta} \in (0,\infty)$
for any $a\in (0,1]$.
\end{proposition}
\noindent The condition \eqref{hypolarge} is in fact optimal, if it fails to hold then our partition function is degenerate.
\begin{proposition}\label{prop:toinfinity}
 If the measure $\gl$ does not satisfy~\eqref{hypolarge} then for any $a\in(0,1]$
 we have $\cZ^{\go,a}_{T,\beta}=\infty$ almost surely.
 \end{proposition}

\begin{rem}
The condition~\eqref{hypolarge} arises from an entropy-energy comparison. We let~$X_{R,T}$ be the largest atom in the environment located within a distance $R$ in the time interval $[0,T]$, that is
$$ X_{R,T}:= \max\big\{ \, v \, : \ (t,x,v)\in \go, \ t\in [0,T], \  R/2 \le \|x\|\le R \big \}\, . $$
In the large $R$ limit, the probability of visiting the corresponding atom under the Wiener measure is of order $\exp(-c_T R^2)$. For any fixed $\delta, T>0$, the criterion \eqref{hypolarge} is equivalent (via a Borel-Cantelli argument) to having
$$    \sup_{R>0} ( X_{R,T}  \, e^{-\delta R^2} )<\infty.$$
From this, it is not difficult to deduce that the  contribution of trajectories that visit one large atom located far away remains finite. This provides a heuristic justification of the criterion \eqref{hypolarge}.
\end{rem}

 \noindent Using the partition function~\eqref{lafirstdeffpti}, we can define a probability measure $\bQ^{\go,a}_{T,\beta}$ on $C_0([0,T])$ in the same way as in \eqref{QTA}, \textit{i.e.}\ setting for any Borel set  $A\subset C_0([0,1])$ \,
 \[
 \bQ^{\go,a}_{T,\beta}(A):= \frac{\cZ^{\go,a}_{T,\beta}(\ind_{A})}{ \cZ^{\go,a}_{T,\beta} } \, .
 \] 
Note that $A\mapsto \cZ^{\go,a}_{T,\beta}(\ind_{A})$ defines a locally finite measure on $C_0([0,T])$ even when \eqref{hypolarge} is not satisfied,  cf.\ Proposition~\ref{prop:contiZfini}.
Our main result in this section is that the limit when $a$ goes to $0$ is non-degenerate if $\gl$ satisifes the following assumption:
\begin{equation}\label{hyposmall}
\begin{cases}
 \int_{(0,1)}  \ups^{2} \gl(\dd \ups)<\infty, \quad   &\text{ if } d=1.\\
 \int_{(0,1)}  \ups^{p} \gl(\dd \ups)<\infty \ \text{ for some } p<1+\frac{2}{d}, \quad &\text{ if } d\ge 2.
 \end{cases}
 \end{equation}

\begin{theorem}\label{thm:zalpha}
 Under the assumption \eqref{hyposmall},   for any fixed $f\in \cB_b$  the limit 
 \begin{equation}\label{convf}
  \lim_{a\to 0}\cZ^{\go,a}_{T,\beta}(f)=  \cZ^{\go}_{T,\beta}(f)
 \end{equation}
 exists almost surely  and is finite. We have  $\cZ^{\go}_{T,\beta}(f)>0$ if $f$ is non-negative and $\bQ(f)>0$.
Furthermore, if~\eqref{hypolarge} also holds  then  \eqref{convf} also holds for $f\in \mathcal B$. In particular, in the case $f\equiv 1$ we have,
 \begin{equation}\label{convparty}
  \lim_{a\to 0}\cZ^{\go,a}_{T,\beta}=  \cZ^{\go}_{T,\beta} \  \in (0,\infty).
 \end{equation}
The convergence holds in $\bbL_1$ if and only if $\int_{[1,\infty)} \ups \gl(\dd \ups)<\infty$.

\smallskip

\noindent Additionally,  $\bQ^{\go,a}_{T,\beta}$ converges weakly when $a\to 0$: There exists a probability measure  $\bQ^{\go}_{T,\beta}$ on $C_0([0,T])$ such that, almost surely, for every 
 $f\in \cC$ we have
 \begin{equation}\label{lescontis}
  \frac{ \cZ^{\go}_{T,\beta}(f)}{ \cZ^{\go}_{T,\beta}}\, =\,  \lim_{a\to 0}\bQ^{\go,a}_{T,\beta}(f)= \bQ^{\go}_{T,\beta}(f).
 \end{equation}
 \end{theorem}

\begin{rem}
 Note that the conditions \eqref{hypolarge}-\eqref{hyposmall} are satisfied when $\gl(\dd \ups)= \alpha \ups^{-(1+\alpha)}\dd \ups$ for $\alpha\in (0,\alpha_c)$. When $\alpha\in [\alpha_c,\infty)$, Proposition \ref{prop:tozero}  below establishes that the limit is degenerate.
\end{rem}

\begin{rem}
When  \eqref{hypolarge} is not satisfied, 
it is not difficult to check from our proof that
almost surely, the convergence \eqref{convf} holds simultaneously for all $f\in \cC_b$, that is $\cZ^{\go,a}_{T,\beta}(\,\cdot\, \cap A)$ converges vaguely (as a measure), for any bounded set $A$.
\end{rem}

\begin{rem} 
A statement similar to the weak convergence of $\bQ^{\go,a}_{T,\beta}$ can also be made in the case when \eqref{hypolarge} does not hold.
In that case there exists a locally finite measure~$\bZ^{\go}_{T,\beta}$ on $C_0([0,T])$ (in the sense that it gives finite mass to the set $\{ \varphi  \colon \|\varphi\|_{\infty}<M\}$ for every~$M$) which is such that almost surely, for every $f\in \mathcal C_b$
$$ \lim_{a\to 0}\cZ^{\go,a}_{T,\beta}(f)= \bZ^{\go}_{T,\beta}(f).$$
\end{rem}

\noindent The condition \eqref{hyposmall}, which prevents 
 $\cZ^{\go,a}_{T,\beta}(f)$ from vanishing as $a$ tends to zero, is close to optimal.
 Let us introduce the following alternative and almost equivalent condition
\begin{equation}\label{hyposmall2}
\begin{cases}
 \int_{(0,1)}  \ups^{2} \gl(\dd \ups)<\infty, \quad   &\text{ if } d=1,\\
  \int_{(0,1)}  \ups^{2}|\log (\ups)| \gl(\dd \ups)<\infty, \quad   &\text{ if } d=2,\\
 \int_{(0,1)}  \ups^{1+\frac{2}{d}} \gl(\dd \ups)<\infty, \quad &\text{ if } d\ge 3.
 \end{cases}
 \end{equation}
Then we prove that the limit is degenerate as soon as \eqref{hyposmall2} is violated. 
In particular the following result ensures
that one cannot define the continuum polymer model when 
$\int_{(0,1)}  \ups^{2} \gl(\dd \ups) =\infty$,
in which case 
the noise $\xi_\go$ is itself not well-defined  (see Remark~\ref{rem:noiseL2}).

\begin{proposition}\label{prop:tozero}
 If the measure $\gl$ does \emph{not} satisfies \eqref{hyposmall2} then for any $f\in \cB_b$ we almost surely have
 \begin{equation}
  \lim_{a\to 0} \cZ^{\go,a}_{T,\beta}(f)=0\, 
 \end{equation}
 If \eqref{hypolarge} also holds, we have   $  \lim_{a\to 0} \cZ^{\go,a}_{T,\beta}=0.$ 
\end{proposition}

\noindent
Note that this result proves the last statement of Theorem~\ref{thm:continuous}.

\begin{rem}
From a Borel-Cantelli consideration, we have (for any value of $d$)
\begin{equation}\label{atoumz}
\int_{(0,1)} \ups^{1+\frac{2}{d}} \gl(\dd \ups)<\infty \quad \Longleftrightarrow \quad
 \suptwo{(t,x,v)\in \go}{ t\in [0,T], v\in(0,1)} v\rho(t,x)<\infty \quad  \text{a.s.}
\end{equation}
The quantity $v\rho(t,x)$ corresponds  to the multiplicative weight gained by trajectories which visit the point 
$(t,x)$ multiplied by the entropic cost to visit it. 
On a heuristic level, if $v\rho(t,x)$ is unbounded, it means that there are atoms with arbitrarily small amplitude $z$ which have a large impact on the value $\log \cZ^{\go,a}_{T,\beta}$,   preventing the  convergence of $\cZ^{\go,a}_{T,\beta}$ to a non-zero value.
 In view of the conditions \eqref{hyposmall}-\eqref{hyposmall2} in
 Theorem~\ref{thm:zalpha} and Proposition~\ref{prop:tozero}, this criterion is not far from being sharp when $d\ge 2$.
When $d=1$ the condition~\eqref{atoumz} is less restrictive than the one necessary for the convergence of $\xi^{(a)}_{\go}$ (cf.\ Proposition~\ref{convergence2}). 
\end{rem}

\begin{rem}
The difference between the conditions \eqref{hyposmall} and  \eqref{hyposmall2}  when  $d\ge 2$ leaves a small family of Lévy measures for which the question whether $\cZ^{\go,a}_{T,\beta}$ converges to a positive limit or to zero  remains open. We do not believe that either condition \eqref{hyposmall} or \eqref{hyposmall2} are optimal. Although refinements of the proofs presented here could most likely yield slightly finer condition on both sides, finding the necessary and sufficient condition remains a challenging issue.
Even if the condition is not optimal, the $|\log \ups|$ factor present in \eqref{hyposmall2}
is of importance in dimension $d=2$, since it underlines that in any dimension $d\geq 2$, in contrast with the case $d=1$, there are some  (Poisson) noises for which the continuum polymer (and the noisy stochastic heat equation see Section \ref{sec:SHE} below) are not defined. 
\end{rem}

\begin{rem}
After a first draft of this manuscript was published, in a collaborative effort with C.\ Chong \cite{BCL21}, the condition \eqref{hyposmall} and 
\eqref{hyposmall2} have been improved. More precisely, it was shown that
Theorem \ref{thm:zalpha} holds under the condition
$\int_{(0,1)}  \ups^{1+\frac{2}{d}}|\log (\ups)| \gl(\dd \ups)<\infty$
when $d\ge 2$, identifying thus the necessary and sufficient condition for convergence when $d=2$. When $d\ge 3$ the condition \eqref{hyposmall2} for Proposition \ref{prop:tozero} has been improved to  
$$\int_{(0,e^{-e})}  \ups^{1+\frac{2}{d}}|\log (\ups)| (\log |\log (\ups)|)^{-c}\gl(\dd \ups)<\infty$$
with $c$ any constant greater than $5+\frac4d$,
leaving only a very narrow gap between the necessary and the sufficient condition.
\end{rem}

\subsection{Main properties of the Continuum Directed Polymer in L\'evy Noise}
\label{sec:mainprop}

Let us assume throughout the rest of this section that Assumptions \eqref{hypolarge}-\eqref{hyposmall} are satisfied.
We describe under these assumptions a few key properties of our polymer measure. First, we underline how $\bQ^{\go}_{T,\beta}$ is in some aspects very similar to the Wiener measure and is in others very singular with respect to it. 
Then, we provide an explicit expression for the finite-dimensional marginal density of the measure, via point-to-point partition functions.

\subsubsection{Basic properties of the continuum polymer in L\'evy noise}
\label{bazics}

Let us define $\bbP\rtimes\bQ^{\go}_{T,\beta}$ the averaged polymer measure as follows
\begin{equation}
\bbP\rtimes\bQ^{\go}_{T,\beta}(A):= \bbE[\bQ^{\go}_{T,\beta}(A) ]\, .
\end{equation}

\begin{proposition} \label{abzolute}
 The averaged polymer measure
$\bbP\rtimes \bQ^{\go}_{T,\beta}$ is  absolutely continuous with respect to $\bQ$.
\end{proposition}

\noindent Proposition~\ref{abzolute} yields an important information concerning $\bQ^{\go}_{T,\beta}$ since it implies that  $\bbP$-almost surely $ \bQ^{\go}_{T,\beta}$ inherits  any given $\bQ$-almost sure property of the Brownian motion.

\begin{cor}
 If  $A$ is a Borel set of $C_0([0,T])$ such that $\bQ(A)=1$, then $\bbP$-a.s.\ we have $ \bQ^{\go}_{T,\beta}(A)=1$.
As an example, 
for almost every~$\go$, a trajectory $(B_t)_{t\in [0,T]}$
has $\bQ_{T,\gb}^{\go}$-a.s.\ a modulus of continuity 
given by $\sqrt{2h \log(1/h)}$: in other words,
\[
\bQ_{T,\gb}^{\go} \bigg( \Big\{ \varphi \in C_0([0,T]) \,; \,   \limsup_{h\downarrow 0} \sup_{ 0\leq t \leq T-h}  \frac{|\varphi(t+h) - \varphi(t)|}{ \sqrt{2h \log(1/h)}} =1 \Big\} \bigg) =1 \, .
\]
This implies in particular that for any $\gamma<1/2$, 
polymer trajectories are $\bQ_{T,\gb}^{\go}$-a.s.\ everywhere locally $\gamma$-H\"older continuous.
\end{cor}

On the other hand we have to mention that $\bQ^{\go}_{T,\beta}$ is very singular with respect to $\bQ$, most strikingly when  $\int_{(0,1)}\ups\gl(\dd \ups)=\infty$.
To illustrate this fact, given $\varphi \in  C_0([0,T])$ let us consider $\varDelta(\varphi,\go)$ the set of times at which the graph of $\varphi$ visits points of $\go$:
\begin{equation}
\label{def:Deltaphi}
\varDelta(\varphi,\go):= \{ t\in[0,T] \ : \exists \ups>0, \  (t , \varphi(t),\ups) \in \go\}\,.
\end{equation}
Let us set
  \begin{equation}\begin{split}
   \cA_{\mathrm{dense}}(\go)&:= \big\{ \varphi\in C_0([0,T]) \ : \   \varDelta(\varphi,\go) \text{ is dense in } [0,T] \big\}, \\
      \cA_{\mathrm{empty}}(\go)&:= \big\{ \varphi\in C_0([0,T]) \ : \   \varDelta(\varphi,\go)=\emptyset \big\},\\
          \cA_{\infty}(\go)&:= \big\{ \varphi\in C_0([0,T]) \ : \   \#\varDelta(\varphi,\go)=\infty \big\},
   \end{split}
  \end{equation}
\begin{proposition}\label{emptydance}
Under Assumptions \eqref{hypolarge}-\eqref{hyposmall}  the following statements hold.
\begin{itemize}
 \item [(i)] We have almost surely $\bQ(  \cA_{\mathrm{empty}})=1$.
 \item [(ii)]
 If $\int_{(0,1)}\ups\gl(\dd \ups)<\infty$ then $\bQ^{\go}_{T,\beta} (\cA_{\mathrm{empty}})\in(0,1)$ and $\bQ^{\go}_{T,\beta}(\cA_{\infty})= 0$ a.s.
 \item [(iii)] If $\int_{(0,1)}\ups\gl(\dd \ups)=\infty$ we have
$\bQ^{\go}_{T,\beta}(\cA_{\mathrm {dense}})= 1$ a.s.
\end{itemize}
\end{proposition}

\begin{rem}\label{tvtrem}
 Note that when $\int_{(0,1)}\ups\gl(\dd \ups)<\infty$, we have in fact  $\bQ^{\go}_{T,\beta} (\cdot \ | \ \cA_{\mathrm{empty}})=\bQ$, so $\bQ^{\go}_{T,\beta}$ is in that case not singular with respect to the Wiener measure.
 The technique used for the proof of Proposition \ref{emptydance} can possibly be pushed a bit further to yield the following statement:
 \begin{itemize}
  \item When $\int_{(0,1)}\ups\gl(\dd \ups)<\infty$ then the convergence of $\bQ^{\go,a}_{T,\beta}$ towards $\bQ^{\go}_{T,\beta}$ holds also for the  total variation distance.
  \item When $\int_{(0,1)}\ups\gl(\dd \ups)=\infty$ then $\| \bQ^{\go,a}_{T,\beta}-\bQ^{\go}_{T,\beta}\|_{TV}=1$ for every $a>0$. 
 \end{itemize}

\end{rem}

\subsubsection{Point-to-point partition functions and finite dimensional marginals}\label{ptpmarginals}
The aim of this section is to give an explicit description of the finite-dimensional marginals of $\bQ^{\go}_{T,\beta}$.
If we fix $0<t_1<\dots<t_k\le T$, then the distribution of $(B_{t_1},\dots, B_{t_k})$ under $\bQ_{\gb,T}^{\go}$ is absolutely continuous with respect to the Lebesgue measure and its density can be expressed using the so-called point-to-point partition functions.
For any $a>0$,
define for all $t>0$ and $x\in \bbR^d$
the partition function from $(0,0)$ to $(t,x)$ as (recall the definition \eqref{def:rho})
\begin{equation}
\label{def:pointtopoint}
\cZ^{\go,a}_{\beta}(t,x) : 
=   \rho_t(x) +  \sum_{k=1}^\infty \beta^k\int_{0<t_1<\dots<t_k<t}\int_{(\bbR^d)^k}  \varrho(\bt,\bx) \rho_{t-t_k}(x-x_k)\prod_{i=1}^k  \xi_{\go}^{(a)} ( \dd t_i, \dd x_i) \,  ,
\end{equation}
if the integral is convergent (and set $\cZ^{\go,a}_{\beta}(t,x)=\infty$ if not).
The following proposition shows that the 
point-to-point partition function of our continuum model --- defined as the limit of $\cZ^{\go,a}_{\beta}(t,x)$ when $a$ tends to zero --- is well-defined, positive and finite.
\begin{proposition}\label{th:superlem}
Suppose that \eqref{hypolarge} holds, then given $a\in (0,1]$, $t>0$ and $x\in \bbR^d$ we have almost surely
\begin{equation}\label{lafinite}
 \cZ^{\go,a}_{\beta}(t,x)\in (0,\infty).
\end{equation}
If \eqref{hyposmall} also holds then given $t>0$ and $x\in \bbR^d$, we have almost surely
 \begin{equation}\label{lapazero}
  \cZ^{\go}_{\beta}(t,x)=\lim_{a\to 0}\cZ^{\go,a}_{\beta}(t,x) \,, \quad \text{ with }  \cZ^{\go}_{\beta}(t,x) \in (0,\infty) \, .
 \end{equation}
 If $\int_{[1,\infty)} \ups \gl(\dd \ups)<\infty$ then the convergence holds in $\bbL_1$. 
\end{proposition}
\noindent For $(s,y)\in \bbR\times \bbR^d$, let us define the shifted environment 
\[
\theta_{(s,y)}\go := \big\{  (t-s,x-y,\ups) \ : \ (t,x,\ups)\in \go \big\}.
\]
Then, we can define for $(t_1,x_1),(t_2,x_2) \in \bbR\times \bbR^d$, $t_1<t_2$ the partition function linking two arbitrary points:
\begin{equation}
 \cZ^{\go,a}_{\beta}[(t_1,x_1),(t_2,x_2)]:=  \cZ^{\theta_{(t_1,x_1)}\go,a}_{\beta}(t_2-t_1, x_2-x_1) \, .
\end{equation}
Finally we set 
\begin{equation}
\label{p2p}
  \cZ^{\go}_{\beta}[(t_1,x_1),(t_2,x_2)]:= \limsup_{a\to 0} \cZ^{\go,a}_{\beta}[(t_1,x_1),(t_2,x_2)],
\end{equation}
and we omit the first coordinate in the notation when it is equal to~$0$.
 We use a $\limsup$ instead of a limit in the definition only to make sure that  $\cZ^{\go}_{\beta}[(t_1,x_1),(t_2,x_2)]$ is defined simultaneously for all $(t_1,x_1)$ and $(t_2,x_2)$.
Note that Proposition~\ref{th:superlem}, together with translation invariance, shows that
for any fixed $(t_1,x_1)$ $(t_2,x_2)$
the $\limsup$ in~\eqref{p2p} can  almost surely be replaced by a limit (so the point-to-point partition function
$\cZ_{\gb}^{\go}[(t_1,x_1),(t_2,x_2)]$
is almost surely well-defined, positive and finite).

\begin{proposition}
\label{twomarg}
For any $0<t_1<\dots<t_k=T$,
 the set 
\[
\left\{ (x_1, \ldots, x_k) \ : \  \forall i\in \lint 1,k\rint,\  \cZ^{\go}_{\beta}[(t_{i-1},x_{i-1}),(t_i,x_{i})] =  \lim_{a\to 0} \cZ^{\go,a}_{\beta}[(t_{i-1},x_{i-1}),(t_i,x_{i})] \  \right\}
\]
has almost surely full Lebesgue measure.
Furthermore, the convergence  
 \begin{equation}
 \lim_{a\to 0}\prod_{i=1}^k \cZ^{\go,a}_{\beta}[(t_{i-1},x_{i-1}),(t_i,x_{i})]=\prod_{i=1}^k \cZ^{\go}_{\beta}[(t_{i-1},x_{i-1}),(t_i,x_{i})],
\end{equation}
 holds almost surely in $L_1( (\bbR^{d})^{k})$.
 Additionally, for almost every $\go$, the  $k$-marginals measure $\bQ^{\go}_{T,\beta}( (B_{t_1},\dots, B_{t_k})\in \cdot)$ is absolutely continuous with respect to the Lebesgue measure
and we have for any bounded measurable $g$ on $(\bbR^{d})^{k}$ 
\begin{equation}
 \bQ^{\go}_{T,\beta} \big( g(B_{t_1},\dots B_{t_k}) \big)= \frac{1 }{\mathcal Z^{\go}_{\beta,T}} \int_{(\bbR^d)^k} g(\bx )\prod_{i=1}^k \cZ^{\go}_{\beta}[(t_{i-1},x_{i-1}),(t_i,x_{i})] \dd \bx \, .
\end{equation}
\end{proposition}

\begin{rem}
For $k=1$ the above proposition states that the density of the distribution of $B_T$ under $\bQ_{T,\gb}^{\go}$ is given by $\cZ^{\go,a}_{\beta}(T,\cdot)/\cZ^{\go,a}_{\beta,T}\,$.
\end{rem}

%

\begin{rem}
Let us stress that in the above proposition,
we fix $0<t_1<\cdots <t_k=T$
before considering a realization of $\go$.
This is an important point since there are exceptional times for which $\bQ^{\go}_{T,\beta}(B_{t}\in \cdot )$
admits no density. In fact is not difficult to check that if $(t,x,\ups)\in \go$ then 
$\bQ^{\go}_{T,\beta}(B_{t}=x)>0$.
%
\end{rem}

\subsection{Connection with the Stochastic Heat Equation with multiplicative L\'evy noise}
\label{sec:SHE}

 In \cite{AKQ14b} the continuum directed polymer model with white noise is constructed directly from the solution of the Stochastic Heat Equation (SHE) with multiplicative Gaussian white noise. It is not possible to proceed in this way with a general L\'evy noise (simply because the solution is not regular enough) and our approach here is quite different.
The continuum model constructed in Theorem \ref{thm:zalpha} bears nonetheless a strong connection 
with the SHE with multiplicative L\'evy noise. We discuss here this connection in some more detail and compare our results with the existing ones concerning the SHE with L\'evy noise.
Our formal definition $\cZ^{\go}_{T,\beta} = \bQ\left[ \mathbf{:}e^{\beta H_\go(B)}\mathbf{:}   \right]$ (see \eqref{giancarlo}) corresponds to a Feynman--Kac formula associated with the following equation
\begin{equation}\label{levySHE}
\partial_t u:= \frac{1}{2d} \gD u + \beta \xi_{\go} \cdot u \, .
\end{equation}
More precisely, the point-to-point partition function  $\cZ^{\go}_{\beta}(t,x)$ defined in \eqref{def:pointtopoint}
formally corresponds to the solution of \eqref{levySHE} with
$\delta_0$ initial condition.
Starting from an arbitrary initial condition $u_0$ (a  locally finite signed measure), the solution of \eqref{levySHE} should take 
the form
 \begin{equation}\label{groove}
  u(t,x):= \int_{\bbR^d}  \cZ^{\go}_{\beta}[(0,y),(t,x)] \, u_0(\dd y) \, .
 \end{equation}
In the case $u_0(\dd y)  = g_0(y) \dd y$ for some bounded and measurable function $g_0$, 
the fact that~\eqref{groove}
is well-defined
derives from Proposition~\ref{twomarg} (in the case $k=1$), combined
with a time-reversal argument giving
\[
\big( \cZ^{\go}_{\beta}[(0,y),(t,x)] \big)_{y \in \bbR^d}\stackrel{(d)}{=} \big( \cZ^{\go}_{\beta}[(0,x),(t,y)]\big)_{y \in \bbR^d} \, ,
\]
that ensures that  $\cZ^{\go}_{\beta}[(0,\cdot),(t,x)]\in L_1(\bbR^d)$ 
almost surely.
For the general case where $u_0$ is a measure,
we refer to Proposition~\ref{prop:SHE} below 
for the well-posedness of~\eqref{groove}.

 The equation \eqref{levySHE} has been extensively studied (often under a more general form, see e.g.\ \cite{Mueller98, Loubert98}). To our knowledge the most complete results concerning the existence of  solutions have been given in~\cite{Chong17}.
More precisely, in \cite{Chong17} the existence of solutions in the integral form
\begin{equation}\label{integralz}
 u(t,x)= \int_{\bbR^d}  \rho_{t}(x-y)  u_0(\dd y) +\beta \int^t_0\int_{\bbR^d} \rho_{t-s}(x-y) u(s,y)
 \xi_{\go}(\dd s, \dd y) \,,
\end{equation}
 called \textit{mild solutions},
are obtained under the condition 
\begin{equation}\label{hypochong}
  \int_{(0,1)} \ups^p \gl(\dd\ups)<\infty  \quad \text{ and } \quad  \int_{[1,\infty)} \ups^q \gl(\dd\ups)<\infty,
\end{equation}
with  $p\in (0,1+\frac2d)$ and $ (2+\frac{2}{d}-p)^{-1} \le q\le p$ and  for $u_0=g_0(y) \dd y$ with $g_0$ bounded and measurable.
Uniqueness has been established earlier \cite{Loubert98} under the more stringent assumption 
$\int_{(0,\infty)} \ups^p \gl(\dd\ups)<\infty$ for  some $p\in [1,1+\frac2d)$, which for instance excludes $\ga$-stable noises.
Let us stress that the above is a very partial account of the results in  \cite{Chong17} since the existence results deal with  a more general class of equations and allows for a wider variety of noise (it allows for complex jumps and when $d=1$ for a Gaussian white noise part as well as for space-time inhomogeneities).

\smallskip

While our assumptions \eqref{hypolarge}-\eqref{hyposmall} are less restrictive than \eqref{hypochong},
we cannot prove that \eqref{groove} solves the equation \eqref{integralz} under these assumptions. However, we can show that the solution of the equation with a truncated noise converges almost surely when the truncation levels goes to zero and infinity respectively. Additionally, we keep quite a large freedom concerning the choice of the initial condition.
Let us write this result in full detail for completeness.
We are going to make the following assumption
on the initial condition~$u_0$:
\begin{equation}
\label{hypou0}
\limsup_{r\to \infty} r^{-2} \log \big( |u_0|( [-r,r]^d ) \big) 
 < \frac{1}{2T}   
\end{equation}
where 
$|u_0|$ is the total variation of the measure $u_0$.
This condition is present to ensure that~\eqref{groove} is well defined 
and almost surely finite on the interval $[0,T]$.
For $b>a$, let us introduce $\xi^{[a,b)}_\go$ the noise truncated at levels $a$ and $b$  (recall the notation~\eqref{def:kappa} for $\kappa_a$)
\begin{equation}\label{splitnoise}
\xi^{[a,b)}_\go:= \sum_{(t,x,\ups)\in \go} \ups \ind_{\{\ups \in[a,b)\}}\delta_{(t,x)}+(\kappa_b- \kappa_a) \cL \, .
\end{equation}  
Then, setting by convention the quantity to be $\infty$ when the integral is not well-defined, we set
\begin{equation}\label{defuab}
u^{[a,b)}(t,x) :=  \int_{\bbR^d} 
\cZ^{\go,[a,b)}_{\beta}[(0,y),(t,x)] \, u_0(\dd y )  \, ,
\end{equation}
where $\cZ^{\go,[a,b)}_{\beta}[(0,y),(t,x)]$ is defined as in \eqref{def:pointtopoint} with $\xi^{(a)}_{\go}$ replaced by $\xi^{[a,b)}_{\go}$.
By Theorem~1.2.1 in \cite{Loubert98}, if $u_0$ is absolutely continuous with bounded density  w.r.t.\ to the  Lebesgue measure, then $u^{[a,b)}$ is the unique solution (in some reasonable functional space) of \eqref{integralz} (with noise $\xi_{\go}^{[a,b)}$).
We first observe that $u^{[a,b)}$ converges when $b$ tends to infinity under very mild assumptions.

\begin{proposition}\label{lem:SHE}
Assume that~\eqref{hypolarge} holds,
and that $u_0$ satisfies~\eqref{hypou0}.
Then for any given  $t \in [0,T]$ and $x\in \bbR^d$,
for any $a >0$
\begin{equation}\label{lemilem}
u^a(t,x):=  \int_{\bbR^d} 
\cZ^{\go,a}_{\beta}[(0,y),(t,x)] \, u_0(\dd y )
\end{equation}
is almost surely finite.
\end{proposition}

\begin{rem}
When $\int_{[1,\infty)} \ups\gl(\dd \ups)<\infty$, and for bounded initial conditions,  applying Theorem~1.2.1 in \cite{Loubert98} we get that $u^a(t,x)$ is  the unique solution of  \eqref{integralz}.
For noise with heavier tails, $u^a(t,x)$ should also be a solution of \eqref{integralz} and coincide with the solution considered in \cite{Chong17} whenever it is well-defined. Since this is not the main focus of the paper we do not include a proof of this statement, which in any case would only provide a minor extension on the class of noises considered \cite{Chong17} which  includes $\int_{[1,\infty)} \ups^p\gl(\dd \ups)<\infty$ for all $p>0$. We do not have an argument establishing uniqueness in that case.
\end{rem}

Let us now present the result. It establishes the convergence of~$u^a$ when $a$ tends to~$0$. While the limit is the natural candidate to be a solution to \eqref{integralz} under less restrictive assumptions than those considered in \cite{Chong17}, we could not verify that~$u$ solves the equation.

\begin{proposition}
\label{prop:SHE}
Assume  that \eqref{hypolarge}-\eqref{hyposmall} are satisfied. 
Given $u_0$ a locally finite signed measure on $\bbR^d$
satisfying~\eqref{hypou0},
then  for $(t,x)\in [0,T]\times \bbR^d$ the integral defining 
$u(t,x)$ in~\eqref{groove} is almost surely finite (and well-defined)
and  for every fixed $t\in[0,T]$ we have
\[
\lim_{a\to 0} u^{a}(t,x)=u(t,x) \, ,
\]
except on a set  (of $\bbR^d$) of Lebesgue measure zero.
\end{proposition}

\begin{rem}
When \eqref{hypochong} is satisfied and the initial condition has a bounded density w.r.t\ to the Lebesgue measure, it follows from results of \cite{Chong17} and \cite{Loubert98} that $u$ is the solution of \eqref{integralz} constructed in \cite[Theorem 3.1]{Chong17}. 
\end{rem}

\subsection{Further discussion on the results}
\label{sec:comments}

Let us now comment further on our results,
and explain how they compare with the literature,
how they can be extended and what interesting open questions remain to be solved.

\subsubsection{Scaling properties in the case of $\ga$-stable L\'evy noise}
Let us come back further on the case
of the $\alpha$-stable noise,
that is when $\lambda(\dd \ups) = \ga  \ups^{-(1+\ga)} \dd \ups$, with $\ga\in (0,2)$.
We have already seen that in that case Assumptions~\eqref{hypolarge}-\eqref{hyposmall} are satisfied provided that $\ga\in (0,\ga_c)$, so Theorem~\ref{thm:zalpha} and more importantly~\eqref{lescontis} holds, so that $\bQ_{T,\gb}^{\go}$ is well-defined.
Now, notice that in the $\ga$-stable case the Poisson point process $\go$ has the following scaling property
$\go \stackrel{(d)}{=} \{ (r t, s x, (rs^d)^{1/\ga} \ups) ;\, (t,x,\ups)\in \go  \}$ for any $r,s>0$.
Using additionally the Brownian scaling,
one can then check that
the continuum polymer in $\alpha$-stable L\'evy environment
satisfies the following
scaling property: if $\ga\in (0,\ga_c)$, for all $r>0$ 
\begin{equation}\label{scaling}
\bQ^{\go }_{T, \beta} (A)  \stackrel{(d)}{=}  \bQ_{r T, r^{-\zeta}\gb }^{\go} ( A_r ) \, , \quad \text{ with }\ \zeta=\tfrac{d}{2\ga} ( 1 +\tfrac{2}{d} - \ga) \, .
\end{equation}
where $A_r := \{ \varphi_r\colon t\mapsto \frac{1}{\sqrt{r}} \varphi(\frac{t}{r}) , \varphi\in A \}$.

\subsubsection{SHE with L\'evy noise: advantages and disadvantages of our method}
Let us now compare our Proposition~\ref{prop:SHE}
with the results of Chong~\cite{Chong17}.
First of all, as we already stressed in Section~\ref{sec:SHE},
our Proposition~\ref{prop:SHE} gives a weaker
notion of solution to the SHE~\eqref{levySHE}
than existence of solutions in the integral form~\eqref{integralz}, as proven in~\cite{Chong17}.
Additionally, 
Chong's results allows to deal with a larger class of integral equations 
\begin{equation}
\label{SHEchong}
Y(t,x)  =  \int_{\bbR^d} G(0,y; t,x)Y_0(y)\dd y  + \int_0^t \int_{\bbR^d}  G(s,y ; t,x ) \sigma( Y(s,y) )  M(\dd s,\dd y) \, ,
\end{equation}
where: (i) $M$ is a noise that can include a white noise part when $d=1$ and a (signed) pure jump component;
(ii) $\sigma$ is a globally Lipschitz function;
(iii) $G(t,x; s,y)$ is measurable and dominated by a constant times the heat kernel $\rho_{t-s}(x-y)$.

We have presented our results in the case where
$M=\xi_{\go}$ (\textit{i.e.}~has no white noise component and only positive jumps), $\sigma(Y)=Y$ and $G(t,x ;s,y) =\rho_{t-s}(x-y)$.
Let us now present the advantages of our method, and in which directions it can be generalized.

\smallskip
\textit{a)} First and foremost, our method enables us to make sense of Feynman--Kac formulas containing a functional $f$ of the Brownian Motion, that is $\cZ^\go_{\beta}(f)$ (see Theorem~\ref{thm:zalpha}-\eqref{convf}). This is something absolutely required to be able to define the continuum model.

\smallskip
\textit{b)}  Our  tail assumptions~\eqref{hypolarge}-\eqref{hyposmall} on the L\'evy  measure are less restrictive
than those~\eqref{hypochong} which are used in~\cite{Chong17}.  In particular our method allows to treat the integrability issues at $0$ and $\infty$ separately.
Note also that in view of Propositions~\ref{prop:toinfinity} and~\ref{prop:tozero}, our assumptions~\eqref{hypolarge}-\eqref{hyposmall} are close to being optimal.

\smallskip
\textit{c)} We are able to deal with more singular initial conditions
than in~\cite{Chong17}. For the application we have in mind, it is of the utmost importance to be able to deal with Dirac initial condition, which corresponds to the point-to-point partition function $\cZ_{\gb}^{\go}(t,x)$ and appears to be excluded in~\cite{Chong17}.

\smallskip
\textit{d)}  We can easily adapt our proof to the case 
of an arbitrary  kernel $\rho_t$ (in particular, not only the ones dominated by Gaussians), but this would require to adapt
the conditions~\eqref{hyposmall}-\eqref{hypolarge}
In particular, we could replace the Laplacian $\gD$ with more general operators. For instance, in dimension $1$, we could replace the Brownian Motion by a L\'evy process,
see the paragraph below for further discussion.

\smallskip
\textit{e)}   While our method does not seem to allow to treat the case of general Lipshitz $\sigma$ (for which we lose the existence of a Feynman-Kac representation of the solution), let us mention, that it should extend without much problem to the case where $\sigma(u)=au+b$ with $a,b>0$ (that is, considering a mixture of additive and multiplicative noise).

\smallskip
\textit{f)} 
To conclude, we stress that maybe
the most problematic part
would be to extend our results to a more general noise.
In particular, our method does not allow to deal with
general complex (or signed) noise:
the issue essentially arises in the proof of Proposition~\ref{thm:continuous2}, which shows that $(\cZ_{\gb}^{\go,a})_{a\in (0,1]}$ is uniformly integrable (if $\int_{[1,\infty)} \ups \lambda(\dd \ups)<\infty$); all the other points extend quite easily.
In view of our techniques (in particular Sections~\ref{sec:dim1}-\ref{sec:contd>2}), this appears to be manageable in dimension $d=1$, but it is possibly more problematic in dimension $d\ge 2$ (the truncation we use is based on a multi-body functional that needs to be adapted in the case of a complex or signed noise).
Similarly, in accordance with the literature on directed polymer models,
adding a white-noise component should be feasible
in dimension $d=1$, but it is likely that in dimension $d\geq 2$ it would make the limit degenerate
(in analogy with the SHE with multiplicative white noise in dimension $d=2$, see~\cite{CSZ18scaling, gu2019}).

\subsubsection{Applications of our method to other disordered systems}
\label{sec:extensions}

Our method appears robust enough to be adapted to 
the setting of other models with heavy-tail disorder.
In particular, in analogy with \cite{CSZ13}, one should be able to consider several (discrete) models, 
and construct their continuum counterpart with L\'evy noise.
This includes for instance:
\begin{itemize}
\item[(A)] the $(1+d)$-dimensional long-range directed polymer, see~\cite{Comets07,Wei16} for the case of dimension $d=1$, where the underlying random walk $(S_n)_{n\ge 0}$ is in the $\gamma$-stable domain of attraction, with $\gamma \in (0,2)$;
\item[(B)] the disordered pinning model, see~\cite{GB10} for an overview  (it has been studied in~\cite{LS17} in the case of a heavy-tail noise).
\end{itemize}
We could also
consider other disordered models, such as the copolymer model (see \cite[Ch.~6]{GB07} for an overview and \cite{BdH97,CG10} for the question of the scaling limit) or the random field Ising model (see~\cite[Ch.~7]{Bov06} for an overview and \cite{CSZ13} for the question of the scaling limit). We however chose to focus on the two examples (A)-(B) above, which might provide a sufficient illustration on how general our construction is.
In both cases (A) and (B) we only briefly present the models and discuss how the assumptions~\eqref{hypolarge}-\eqref{hyposmall} have to be adapted to ensure the convergence of the partition function.  In order to be fully understood, the discussion below requires to be familiarized with the proof of our main result. It can be thus be skipped during the first reading.

\medskip
\noindent
(A) \textit{The continuum $(1+d)$-dimensional long-range directed polymer in L\'evy noise}.
The idea is to replace in the definitions the
Brownian motion $(B_t)_{t\geq 0}$
by a $d$-dimensional $\gamma$-stable process $(X_t)_{t\geq 0}$ with $\gamma\in (0,2)$, that we suppose centered and isotropic for simplicity. More precisely, we can define, analogously to~\eqref{lafirstdef}, the partition function
\begin{equation}
\label{longrange}
\cZ_{\gb,{\rm long}}^{\go,a} := 1+\sum_{k=1}^\infty \beta^k\int_{0<t_1<\dots<t_k<T}\int_{(\bbR^d)^k}  \prod_{i=1}^k \rho^{(\gamma)}_{t_i-t_{i-1}} (x_i-x_{i-1}) \prod_{i=1}^k \xi_{\go}^{(a)} ( \dd t_i, \dd x_i)\, \, ,
\end{equation}
where $\rho_{t}^{(\gamma)}(x)$ is the transition kernel 
of our $\gamma$-stable process, and is defined by
\[
\rho_t^{(\gamma)}(x) := \frac{1}{(2\pi)^d} \int_{\R^d} e^{-t \|z\|^{\gamma}} \cos( x \cdot z) \dd z .
\] 
While $\rho_t^{(\gamma)}$ does not admit a closed expression, its asymptotic properties are  well known  (dating back to~\cite{K70}, see also~\cite[Ch.~2]{ST94}).
It is a bounded radial function and has 
the following asymptotic behavior 
\begin{equation}\label{thetail}
\rho_1^{(\gamma)} (x)  \sim c_{d,\gamma}  \|x\|^{-(d+\gamma)}\,,
\qquad \text{ as }
\| x\|\to \infty\,.
\end{equation}
The scaling relation 
$\rho_{t}^{(\gamma)}(x) = t^{-d/\gamma} \rho_1( t^{-1/\gamma} x)$
also implies that $\|\rho_{t}^{(\gamma)}  \|_{\infty}=  c'_{d,\gamma} t^{-d/\gamma}.$

\smallskip

Now let us discuss under which condition on the L\'evy measure $\gl$ the partition function in~\eqref{longrange} remains finite.  Note first that if $\int_{[1,\infty)} \ups \gl(\ups)<\infty$, then we have $\bbE [ \cZ_{\gb,{\rm long}}^{\go,a} ]<\infty$ from Lemma \ref{prop:contiZfini} (or rather its straightfoward adaptation to this case) and a discussion is necessary only for the integrability of heavier-tailed noises.
In analogy with~\eqref{hypolarge}, we want to make sure that 
the weight of Poisson points with large intensity is compensated by the cost of making a long jump to visit them, which by \eqref{longrange} is of order  
 $\|x\|^{-(d+\gamma)}$. Hence we need a condition that ensures that
 \begin{equation}\label{lesupp}
  \sup \big\{ \ups (1+ \|x\|)^{-(d+\gamma)}  \,\colon  (t,x,\ups) \in \go, \ t\in[0,T]      \big\} <\infty.
 \end{equation}
 We should require in fact a bit more than \eqref{lesupp} but not much more (we opt not to stretch the discussion any further) and we believe
 a condition that ensures that $\cZ_{\gb,{\rm long}}^{\go,a}<\infty$ and thus replaces \eqref{hypolarge} in this case is
\begin{equation}
\label{hypolargelongrange}
\int_{[1,\infty)} \ups^{q} \lambda(\dd \ups) <\infty  
\  \text{ for some } q> \tfrac{d}{d+\gamma}. 
\end{equation}

On the other hand, the condition~\eqref{hyposmall}
prevents the possible accumulation of small weights that would make the limiting partition function degenerate.
It is intimately related to the local limit behavior of $\rho_t^{(\gamma)}(x)$ at small times, more precisely to
$\int_{\bbR^d} (\rho_t^{(\gamma)}(x))^2 \dd x$ which by scaling is equal to  $t^{-d/\gamma} \int_{\bbR^d} (\rho_1^{(\gamma)}(x))^2 \dd x.$
In analogy with~\eqref{hyposmall}, a (near-optimal)  condition that ensures that $ \lim_{a\to 0}\cZ_{\gb, {\rm long}}^{\go,a}$ is non-degenerate should therefore be 
\begin{equation}
\label{hyposmalllongrange}
\int_{(0,1)} \ups^{p} \lambda(\dd \ups) <\infty \, ,\quad \text{ for some } p< \min \big( 1+\tfrac{\gamma}{d}, 2 \big)\, .
\end{equation}

We therefore conjecture that if~\eqref{hypolargelongrange}-\eqref{hyposmalllongrange} hold,
then the partition function~$\cZ_{\gb,{\rm long}}^{\go,a}$ defined in~\eqref{longrange} converges a.s.\ to a non-degenerate limit and that one can construct a continuum measure corresponding to  the $(1+d)$-dimensional long-range directed polymer in L\'evy noise.
Let us stress that in the case of an $\ga$-stable noise (\textit{i.e.} $\lambda(\dd \ups) = \ga \ups^{-(1+\ga)} \dd \ups$), 
the conditions~\eqref{hypolargelongrange}-\eqref{hyposmalllongrange} translate into the condition $\frac1\gamma < \ga < \min \big( 1+\tfrac{d}{\gamma}, 2 \big)$.
Additionally, in analogy with Theorem~\ref{thm:conv}, the continuum long-range directed polymer model in $\ga$-stable noise should appear
as the scaling limit of the long-range
directed polymer model, defined as in~\eqref{def:Zn}
with a random walk $(S_n)_{n\geq 0}$
in the domain of attraction of a $\gamma$-stable law
and heavy tailed  disorder satisfying~\eqref{def:eta}.

\medskip
\noindent
(B) \textit{The continuum disordered pinning model in L\'evy noise}.
The disordered pinning model describes a renewal process $\tau = \{\tau_0=0,\tau_1,\tau_2,\ldots\}$ on $\bbN$ (representing contact points) interacting with an inhomogeneous defect line.
In the case of a heavy tailed environment
$(\eta_x)_{x\in \bbN}$, 
it is convenient to write the partition function of the model as follows,
see~\cite{LS17}:
\begin{equation}
\label{def:pinning}
Z_{N,\gb,h}^{\eta} := \bE\Big[ \prod_{n=1}^N e^{h \ind_{\{n\in\tau \}}} \big( 1+ \gb \eta_{n} \ind_{\{n\in \tau\}}  \big) \Big] \, ,
\end{equation}
where $h$ in an additional (homogeneous) pinning parameter.
A standard (and natural) assumption in the literature is that 
$\bP(\tau_1= n)  = (1+o(1)) c n^{-(1+\gamma)}$ as $n$ goes to infinity, for some $\gamma>0$.  Under this assumption, if $\gamma\in (0,1)$, then the set of contact
points $\tau \cap [0,N]$, properly scaled,
converges to what is called the regenerative set of index~$\gamma$.
This leads us to make the following definition
for the truncated partition function of the continuum 
disordered pinning model: for $\gb>0$ and $h\in \bbR$,
\begin{equation}
\label{pinning}
\cZ_{\gb,h,{\rm pin}}^{\go,a} := 1+\sum_{k=1}^\infty \beta^k\int_{0<t_1<\dots<t_k<1}  \prod_{i=1}^k u_{\gamma} (t_i-t_{i-1}) \prod_{i=1}^k (\xi_{\go}^{(a)} + h\cL ) ( \dd t_i)\,,
\end{equation}
where $u_{\gamma}(t) := c_{\gamma} t^{-(1-\gamma)}$
is the transition kernel of the regenerative set of index $\gamma$.

Here, no condition analogue to~\eqref{hypolarge} is needed to keep $\cZ_{\gb,h,{\rm pin}}^{\go,a} $ a.s.\ finite, since there is no spatial dimension.
On the other hand, in analogy with~\eqref{hyposmall}, in view of the form $u_{\gamma}(t) = c_{\gamma} t^{-(1-\gamma)}$  and since there is no spatial dimension, 
a (near-optimal) condition that ensures that
$\lim_{a\to 0}\cZ_{\gb,h, {\rm pin}}^{\go,a}$ is non-degenerate should therefore be 
\begin{equation}
\label{hyposmallpinning}
\int_{(0,1)} \ups^{p} \lambda(\dd \ups) <\infty \, ,\quad \text{ for some } p<  \min\big( \tfrac{1}{1-\gamma} , 2 \big)\, .
\end{equation}

Hence, we conjecture that if~\eqref{hyposmallpinning} holds, the partition function~$\cZ_{\gb,h,{\rm pin}}^{\go,a}$ defined in~\eqref{longrange} converges a.s.\ to a non-degenerate limit, and that one can construct a continuum measure corresponding to the disordered pinning model in L\'evy noise.
In the case of an $\ga$-stable noise (\textit{i.e.} $\lambda(\dd \ups) = \ga \ups^{-(1+\ga)} \dd \ups$), 
the condition~\eqref{hyposmallpinning} translates into $\ga < \min \big( \tfrac{1}{1-\gamma}, 2 \big)$,
which corresponds to the disorder relevance condition found in~\cite{LS17} (where the roles of $\gamma$ and $\alpha$ are exchanged).
Additionally, in analogy with Theorem~\ref{thm:conv}, the continuum pinning model in  $\ga$-stable noise should then appear
as the scaling limit of the disordered pinning model defined
above in~\eqref{def:pinning}
and heavy tailed disorder satisfying~\eqref{def:eta}.

\subsubsection{Other open questions}

To conclude this section, we present a brief list of interesting open questions.

\smallskip
\textit{a)} A first question that we already raised is that of
considering a more general noise. 
We leave as an open problem the
issue of adding a Gaussian white-noise component to  $\xi_{\go}$.  We believe that 
in dimension $d=1$
the partition function converges to a non-degenerate limit even when a Gaussian component is added to the noise and that consequently on can  define a continuum polymer in that case.

\smallskip
\textit{b)} Another natural question is that of the $\bbL_p$ convergence in Theorem~\ref{thm:zalpha}.
It is natural to expect some $\bbL_p$ convergence to hold, 
but this appears to be technically challenging.
We leave as an open problem to show that,
if $\int_{(0,1)} \ups^{p} \lambda(\dd \ups) < \infty$
for some $p < \min(1+\frac2d,2)$ and 
 $\int_{[1,\infty)} \ups^{q} \lambda(\dd \ups) < \infty$
for some $q \geq 1$,
then $\cZ_{\gb,T}^{\go,a}$ converges to $\cZ_{\gb}^{\go}$
in~$\bbL_{\min(p,q)}$.

\smallskip
\textit{c)} To conclude, let us mention an important and challenging question.
In the case of an $\alpha$-stable noise, 
we have treated the case $\ga < \ga_c = \min(1+\frac{2}{d},2)$; in particular, $\ga_c <2$ in dimension $d\geq 3$.
It would then be an interesting question to investigate the
case $\ga = \ga_c$, called \emph{marginal} (in particular in the case where $\ga_c <2$, since marginal behavior may depend on $\ga_c$).
In analogy with other marginally relevant
disordered systems (see~\cite{CSZ15} in the context of scaling limits),
one
 would then expect that  one should find a non-trivial limit in distribution by allowing $\beta$ to depend on $a$ in such a way that $\lim_{a\to 0}\beta_a=0$ (recall that from Proposition \ref{prop:tozero} we have $\lim_{a\to 0}\cZ_{\gb,T}^{\go,a}=0$ in that case). It seems likely (again considering the analogy with~\cite{CSZ15}) that the right scaling for $\beta_a$ should be of the form $\beta_a= c(\log a)^{-\nu}(1+o(1))$ where $c$ and $\nu$ are positive constant. 
 
 \begin{rem}
  Since the first version of this paper appeared as a draft, progress has been made in \cite{BCL21} on several fronts including a  few of the open questions mentionned above. We refer to the introduction of \cite{BCL21} for a full account.
 \end{rem}

\subsection{Organisation of the rest of the paper}

Let us briefly present how the rest of the paper is organized and outline the ideas of the proofs of Theorems~\ref{thm:continuous}
and~\ref{thm:zalpha}.

\smallskip
\textbullet\
 In Section~\ref{sec:conti} we present preliminary
results concerning the partition function with truncated noise $\cZ_{T,\gb}^{\go,a}$  that are needed in the rest of the paper.
We prove in particular its well-posedness (Proposition~\ref{prop:contiZfini}), its positivity (Lemma~\ref{eazy}, which provides an important alternative representation for the partition function),
and a martingale property (under suitable integrability condition, see Lemma~\ref{lem:martingale}).
We also give an enlightening representation of
the size-biased  law of the environment (\textit{i.e.}\ its law biased by the partition function, see the definition~\eqref{tildemeasure})
and we recall Mecke's multivariate equation for Poisson point processes, which is used throughout the paper.

\smallskip
\textbullet\
In Section~\ref{sec:UI}, we prove our main result, that is, Theorem~\ref{thm:zalpha}
(Theorems~\ref{thm:continuous} and~\ref{thm:contmeasure} being only particular cases).
The proof needs to be decomposed in several steps, a detailed account of which is given in Section \ref{sec:orga}.  Most of the proofs of this section can be adapted to control the point-to-point
partition function, and thus we prove along the way Proposition~\ref{th:superlem} and Proposition \ref{lem:SHE}.

\smallskip
\textbullet\
In Section~\ref{sec:trivial}, we study the cases where the limiting partition function degenerates either to zero or infinity, that is we prove  Proposition~\ref{prop:toinfinity} and 
 Proposition~\ref{prop:tozero}.

\smallskip
\textbullet\
In Section~\ref{sec:properties}, we prove the various properties of the continuum directed polymer in L\'evy noise
that are gathered in Section~\ref{sec:mainprop},
that is 
Proposition~\ref{abzolute},
 Proposition~\ref{emptydance}
and 
Proposition~\ref{twomarg}.

\textbullet\
In Section \ref{sec:propSHE}, 
 we prove our statement concerning the convergence of the solution of the SHE with truncated noise, Proposition~\ref{prop:SHE}.

\smallskip

\noindent Finally,  we collect in the appendix several technical results that are used along the paper.

\medskip
\noindent
{\bf Notational warning.}
For simplicity we assume in the rest of the paper that $T=1$, and we drop the dependence in~$T$ in all notations.

\section{Preliminaries: some properties of $\cZ_{\gb}^{\go,a}$}
\label{sec:conti}

We let $|\xi_{\go}^{(a)}|$ denote the total variation associated with the locally finite signed measure~$\xi_{\go}^{(a)}$ defined in~\eqref{petitnoise},
and we  let
 $$
 \mathfrak{X}^k:=\big\{  \bt = (t_1, \ldots, t_k) \in \bbR^k \ : \  0<t_1<\dots<t_k<1  \big\},
 $$
denote the open simplex.

\subsection{Well-posedness}

Our first task is to check that our definitions $\cZ^{\go,a}_{\beta}$ in \eqref{lafirstdef} and $\cZ^{\go,a}_{\beta}(f)$ in~\eqref{lafirstdeffpti} are well posed.  This is given by the following result.
\begin{proposition}
\label{prop:contiZfini}
 For any choice of  $\gl$ satisfying $\lambda([a,\infty)) <+\infty$ for every $a>0$,
for any~$f\in \cB_b$,
the function 
$\varrho(\bt,\bx,f)$ defined on $\mathfrak{X}^k\times (\bbR^d)^k$
is almost surely integrable with respect to the product measure~$|\xi_{\go}^{(a)}|^{\otimes k}$.
Moreover we almost surely have, for any $\beta>0$,
\begin{equation}\label{finitesum2}
 \sum_{k=0}^\infty \beta^k\int_{\mathfrak{X}^k \times (\bbR^d)^k} \varrho(\bt,\bx,f)  \prod_{i=1}^k |\xi_{\go}^{(a)}|(\dd t_i ,\dd x_i)<\infty \, .
\end{equation}
If $\mu:=\int_{[1,\infty)} \ups \gl(\dd \ups) <\infty$, then $\varrho(\bt,\bx)$ is integrable  with respect to the product measure~$|\xi_{\go}^{(a)}|^{\otimes k}$ and 
\begin{equation}\label{finitesum}
 \sum_{k=0}^\infty \beta^k\int_{\mathfrak{X}^k \times (\bbR^d)^k} \varrho(\bt,\bx) \prod_{i=1}^k |\xi_{\go}^{(a)}|(\dd t_i,\dd x_i)<\infty \, .
\end{equation}
Furthermore, we have for all $f\in \cB_b$
 \begin{equation}\label{lamine}
 \forall\, a\in (0,1], \, \quad \bbE \big[ \cZ^{\go,a}_{\beta}(f) \big] =  e^{\gb \mu} \bQ(f)\, .
 \end{equation}
\end{proposition}

\begin{proof}

Let us start with the case
$\int_{[1,\infty)} \ups \gl(\dd \ups) <\infty$.
It is sufficient to check that the expectation of the l.h.s.\ in \eqref{finitesum} is finite.
Now using the definition~\eqref{petitnoise} for  $\xi_{\go}^{(a)}$ we have on $\mathfrak{X}^k\times (\bbR^d)^k$, 
\[
\bbE\left[\otimes_{i=1}^k |\xi_{\go}^{(a)}|(\dd t_i,\dd x_i)\right]= \Big(\int_{[a,\infty)} \ups \gl(\dd \ups)  +\kappa_a\Big)^k\prod_{i=1}^k \dd t_i\dd x_i \, .
\]
Letting $C_a:= \int_{[a,\infty)} \ups \gl(\dd \ups)  +\kappa_a <\infty$, we therefore get that
\[
\bbE \bigg[ \int_{\mathfrak{X}^k \times (\bbR^d)^k}  \varrho(\bt,\bx) \prod_{i=1}^k |\xi_{\go}^{(a)}| ( \dd t_i, \dd x_i) \bigg]
=    C_a^k  
\int_{\mathfrak{X}^k \times (\bbR^d)^k}   \varrho(\bt,\bx)  \prod_{i=1}^k \dd t_i \dd x_i =   \frac{  C_a^k }{k!} \, . 
\]
This implies both the convergence of the integral and
the summability in $k$.

The fact that $\bbE[\cZ^{\go,a}_{\beta}(f)]= e^{\gb \mu} \bQ(f)$  directly follows from the definition~\eqref{def:kappa}-\eqref{lafirstdeffpti}
and Fubini,
using that $\prod_{i=1}^k\xi_{\go}^{(a)} ( \dd t_i, \dd x_i)$  has mean $(\mu \cL)^{\otimes k}$.

\smallskip
Now let us prove \eqref{finitesum2} when  $\int_{[1,\infty)} \ups \gl(\dd \ups) =\infty$. 
For this we first consider a truncated version of the noise to place ourselves back in the integrable case, and then let the truncation threshold go to infinity. This procedure is going to be used repeatedly in the paper.
For $b>a$, recall the definition~\eqref{splitnoise}
of $\xi^{[a,b)}_\go$.
 Using the assumption $f\in \cB_b$, we let $M>0$ be such that $f(\varphi)=0$ if $\|\varphi\|_{\infty}\ge M$. Then $\varrho(\bt,\bx,f)=0$ if $\max_{i=1}^k \|x_i\|_{\infty} \ge M$.
Therefore,  since $\lambda([a,\infty))<\infty$, there exists $b_0(M,\go)$ such that for every $b>b_0$, the restriction of $\xi^{(a)}_{\go}$ on $[0,1]\times [-M,M]^{d}$ coincides with that of $\xi^{[a,b)}_{\go}$. Hence it is sufficient to show that \eqref{finitesum2} holds for $\xi^{[a,b)}_{\go}$ for every~$b>1$, which we can do by repeating the proof of \eqref{finitesum}.
\end{proof}

\begin{rem}
\label{rem:contiZfini}
Notice that we have the analogous result for the point-to-point partition function $\cZ_{\gb}^{\go,a} (t,x)$.
For any $x\in \bbR^d$,
the function  $\varrho (\bt,\bx)  \rho_{t-t_k}(x-x_k)$
is almost surely integrable on $\mathfrak{X}^k \times [-M,M]^k$ for any $M>0$
 with respect to the product measure $|\xi_{\go}^{(a)}|$,
 and it is integrable on $\mathfrak{X}^k \times (\bbR^d)^k$ if $\int_{[1,\infty)} \ups \lambda(\dd \ups) <\infty$.
 We also have 
 \begin{equation}\label{zamine}
   \forall\, a\in (0,1], \, \quad \bbE \big[ \cZ^{\go,a}_{\beta}(t,x) \big] =  e^{\gb \mu t} \rho_t(x).
 \end{equation}
 \end{rem}

\subsection{The partition function $\cZ_{\gb}^{\go,a}$ as a sum}

Let us now present an alternative expression for~$\cZ^{\go,a}_{\beta}$,
from which it will be clear that $\cZ^{\go,a}_{\gb}$ is positive.
 Indeed, it  is not obvious that the chaos decomposition~\eqref{lafirstdeffpti} is non-negative,
due to the term  $-\kappa_a \cL$ in the definition~\eqref{petitnoise} of $\xi_{\go}^a$.
The idea of Lemma~\ref{eazy} below is to integrate out this ``Lebesgue part'' of $\xi_{\go}^a$. This operation has the effect of giving  rise to a prefactor  $e^{-\gb \kappa_a}$.

We let  $\cP_{[0,t]}(\go)$ denote the set of finite collections of points in $\go$ whose time coordinates belong to the interval $[0,t]$. When $t=1$ we simply write $\cP(\go)$. We define similarly $\cP_{[0,t]}(\go^{(a)})$ and $\cP(\go^{(a)})$ the sets of finite collections of points in  $\go^{(a)} = \{(t,x,\ups)\in \go \colon \ups\ge a\}$ (as defined above~\eqref{def:omegabar}).
For $\sigma \in \cP(\go)$ we let $|\sigma|$ denote its cardinality 
and we use the notation $(t_i,x_i,u_i)_{i=1}^{|\sigma|}$ to denote the points in $\sigma$ ordered in increasing time.
Given $a>0$ we define the following weight function $w_{a,\beta}(\sigma)$ on $\cP(\go)$
\begin{equation}\label{weights}
 w_{a,\beta}(\sigma):= e^{- \beta  \kappa_a } \gb^{|\sigma|}   \varrho(\bt,\bx) \prod_{i=1}^{|\sigma|}  u_i\ind_{\{u_i\geq a\}} \, ,
 \end{equation}
  with $\kappa_a =\int_{[a,1)} \ups \lambda (\dd \ups)$ as defined in~\eqref{def:kappa}.
 Let us stress here that $w_{a,\gb}$ puts a positive weights only on 
 elements of $\cP(\go^{(a)})$. By convention, we say that 
 the empty set belongs to~$\cP(\go)$ and we set 
$w_{a,\beta}(\emptyset):=e^{-\beta \kappa_a }$.
 Similarly, for any $f\in \cB$, we define
 \begin{equation}\label{weightsf}
 \begin{split}
 w_{a,\beta}(\sigma,f)&:= e^{- \beta  \kappa_a }  \gb^{|\sigma|}  \varrho(\bt,\bx,f) \prod_{i=1}^{|\sigma|}  u_i\ind_{\{u_i\geq a\}}, \\ w_{a,\beta}(\emptyset,f)&:= e^{-\gb \kappa_a} \bQ(f).
 \end{split}
 \end{equation}

 The following lemma provides an alternative formulation of the partition function which is convenient to assert positivity.
  Its proof is straightforward (one only needs to integrate out the Lebesgue part of $\xi_{\go}^a$) and is presented in Appendix~\ref{app:eazy} for completeness.

 \begin{lemma}\label{eazy}
Given $a>0$, $\beta>0$ and $f\in \cB_b$ we have 
\begin{equation}\label{altexpf}
 \cZ^{\go,a}_{\beta} (f) = \sum_{\sigma \in \cP(\go)} w_{a,\beta}(\sigma,f) \, .
\end{equation}
In particular $\cZ_{\gb}^{\go,a}(f) > 0$ if $f\ge 0$ and $\bQ(f)>0$.
Also, by monotone convergence, recalling the definition~\eqref{def:Zmonotone}, we have
\begin{equation}\label{altexp}
 \cZ^{\go,a}_{\beta}= \sum_{\sigma \in \cP(\go)} w_{a,\beta}(\sigma) \, .
\end{equation}
\end{lemma}
\noindent
Let us stress that the representations~\eqref{altexpf}
 and~\eqref{altexp} are valid without any assumption
 on the intensity measure $\gl$,
 since all the terms in the sums are positive, but it may be the case that both sides of \eqref{altexp} are infinite.

\begin{rem}
Similarly, for the point-to-point partition function, the following identity holds 
\begin{equation}
\label{altexpp2p}
\cZ^{\go,a}_{\beta}(t,x) = \sum_{\sigma \in \cP_{[0,t]}(\go)}
w_{a,\gb}(\sigma, (t,x)) \, ,
\end{equation}
with
$w_{a,\gb}(\sigma,(t,x)) := e^{-\gb\kappa_a} \gb^{|\sigma|} \varrho(\bt,\bx) \rho_{t-t_k}(x-x_k)  \prod_{i=1}^{|\sigma|} u_i \ind_{\{u_i\geq a\}}$,
as long as the r.h.s.\ is finite.
Note that by Fubini's theorem this gives that, for any $a>0$
\begin{equation}
\label{integreZ}
\int_{\R^d} \cZ^{\go,a}_{\beta}(t,x)  \dd x=  \cZ_{\gb,t}^{\go,a} \, .
\end{equation}
\end{rem}

\subsection{Martingale property}

In the case $\int_{[1,\infty)} \ups \lambda(\dd \ups)<\infty$,
the convergence of $\cZ^{\go,a}_{\beta}(f)$ as $a\downarrow 0$ is an immediate consequence of the following observation.

\begin{lemma}
\label{lem:martingale}
Let $\cF = (\cF_a)_{a \in (0,1]}$ be the filtration where $\cF_a$ is the $\sigma$-field generated by~$\go^{(a)}$.
If the measure $\gl$ satisfies $\mu:=\int_{[1,\infty)} \ups \gl(\ups)\dd \ups<\infty$, then the following processes
are (time-reversed) martingales  for the filtration~$\cF$:
\begin{itemize}
\item   $(\cZ^{\go,a}_{\beta}(f))_{a\in (0,1]}$ for any  $f\in \cB$, and in particular $(\cZ^{\go,a}_{\beta})_{a\in(0,1]}$;
\item $(\cZ^{\go,a}_{\beta}(t,x))_{a\in (0,1]}$ for any  $(t,x)\in \bbR_+^*\times  \bbR^d$.
\end{itemize}
The mean of these martingales are $\bbE[\cZ^{\go,a}_{\beta}(f))] =e^{\gb \mu} \bQ(f)$ 
and $\bbE[ \cZ^{\go,a}_{\beta} (t,x)]= e^{\gb \mu t}\rho_t(x)$.

Moreover, if $g$ is a bounded measurable function
of $\varphi$ and $\go$ and $g(\varphi,\go)$ is $\cF_{a_0}$-measurable for every $\varphi$,
then $\big(\cZ^{\go,a}_{\beta}(g(\cdot, \go))\big)_{a\in(0,a_0]}$ 
 is a (time-reversed) martingale. 
\end{lemma}

\begin{proof}
Using the expression \eqref{lafirstdeffpti} 
 (or \eqref{def:pointtopoint} for the point-to-point partition function), the result follows from the fact that the sequence of measures $( \prod_{i=1}^k\xi_{\go}^{(a)} ( \dd t_i, \dd x_i) )_{a\in (0,1]}$ on $\mathfrak{X}^k \times (\bbR^d)^k$ is a martingale. 
Indeed 
for $b< a\leq 1$ we have
\begin{multline*}
 \bbE \bigg[ \prod_{i=1}^k\xi_{\go}^{(b)} ( \dd t_i, \dd x_i)-\prod_{i=1}^k\xi_{\go}^{(a)} ( \dd t_i, \dd x_i)  \,\Big|\, \cF_a \bigg]\\
 = \sum_{i=1}^k  \bbE \bigg[ \Big(\prod_{j=1}^{i-1}\xi_{\go}^{(b)}  ( \dd t_j, \dd x_j) \Big)(\xi_{\go}^{(b)}-\xi_{\go}^{(a)})( \dd t_i, \dd x_i)
 \Big(\prod_{j=i+1}^{k}\xi_{\go}^{(a)} ( \dd t_j, \dd x_j) \Big)  \,\Big|\,  \cF_a \bigg]=0,
\end{multline*}
where in the last equality we used that by construction $(\xi_{\go}^{(b)}-\xi_{\go}^{(a)})( \dd t_i, \dd x_i)$ is of zero average, independent of $\cF_a$ and conditionally independent of 
$\prod_{j=1}^{i-1}\xi_{\go}^{(b)} ( \dd t_j, \dd x_j)$.
The proof for
a random function  follows the same line, using that 
$\varrho(\bt,\bx,g(\cdot, \go))$ is $\cF_{a_0}$-measurable for all $\bt$ and $\bx$.
 \end{proof}

Since $\cZ_{\gb}^{\go,a}\geq 0$, this directly implies
in the case $\int_{[1,\infty)} \ups \gl(\ups)\dd \ups<\infty$  that  $\lim_{a\downarrow 0}\cZ^{\go,a}_{\beta}$ exists almost surely. 
We will show that if additionally assumption~\eqref{hyposmall} holds,
the martingale is uniformly integrable.

 \begin{rem}
 In the case $\int_{[1,\infty)} \ups \gl(\ups)\dd \ups=\infty$,
we will consider the
truncated  partition function
$\cZ_{\gb}^{[a,b)}(f)$ defined as in \eqref{lafirstdeff} but with $\xi_{\go}^{(a)}$ replaced by the truncated noise $\xi_{\go}^{[a,b)}$ defined in~\eqref{splitnoise} ---~this corresponds to considering
the intensity measure $\lambda^{(0,b)} (\dd \ups) := \ind_{\{\ups < b\}} \lambda(\dd \ups)$.
Then, for any $b>0$, Lemma~\ref{lem:martingale} shows that $(\cZ_{\gb}^{[a,b)}(f))_{a\in (0,b]}$ is a martingale.
 \end{rem}


\subsection{A representation for the size-biased measure}
When $\mu:=\int_{[1,\infty)} \ups \gl(\ups)\dd \ups<\infty$,  since  
\[
\bar \cZ^{\go,a}_{\beta}:= e^{-\beta \mu} \cZ_{\gb}^{\go,a}
\]
 is non-negative and of average one, one can define 
an alternative measure $\tilde \bbP^a_{\beta}$  for the environment, defined by 
\begin{equation}
\label{tildemeasure}
\tilde \bbP^a_{\beta}(\go\in A):= \bbE\big[ \bar \cZ_{\gb}^{\go,a} \ind_{\go\in A}\big] \, .
\end{equation}

The measure $\tilde \bbP^a_{\beta}$ is often referred to as the size-biased measure
---~the probability of an event is biased by the ``size'' of the partition function.
Convenient representations of the size-biased measure have been given for directed polymers \cite[Lemma 1]{Bir04} and similar models such as branching random walks (see~\cite[Ch.~4]{ShiBRW} and references therein) or the disordered pinning model \cite[\S\,5.2]{LS17}.
The size-biased measure for all these models is obtained by tilting the distribution
of the environment along a randomly chosen trajectory.
The result we present below is a strict analog in a continuous setup. 

\smallskip

We let~$\bbP'_a$ be the distribution of a Poisson point process $\go'_a$ on $[0,1]\times\bbR_+$ whose intensity is $\dd t \otimes \beta   \ups  \ind_{\{\ups \geq a\}}\gl(\dd \ups)$, (that is $\dd t \otimes \beta\alpha \ups^{-\alpha}\ind_{\{\ups \geq a\}} \dd \ups$ in the $\alpha$-stable case)
and we recall that $\mathbf{Q}$ is the distribution 
of a standard Brownian motion.
We then introduce the random set of points $\hat \go(\go, \go_a',B)$ in $\bbR\times \bbR^{d}\times \bbR_+$ defined by 
\begin{equation}
\label{def:hatomega}
 \hat \go:= \go \cup  \big\{ (t,B_t,\ups) \ : \ (t,\ups)\in \go'_a \big\}  \, .
\end{equation}
Then, the distribution of $\omega$ under the measure $\tilde \bbP^a_{\beta}$ can be described as follows.

\begin{lemma}\label{spayne}
Suppose that $\mu:=\int_{[1,\infty)} \ups \gl(\ups)\dd \ups<\infty$. 
Then with the notation defined above, for any measurable bounded function $g$ we have
\begin{equation}\label{sides}
 \tilde \bbP^a_{\beta}\left[ g(\go)\right]= \bbP\otimes \bbP'_a\otimes \bQ\left[ g(\hat \go(\go,\go'_a,B))\right]
\end{equation}
In other words, the distribution of $\go$ under  $\tilde \bbP^a_{\beta}$ is obtained by adding to the original point process an independent Poisson process of intensity 
 $\dd t \otimes  \beta\ups    \ind_{\{\ups \geq a\}} \gl(\dd \ups)$ drawn on the trajectory a Brownian Motion.
\end{lemma}
The proof, though elementary, requires some cumbersome computation. We present it in Appendix~\ref{App:spayne} for completeness.

\subsection{An important tool: Mecke's multivariate equation}

Let us recall here a classical formula for Poisson point processes which we will repeatingly use in our computations.
It is a particular case of Mecke's multivariate equation
 (see e.g.\ \cite[Theorem 4.4]{PoiBook}).

 \begin{proposition}\label{OKLMecke}
 Given $\gl$ a sigma-finite measure  on a measurable space $(\bbX,\mathcal{X})$, and $\go$ 
a Poisson point process with intensity $\gl$, then
 for any $k\in \bbN$ and any measurable function $g: \bbX^k \to \bbR_+$ such that  $g (x_1,\dots, x_k)=0 $ as soon as $x_i=x_j$ for some $i\ne j$ then 
\begin{equation}
\bbE \bigg[ \sum_{(x_1,\dots, x_k)\in \go^k} g(x_1,\dots, x_k) \bigg]= \int_{\bbX^k} g(x_1,\dots,x_k) \gl^{\otimes k}(\dd x_1,\dots ,\dd x_k). 
\end{equation}
  
 \end{proposition}
Of course we are going to apply this formula for the Poisson process $\go$. In our applications we mostly deal with sums running on subsets of $\go$ whose cardinality is not fixed, see the expression \eqref{altexp} for the partition function above. Hence, in practice, the formula  we will use is rather 
\begin{equation}\label{nonfixedcard}
\bbE \bigg[ \sum_{k\ge 1} \sum_{(x_1,\dots, x_k)\in \go^k} g_k(x_1,\dots, x_k) \bigg]= \sum_{k\ge 1} \int_{\bbX^k} g_k(x_1,\dots,x_k) \gl^{\otimes k}(\dd x_1,\dots ,\dd x_k), 
\end{equation}
where $g_k$ is a sequence of positive functions on $\bbX^k$.

 \section{Convergence of the partition function and of the measure: proof of Theorem~\ref{thm:zalpha}}
 \label{sec:UI}

 \subsection{Organization of the section}
 \label{sec:orga}

We decompose the proof of the theorem  in several steps. We provide the  details of this decomposition before going to the core of the proof.

\smallskip
\noindent
{\it First step.}
Our first and main task is to prove the convergence of the partition function under the additional assumption $\int_{[1,\infty)} \ups \gl( \dd \ups)<\infty$.

\begin{proposition}
\label{thm:continuous2}
If the measure $\gl$ satisfies $\int_{[1,\infty)} \ups \gl( \dd \ups)<\infty$ and \eqref{hyposmall}, then the martingale 
$(\cZ^{\go,a }_{\beta})_{a\in(0,1]}$
 is uniformly integrable.
As a consequence there exists $\cZ^{\go }_{ \beta}$ such that the following convergence holds holds almost surely and in $\bbL_1$, 
\begin{equation*}
 \lim_{a \to 0}\cZ^{\go,a }_{\beta}=\cZ^{\go }_{ \beta}
\end{equation*}

\end{proposition}
Since from Lemma \ref{lem:martingale}  we know that $(\cZ^{\go,a}_{\beta})_{a\in(0,1]}$ is a positive martingale ,
 it is sufficient to show that $(\cZ^{\go,a}_{\beta})_{a\in(0,1]}$ is uniformly integrable. 
 Our strategy consists in considering a sequence of approximation $(\hat \cZ_{\gb,q}^{\go,a})_{q\ge 1}$ of $\cZ^{\go,a}_{\beta}$,
 obtained by somehow restricting the partition function 
 to ``not-too-large'' weights.
 We choose our restriction so that two key properties are satisfied
 \begin{itemize}
  \item [(A)] For large $q$'s, $\hat \cZ_{\gb,q}^{\go,a}$ is a good approximation of $\cZ^{\go,a}_{\beta}$ in~$\bbL_1$, uniformly in $a$.
  \item [(B)] For any $q$, $(\hat \cZ_{\gb,q}^{\go,a})_{a\in(0,1]}$ is bounded in $\bbL_2$.
 \end{itemize}
 We refer to Section~\ref{sec:uniformcriterion} for 
 a more detailed description of this strategy, which is then implemented 
in Section~\ref{sec:dim1} in dimension $d=1$ and in Section~\ref{sec:contd>2} in dimension $d\geq 2$, where the restriction strategy is more subtle.

\medskip

\noindent Notice that from Lemma \ref{eazy} (in particular \eqref{altexp}),
we have  that 
$\cZ^{\go,a }_{\beta}(f)\le \|f\|_{\infty} \cZ^{\go,a }_{\beta}$ for any $f\in \cB$. Hence an immediate consequence of 
Proposition \ref{thm:continuous2} is the following.

\begin{cor}\label{toutlesf}
 If the measure $\gl$ satisfies $\int_{[1,\infty)} \ups \gl( \dd \ups)<\infty$ and \eqref{hyposmall}, then for every $f\in \cB$ the martingale 
 $(\cZ^{\go,a }_{\beta}(f))_{a\in(0,1]}$ is uniformly integrable, and
 the following convergence holds almost surely and in $\bbL_1$
 \begin{equation*}
 \lim_{a \to 0}\cZ^{\go,a }_{\beta}(f)=\cZ^{\go }_{ \beta}(f).
\end{equation*}
\end{cor}

\smallskip
\noindent
{\it Second step.} Our second task is to  remove the assumption $\int_{[1,\infty)} \ups \gl( \dd \ups)<\infty$ from Proposition~\ref{thm:continuous2} (and Corollary~\ref{toutlesf}), \textit{i.e.}\ to prove Proposition \ref{prop:Zfinite}.  This is done in Section~\ref{sec:generalnoiseproof}. Along the way we also prove Proposition \ref{lem:SHE} and the first part of Proposition \ref{th:superlem}, that is \eqref{lafinite}.
We then use Proposition~\ref{prop:Zfinite} to prove the following lemma (which corresponds to \eqref{convf}-\eqref{convparty}),  in Section \ref{sec:lemmaconv}.

\begin{lemma}\label{lemmaconv}
 Under the assumption \eqref{hyposmall}, for every $f\in \cB_b$
 the following convergence holds almost surely  and the limit is finite
 \begin{equation*}
 \lim_{a \to 0}\cZ^{\go,a }_{\beta}(f)=\cZ^{\go }_{ \beta}(f) .
\end{equation*}
 Furthermore if \eqref{hypolarge} also holds  then the statement is valid for $f\in \mathcal B$. In particular, we have
  \begin{equation}
   \label{eq:lemmaconv}
 \lim_{a \to 0}\cZ^{\go,a }_{\beta}=\cZ^{\go }_{ \beta}  <+\infty \, .
\end{equation}
\end{lemma}

\smallskip
\noindent
{\it Third step.} Our third task, which is crucial for the convergence of $\bQ^{\go,a}_{\beta}$, is to ensure that the limiting partition function is positive (let us record the statement as a proposition).
 This is done in Section~\ref{sec:positivity}.

\begin{proposition}\label{lapositivity}
 If $\gl$ satisfies \eqref{hyposmall}, then for any non-negative $f\in \cB_b$ with $\bQ(f)>0$,
 we have almost surely 
 \begin{equation}
 \cZ^{\go }_{\beta}(f)>0.
 \end{equation}
As a consequence, if $\gl$ satisfies  \eqref{hypolarge}-\eqref{hyposmall} then we have almost surely
\begin{equation}
 \label{eq:lapositivity}
 \cZ^{\go }_{\beta} \in (0,\infty)\, .
\end{equation}
\end{proposition}

\smallskip
\noindent
{\it Fourth step.}
 Finally we complete the proof of Theorem \ref{thm:zalpha} by proving the convergence of~$\bQ^{\go,a}_{\beta}$. Note that  Lemma~\ref{lemmaconv} and Proposition \ref{lapositivity} imply for any given  $f\in \cB_b$ the almost sure convergence of  $\bQ^{\go,a}_{\beta}(f)$. Hence we only need to prove tightness.

\begin{proposition}\label{proptight}
 If  $\gl$ satisfies  \eqref{hypolarge}-\eqref{hyposmall},
 then for almost every $\go$,  the family of measures $(\bQ^{\go,a}_{\beta})_{a\in(0,1]}$ is tight
 with respect to the topology of weak convergence on $\cM_T$,
the set of probability measures on $C_0([0,1])$.
\end{proposition}

\noindent The reader can then check that combining all the statements above yields the complete proof of Theorem~\ref{thm:zalpha}.

\medskip

Let us finally comment on how the proof of the second part of Proposition \ref{th:superlem} (the convergence \eqref{lapazero} to a positive limit) is completed.
We simply need to show that Proposition~\ref{thm:continuous2} and Proposition \ref{lapositivity} (that is Equation \eqref{eq:lapositivity})
remain valid for the point-to-point partition function.
Since the proofs are nearly identical we will point at the end of the various proofs which modifications are required, when there are any.

 \subsection{A uniform integrability criterion}
 \label{sec:uniformcriterion}
As outlined above, our proof of uniform integrability is going to rely on a second moment computations. This requires to overcome some subtleties since the second moment of $\cZ^{\go,a}_{\beta}$  might be infinite for every~$a>0$ (this is for instance the case when $d\ge 2$).
 We follow an approach similar to the one used in \cite{ber2017} for the proof of the convergence of Gaussian Multiplicative Chaos. We look for a family of restrictions of the partition functions which is bounded in~$\bbL_2$ but does not produce any loss of mass at infinity. 
 Let us summarize our approach in the form of a proposition.

 \begin{proposition}\label{metaprop}
 Consider $(X_a)_{a\in (0,1]}$ a collection of positive random variables.
 Assume that there exists  
 $X^{(q)}_a$ a sequence of approximation of $X_a$, indexed by $q\ge 1$, which satisfies:
\begin{align*}
\mathrm{(A)} \qquad  &  \lim_{q\to \infty}  \sup_{a\in(0,1]} \bbE\big[ |X^{(q)}_a-X_a|\big]=0\,  ; \\
 \mathrm{(B)}  \qquad  & \sup_{a\in(0,1]}   \bbE\big[ (X^{(q)}_a)^2\big]<\infty \quad \text{ for every $q\geq 1$}.
\end{align*}
Then $(X_a)_{a\in(0,1]}$ is uniformly integrable.
 \end{proposition}

 \begin{proof}
  We may write, for any $M>0$ and $a>0$, 
\begin{align*}
\bbE\big[ |X_a| \ind_{\{|X_a|>M\}}  \big] \leq \bbE\big[  |X^{(q)}_a-X_a| \big] + 
  \bbE\big[ (X^{(q)}_a)^2\big]^{1/2} \bbP\big(|X_a|>M \big)^{1/2},
\end{align*}
where we have used Cauchy--Schwarz inequality for the second term.
Applying Markov's inequality and taking the supremum over $a\in(0,1]$,
we therefore get 
\begin{align*}
\sup_{a\in(0,1]} \bbE\big[ |X_a| \ind_{\{|X_a|>M\}}  \big] 
& \leq \sup_{a\in(0,1]} \bbE\big[  |X^{(q)}_a-X_a| \big]  \\
&\qquad  \qquad + M^{-1/2} \sup_{a\in(0,1]}\bbE\big[ (X^{(q)}_a)^2\big]^{1/2}  \sup_{a\in(0,1]} \bbE[|X_a|]^{1/2} \, .
\end{align*}
The first term can be made smaller than $\gep/2$ by choosing $q$ sufficiently large. Then once $q$ is fixed,  we can make the second term smaller than $\gep/2$ by choosing $M$ large (our assumptions imply that $(X_a)_{a\in(0,1]}$ is bounded in $\bbL_1$).
 \end{proof}

Our idea is now to apply Proposition~\ref{metaprop} to variables  $X^{(q)}_a$ which are obtained by considering the sum of the weights $w_{a,\beta}(\sigma)$ on a strict subset of $\cP(\go)$ (recall the representation of $\cZ_{\gb}^{\go,a}$ in Lemma~\ref{eazy}).

\subsection{Proof of Proposition~\ref{thm:continuous2} in dimension $d=1$}
\label{sec:dim1}

The case $d=1$ gives us the occasion to apply  Proposition \ref{metaprop} with a relatively simple setup. In this case, the only thing that prevents $\cZ^{\go,a}_{\beta}$ from being bounded in $\bbL_2$ are the large values of $u_i$. The modified partition function obtained by 
ignoring these points in the Poisson point process $\go$ turns out to be bounded in $\bbL_2$.
 The idea is thus to apply Proposition \ref{metaprop} for partition functions  with truncated environment, taking 
$X^{(q)}_a= \cZ^{\go,[a,q)}_{\beta}$ (recall the definition  after \eqref{defuab}).
We then prove the following.

\begin{proposition}\label{d11}
Suppose that $\lambda$ satisfies $\mu:=\int_{[1,\infty)} \ups \lambda (\dd \ups) <\infty$ and that~\eqref{hyposmall} holds.
For every $d\ge 1$, we have
\begin{equation}\label{complea}
  \lim_{q\to \infty}  \sup_{a\in (0,1]} \bbE \Big[    \big|  \cZ^{\go,a}_{\beta}-  \cZ^{\go,[a,q)}_{\beta}    \big|  \Big]=0.
\end{equation}
 Moreover, when $d=1$, we additionally have that for every $q\ge 1$
   \begin{equation}\label{sedstuf}
   \sup_{a\in(0,1]}   \bbE\Big[ \big( \cZ^{\go,[a,q)}_{\beta} \big)^2\Big]<\infty.
 \end{equation}
 \end{proposition}

These statements imply that both requirement of Proposition~\ref{metaprop} are satisfied and therefore that $(\cZ^{\go,a}_{\beta})_{a\in(0,1]}$ is uniformly integrable and converges in $\bbL_1$. As can be checked from the proof, \eqref{sedstuf} is false when $d\ge 2$ and in that case we will need a more subtle restriction for the set of trajectories (developed in the next subsection). While the latter restriction also covers the $d=1$ case, the proof presented in this section is considerably simpler,
and may prepare the reader for the more involved proof in dimension $d\geq 2$.
Additionally note that \eqref{complea} is valid when $d\ge 2$; it will be used in Section \ref{sec:contd>2}.

\begin{proof}[Proof of Proposition \ref{d11}]
To compute the expectation in \eqref{complea} we use
Proposition \ref{OKLMecke}:
recalling the definitions~\eqref{weights} of $w_{a,\gb}(\sigma)$  and~\eqref{def:kappa} of $\kappa_a$, we obtain that 
 $ \cZ^{\go,a}_{\beta}-  \cZ^{\go,[a,q)}_{\beta} \geq 0$ and that
\begin{multline}\label{originn}
 \bbE \Big[ \cZ^{\go,a}_{\beta}-  \cZ^{\go,[a,q)}_{\beta} \Big] = 
 \bbE \bigg[ \sum_{\sigma \in  \cP(\go)}  w_{a,\beta}(\sigma) \ind_{\{\exists i\in \lint 1, |\sigma|\rint,\, u_i\ge  q\}} \bigg]\\
 =   e^{-\beta \kappa_a}   \sum_{k\ge 1}  \beta^k \int_{\mathfrak{X}^k\times (\bbR^d)^k \times (0,\infty)^k}\ind_{\big\{ \max\limits_{1\le i\le k} u_i \geq q\big\}} \varrho(\bt,\bx)  \dd \bt \dd \bx  \prod_{i=1}^{k} u_i \ind_{\{u_i\ge a\}}   \lambda(\dd u_i).
\end{multline}
Note that the integral in $\bx$ and $\bt$ readily simplifies since we have (recall that $\rho_t(x)$ is a probability density)
\begin{equation}\label{denzit}
\int_{(\bbR^d)^k}\varrho(\bt,\bx)  \dd \bx = 1
\quad \text{and} \quad
\int_{\mathfrak{X}^k}  \dd \bt= \frac{1}{k!}.
\end{equation}
Hence, setting $\mu:=\int_{[1,\infty)} \ups \lambda(\dd \ups)$
 and  $ \mu_q:=\int_{[1,q)} \ups \lambda(\dd \ups)$, 
the r.h.s.\ of \eqref{originn} is equal to
\begin{equation}
 \label{kalkul}
\begin{split}
 e^{-\beta \kappa_a}   \sum_{k\ge 1}  \frac{\beta^k }{k!}
   & \int_{[a,\infty)^k}  \ind_{\big\{ \max\limits_{1\le i\le k} u_i\geq q\big\}}  \prod_{i=1}^{k}  u_i \lambda(\dd u_i)  \\
  & =  e^{-\beta \kappa_a}   \sum_{k\ge 1}   \frac{\beta^k }{k!}
  \bigg[\int_{[a,\infty)^k} \prod_{i=1}^{k}   u_i \lambda(\dd u_i)  -\int_{[a,q)^k} \prod_{i=1}^{k}   u_i \lambda(\dd u_i)\bigg] \\
& = e^{-\beta \kappa_a}  \sum_{k\ge 1}  \frac{\beta^k }{k!} \left[( \kappa_a+\mu)^k-(\kappa_a+\mu_q)^k\right]  =  e^{\gb \mu} -e^{\gb \mu_q}   \,.
\end{split}
\end{equation}
Hence we have
$
 \bbE \big[ \cZ^{\go,a}_{\beta}-  \cZ^{\go,[a,q)}_{\beta} \big]  = e^{\gb \mu} -e^{\gb \mu_q},
$
which does not depend on $a$, and converges to $0$ as $q\to\infty$ (notice that this is true even for $d\geq 2$).

To check \eqref{sedstuf}, set 
$\cA_q:=\{ (t_i,x_i,u_i)^k_{i=1} \, : \,  \forall i\in \lint 1,k\rint,\,  u_i\le q \}$.
By Lemma~\ref{eazy} we have
 \begin{equation}
    \bbE \Big[  \big( \cZ^{\go,[a,q)}_{\beta} \big)^2 \Big]= \bbE \bigg[  \sum_{\sigma_1, \sigma_2 \in  \cP(\go)}  w_{a,\beta}(\sigma_1)w_{a,\beta}(\sigma_2) \ind_{\cA_q}(\sigma_1)\ind_{\cA_q}(\sigma_2) \bigg].
 \end{equation}
In order to facilitate the of use Mecke's multivariate equation (Proposition~\ref{OKLMecke}) we set
 \begin{equation}\label{lanotat}
\varsigma=\sigma_1 \cap \sigma_2 \quad \text{ and } \quad  \varsigma_i=\sigma_i\, \setminus\, \varsigma \ \text{ for } i=1,2.
 \end{equation} 
 By removing the constraint that $u\leq q$ on  $\varsigma_i$, $i=1,2$, we obtain
\begin{equation}\label{abovz}
    \bbE \left[  (\hat \cZ^{\go,a}_{\beta,q})^2 \right]\le 
\bbE \bigg[  \sum_{\varsigma_1, \varsigma_2, \varsigma\in  \cP(\go) \text{ disjoints}}   w(\varsigma_1\cup \varsigma)w(\varsigma_2\cup \varsigma) \ind_{\cA_q}(\varsigma) \bigg] \, .
\end{equation}
Now we can apply 
Proposition \ref{OKLMecke}. To do so, we split the sum according to the cardinality of $\varsigma$ ($=\{ (t_i,x_i,u_i)\}^m_{i=1}$), and also according to the number of points in $\varsigma_1$ and $\varsigma_2$
in each of the intervals $(t_{i-1},t_i)$, $i\in\lint 1, m+1\rint$ ($t_0=0$ $t_{m+1}=1$).
After factorizing
we obtain that the r.h.s.\ in \eqref{abovz} is equal to
\begin{multline}\label{zruc}
\sum_{m \ge 0} \int_{0<t_1<\dots<t_m<1} \int_{ (\bbR^d)^m} \int_{[a,q)^m} \gb^{m} 
  \prod_{i=1}^{m} u_i^2 \,  z_{\beta,a}\big( (t_{i-1},x_{i-1}), (t_{i},x_{i}) \big)^2  
  \\ \Big( \int_{\bbR^d}   z_{\beta,a}\big( (t_m,x_m), (1,x) \big) \dd x  \Big)^2  \prod_{i=1}^m \dd t_i \dd x_i  \lambda(\dd u_i) ,
\end{multline}
where $z_{\beta,a}\big( (t,x), (t',x') \big)$ is the expected value of the point-to-point partition function for the polymer in the environment $\go^{(a)}$. With the convention $s_0=t$, $s_{\ell+1}=t'$ and $y_0=x$, $y_{\ell+1} =x'$, it is given by
\begin{multline}\label{zruc2}
 z_{\beta,a} ( (t,x), (t',x'))
 :=   e^{ -\gb \kappa_a (t'-t)}
\bigg[ \rho_{t'-t}(x'-x)\\+ \sum_{\ell=1}^{\infty}  \beta^{\ell} \int_{ t< s_1<\dots<s_{\ell}<t' }  \int_{(\bbR^d)^{\ell}} \int_{[a,\infty)^\ell}  \prod_{j=1}^{\ell+1} \rho_{s_j-s_{j-1}} (y_j-y_{j-1}) \prod_{j=1}^{\ell} 
\dd s_j \dd y_j   v_j \lambda(\dd v_j) \bigg].
\end{multline}
To see that \eqref{zruc} holds, observe that expanding all the products we obtain a sum over the indices $m$ (standing for the number of points in $\varsigma$) and $\ell^{(1)}_i$, $\ell^{(2)}_i$ (we need two indices to expand $z^2_{\beta,a}$) which stands for 
the number of points of $\varsigma_1$ and $\varsigma_2$ in the time interval $(t_{i-1},t_{i})$
(the term $\rho_{t'-t}(x'-x)$ in \eqref{zruc2} corresponding to $\ell=0$).
The expression of  $z_{\beta,a} ( (t,x), (t',x'))$ simplifies after integration over all intermediate variables
\begin{multline}
 z_{\beta,a} ( (t,x), (t',x')) \\= e^{ -\gb \kappa_a (t'-t)} \rho_{t'-t}(x'-x) \sum_{\ell=0}^{\infty}  \frac{[\gb (\kappa_a +\mu) (t'-t)]^{\ell}}{ \ell !}=  e^{ \gb \mu (t'-t)}
\rho_{t'-t}(x'-x).
\end{multline}
Reinjecting this into \eqref{zruc} and performing the integral over $u_i\in (0,q)$ instead of $[a,q)$ ---~this yields an upper bound which is uniform in $a$ ---,  we obtain that 
\begin{equation}
\label{eq:secondmomentZq}
     \bbE \left[  (\hat \cZ^{\go,a}_{\beta,q})^2 \right]
     \le \sum_{m \ge 0} \left( \gb  e^{2\gb \mu} V_q \right)^{m}  \int_{ \mathfrak{X}^m \times (\bbR^d)^m}
   \varrho(\bt,\bx)^2  
     \dd \bx \dd \bt \, ,
\end{equation}
where we also used that $(\int_{\bbR^d} \rho_{1-t_m}(x-x_m) \dd x)^2=1$, and defined $V_q := \int_{(0,q)} \ups^2 \lambda(\dd \ups)  <\infty$ (recall~\eqref{hyposmall}).
Now using the definition \eqref{defrhotx} of $\rho_t(x)$, we have
\begin{equation}\label{lekalek}
  \int_{\bbR^d}(\rho_{t}(x) )^2\dd x
= 2^{-d}(\pi t)^{-d/2}\, .
\end{equation}
Hence, in dimension $d=1$ we get that
\begin{equation}\label{bigsum}
      \bbE \left[  (\hat \cZ^{\go,a}_{\beta,q})^2 \right]\le \sum_{m \ge 0} \left( 2^{-1}\pi^{-1/2} \beta e^{2\gb \mu}  V_q \right)^{m}  \int_{\mathfrak{X}^m} \prod_{i=1}^m
      \frac{ \dd t_i}{\sqrt {t_i-t_{i-1}}} \, .
\end{equation}
To conclude, notice that (with the convention $t_{m+1}=1$)
\[
a_m= \int_{\mathfrak{X}^m} \prod_{i=1}^m
      \frac{ \dd t_i}{\sqrt {t_i-t_{i-1}}}
      \leq a_m' := \int_{\mathfrak{X}^m} 
      \prod_{i=1}^{m+1}
      \frac{ \dd t_i}{\sqrt {t_i-t_{i-1}}} 
= \frac{\gG(1/2)^{m+1} }{\gG( (m+1)/2)} 
\]
where the last identity is a standard calculation, see e.g.~\cite[Lemma A.3]{BL20_disc}.
Hence $(a_m)_{m\geq 0}$ decays super-exponentially,
 so the r.h.s.\ of \eqref{bigsum} is finite for every value of~$q$.
\end{proof}

\begin{rem}
\label{rem:p2p1}
In the case of the point-to-point partition function
$\cZ_{\gb}^{\go,a}(t,x)$,
one uses $X_a^{(q)} = \cZ_{\gb}^{\go,[a,b)}(t,x)$ instead of $Z_{\gb}^{\go, [a,b) }$ and  the representation~\eqref{altexpp2p} instead of~\eqref{altexp} to compute the first an second moment.
The proof is carried out in an identical manner as above, 
replacing $\varrho(\bt,\bx)$ with $\varrho(\bt,\bx)\rho_{t-t_k}(x-x_k)$;
whose integral on $(\bbR^d)^k$ is $\rho_t(x)$.
 The main difference is in~\eqref{zruc} where the variable $x$ in $z_{\beta,a}((t_m,x_m),(1,x))$ is
no longer integrated, 
which leads to having  an extra $\rho_{t-t_m}(x-x_m)^2$ in the last integral in~\eqref{eq:secondmomentZq}.
An easy induction on $m$ yields that
\begin{equation}\label{inteinte}
\int_{(\bbR^d)^m}
   \varrho(\bt,\bx)^2 \rho_{t-t_m}(x-x_m)^2 \dd x=  
      \frac{e^{-\|x\|^2 / t}}{  \pi^{\frac{m+2}{2}}  \prod_{i=1}^{m+1} \sqrt {t_i-t_{i-1}}}
      \end{equation}
This leads us to having $ a_m'$ instead of $a_m$ in the series~\eqref{bigsum}  (with an extra $t^m$ if $t\neq 1$) and does not change the conclusion.
\end{rem}

\subsection{Proof of Proposition~\ref{thm:continuous2} in dimension $d\geq 2$}
\label{sec:contd>2}

The proof of the previous section cannot apply to higher dimension. From a purely technical point of view the reason for this is that the r.h.s.\ of \eqref{lekalek} is not integrable in $t$ when $d\ge 2$, making the r.h.s.\ of \eqref{bigsum} infinite.

To circumvent this problem we need to refine our selection of trajectories. As the divergence in \eqref{lekalek} comes from small values of $t$, we want to add a restriction that forbids favorable sites to have an abnormaly high concentration in a small time frame. Our selection scheme presents some formal analogy and was inspired by the multibody techniques used in \cite{BL18, GLT11} in the very different context of  disordered pinning models at marginality.

Fine tuning our parameters, under the assumption~\eqref{hyposmall} we find a restriction of the trajectories based on this idea which allows to apply Proposition \ref{metaprop}.

\begin{rem}\label{chongsremark}
An alternative proof of the uniform integrability of $\cZ^{\go,a}_{\beta}$ for $d\ge 2$ was brought to our attention by C.\ Chong. 
Once \eqref{complea} has been proven, it is sufficient to show that 
\begin{equation}\label{lpbounded}
\sup_{a\in(0,1]}\bbE\left[(\cZ^{\go,[a,q)}_{\beta})^p\right]<\infty \, ,
\end{equation}
for some $p>1$ (the conclusion of Proposition~\ref{metaprop} remains valid, using H\"older's inequality instead of Cauchy--Schwarz).
This last bound can be extracted from  \cite[Theorem 1.3.1]{Loubert98} after observing that $\cZ^{\go,[a,q)}_{\beta}$ is the solution at time $1$ and coordinate $0$ of the Stochastic Heat Equation with initial condition $u_0\equiv 1$. To extend this argument to the point-to-point partition function, some more care is required since in this case one has to consider the solution of the SHE with Dirac initial condition (not treated in \cite{Loubert98}) but the argument should in principle also work.

However, our argument presents a few advantages.
Firstly, it does not rely on any tool of stochastic integration and only marginally on the properties of the heat-kernel: it is therefore easily adaptable to the context of other disordered systems  presented in Section~\ref{sec:comments}.
Also, our proof of Theorem \ref{thm:conv} in \cite{BL20_disc} relies on a similar strategy and we believe that the proof in the continuum setup (which is much simpler than that in the discrete one) could be of use for potential readers of  
\cite{BL20_disc}.
\end{rem}

\subsubsection{A finer restriction of the set of trajectories}

Let us now consider the  restriction of the partition function to ``good trajectories'' $\sigma$.
Thanks to assumption~\eqref{hyposmall}, we can fix some $p\in (1,1+\frac 2d)$ with $p<2$ such that $\int_{(0,1)} \ups^{p} \lambda(\dd \ups)<\infty$.
We then fix for the rest of this section a parameter $\gamma>0$ which satisfies
\begin{equation}
\label{lesdeuxconds}
\frac{d-2}{2(2-p)} < \gamma < \frac{1}{p-1} \qquad \big(i.e.\ \ \gamma(p-1)<1 \ \text{ and } \  \frac d2 -\gamma(2-p) <1 \big).
\end{equation}
The assumption $p \in (1, 1+\frac 2d)$ entails that $\gamma=d/2$ is always a valid choice, but
we prefer to write the two separate conditions we have on $\gamma$ to make the requirements more transparent.
Then, for any $q \geq 1$, we define $ \cB_q$ as 
\begin{equation}
\label{def:BNcont}
 \cB_q :=\bigg\{ \sigma \in \cP(\go) \ : \
 \forall \sigma'\subset \sigma, \  \prod_{j=1}^{|\sigma'|} u'_j < q^{|\sigma'|} \prod_{j=1}^{|\sigma'|} (t_j'-t_{j-1}')^{\gamma} \bigg\}
 \end{equation}
 where  $t'$ and $u'$ is used to denote the coordinates of time ordered points in $\sigma'$ (with $t'_0=0$ by convention).
We set in this section.
\begin{equation}
 \hat \cZ^{\go,a}_{\beta,q}:=  \sum_{\sigma \in  \cP(\go)}  w_{a,\beta}(\sigma) \ind_{ \cB_q}(\sigma).
\end{equation}
 Note that $\sigma\in \cB_q$ implies in particular that $u_i\le q$ for $i\in \lint 1, |\sigma| \rint$ and hence $\hat \cZ^{\go,a}_{\beta,q}\le \hat \cZ^{\go,[a,q)}_{\beta}$.
Now, Theorem \ref{thm:continuous} is a consequence of Propositions \ref{lesper} and \ref{llavar} below (which allow to control respectively the first and second moment of $\hat \cZ^{\go,a}_{\beta,q}$) and of Proposition \ref{metaprop}.

\begin{proposition}\label{lesper}
Assuming that  $\gamma < \frac{1}{p-1}$ we have
 \begin{equation}\label{unif1}
 \lim_{q \to \infty} \sup_{a\in(0,1]} \bbE \Big[   \big| \cZ^{\go,a}_{\beta}- \hat \cZ^{\go,a}_{\beta,q}   \big|  \Big]=0 \, .
\end{equation}
\end{proposition}

\begin{proposition}\label{llavar}
 Assuming that $\gamma>\frac{d-2}{2(2-p)}$,
for every $q\ge 1$ we have
\begin{equation}\label{unif2}
\sup_{a\in(0,1]} \bbE \Big[  \big( \hat \cZ^{\go,a}_{\beta,q} \big)^2 \Big]<\infty \, .
\end{equation}
\end{proposition}

The proof of Propositions \ref{lesper} and \ref{llavar} are technically more involved than that of Proposition \ref{d11}. This is in particular because the restriction $\cB_q$ produces an integral that does not factorize as well as the one obtained when only $\cA_q$ is considered. We first need to introduce some technical bounds on some type of multivariate integrals which appear in our first and second moment computations respectively.

\subsubsection{Technical preliminaries: an upper bound on multivariate integrals}

The following upper bounds are the key ingredients in  the proof of  Proposition \ref{lesper} and \ref{llavar} as they allow to control the multivariate integrals produced by the application of Mecke's multivariate equation (Proposition \ref{OKLMecke}).

\begin{lemma}
\label{lem:tailzzz1}
Assume that  $\int_{(0,1)} \ups^p \lambda(\dd \ups) <\infty$
for some $p \in (1,2)$ and also that $\int_{[1,\infty)} \ups \lambda(\dd \ups) <\infty$. Then for any $q\geq 1$ there is
a constant $\bc_q$, verifying $\lim_{q\to\infty} \bc_q =0$,
such that for every $m\geq 1$ and every $h\in (0,1)$, we have
\begin{align}
\label{firstline1}
\int_{(0,q)^m}  \ind_{\big\{\prod^m_{j=1}u_j \ge  h\, q^m\big\}} \prod_{i=1}^m u_i \lambda( \dd u_i ) \le   \frac{1}{p-1}
\frac{  (\bc_q)^m  h^{1-p} }{(m-1)!} \big[ \log (  2^m/h)\big]^{m-1} \, .
\end{align}
In particular, for any $0< \gep <  1$, we have
\begin{equation}
\label{secondline}
\int_{(0,q)^m}  \ind_{\big\{\prod^m_{j=1}u_j \ge  h \, q^m\big\}} \prod_{i=1}^m u_i \lambda( \dd u_i ) \le \frac{ (  2^{\gep}\bc_q /\gep)^{m}  }{p-1} \, h^{1-p-\gep}  \, .
\end{equation}
\end{lemma}

\begin{lemma}
\label{lem:tailzzz2}
Assume that $\int_{(0,1)} \ups^p \lambda(\dd \ups) <\infty$
for some $p \in (1,2)$. Then for any $q\geq 1$ there is
a constant $\bC_q$
such that for every $m\geq 1$ and every $h\in (0,1)$, we have
\begin{align}
\label{firstline2}
\int_{(0,q)^m}  \ind_{\big\{\prod^m_{j=1}u_j \le  h\, q^m\big\}} \prod_{i=1}^m u_i^2 \lambda( \dd u_i ) \le  
 h^{2-p} \, (\bC_q)^{m}\sum_{\ell=1}^{m}\frac{(2-p)^{-\ell}}{(m-\ell)!}  \big[\log (1/h)\big]^{m-\ell} \, .
\end{align}
In particular, for any $0< \gep < 2-p$, we have
\begin{equation}
 \label{thirdline}
 \int_{(0,q)^m}  \ind_{\big\{\prod^m_{j=1}u_j \le  h\, q^m\big\}} \prod_{i=1}^m u^2_i \lambda(\dd u_i ) \le \frac{ (\bC_q/\gep)^m  }{2- p-\gep} \,  h^{2-p-\gep}  \, .
\end{equation}
\end{lemma}

For both lemmas, the idea is to compare the integrals with some integrals  with respect to the Lebesgue measure: we postpone to the Appendix the proof of the two following claims (presented as  Propositions~\ref{prop:compare1} and~\ref{prop:compare2} respectively).

\smallskip
\noindent
\textit{{\bf Claim 1.}
Under the assumptions of Lemma~\ref{lem:tailzzz1},
there exists a constant $c_q$ verifying 
$\lim_{q\to\infty} c_q q^{1-p} =0$ such that}
\begin{equation}
\label{plusdidee}
\int_{(0,q)^m}  \ind_{\big\{\prod^m_{j=1}u_j \geq   h\, q^m\big\}} \prod_{i=1}^m u_i \lambda( \dd u_i ) \le   (c_q)^m \int_{(0,2q)^m}  \ind_{\big\{\prod^m_{j=1}u_j \ge  h\, q^m\big\}} \prod_{i=1}^m u_i^{-p} \dd u_i  \, .
\end{equation}

\smallskip
\noindent
\textit{{\bf Claim 2.}
Under the assumptions of Lemma~\ref{lem:tailzzz2},
there exists a constant $C_q$ such that}
\begin{equation}
\label{plusdidee2}
\int_{(0,q)^m}  \ind_{\big\{\prod^m_{j=1}u_j \le  h\, q^m\big\}} \prod_{i=1}^m u_i^2 \lambda( \dd u_i ) \le  (C_q)^m
\int_{(0,q)^m}  \ind_{\big\{\prod^m_{j=1} u_j \le  h\, q^m\big\}} \prod_{i=1}^m u_i^{1-p} \dd u_i  \, .
\end{equation}

\begin{proof}[Proof of Lemma~\ref{lem:tailzzz1}]
Thanks to~\eqref{plusdidee}, 
we only have to prove that for any $h\in (0,1)$
\begin{equation}
\label{fortailzzz1}
\int_{(0,2q)^m}  \ind_{\big\{\prod^m_{j=1}u_j \ge  h\, q^m\big\}} \prod_{i=1}^m u_i^{-p} \dd u_i  
\leq  \frac{h^{1-p}}{p-1}
\frac{    q^{(1-p)m}   }{(m-1)!} \big[ \log (  2^m/h)\big]^{m-1} 
\end{equation}
and then set $\bc_q :=  c_q \, q^{1-p}$ (which verifies $\lim_{q\to \infty} \bc_q =0$).
First of all, notice that by a change of variable $v=u/2q$ it is sufficient to prove~\eqref{fortailzzz1} only in the case  $q=1/2$.
We set, for all $h\in (0,1)$,
\[
p_k(h):= (p-1) h^{p-1} \int_{(0,1)^k}  
\ind_{\big\{   \prod^k_{j=1}u_j \ge  h\big\}} \prod_{i=1}^k u^{-p}_i \dd u_i \, ,
\]
 so to obtain~\eqref{fortailzzz1} with $q=1/2$ we need to show that for $k\geq 1$
\begin{equation}
\label{fortailzzzzzzzzz}
 p_k(h)\le \frac{1}{(k-1)!}(\log (1/h))^{k-1}  \qquad \forall\, h \in (0,1)
 \end{equation}
and apply it to $2^{-m}h$ instead of $h$.
By a direct calculation, we have  $p_1(h) = 1-h^{p-1} \leq 1$ for all $h\in(0,1)$, which gives the result for $k=1$.
Then we can proceed by induction. Integrating with respect to the value of $u_1$ and using the change of variable $v=\log(u_1/h)$, we obtain
\begin{equation*}
p_{k+1}(h)= h^{p-1}\int_{h}^1 u_1^{-p} \left(\frac{h}{u_1}\right)^{1-p} p_{k}(h/u_1) \dd u_1
= \int_{0}^{\log (1/h)}  p_{k}(e^{-v}) \dd v.
\end{equation*}
From this and $p_1(h)\le 1$  we easily obtain~\eqref{fortailzzzzzzzzz} by induction.

To obtain \eqref{secondline} from \eqref{firstline1},  
we just use that $e^{\gep t} \geq  \frac{1}{(m-1)!} (\gep t)^{m-1}$  for any $t>0$ and any $m\geq 1$,   by a Taylor expansion: applying this to $t= \log(2^m/h)$, we get the bound~\eqref{secondline}.
\end{proof}

\begin{proof}[Proof of Lemma~\ref{lem:tailzzz2}]
Similarly to Lemma~\ref{lem:tailzzz1},
thanks to \eqref{plusdidee2} we only have to prove
\begin{equation}
\label{fortailzzz2}
\int_{(0,q)^m}  \ind_{\big\{\prod^m_{j=1}u_j \le  h\, q^m\big\}} \prod_{i=1}^m u^{1-p}_i \dd u_i \le   h^{2-p} \, q^{(2-p)m}\sum_{\ell=1}^{m}\frac{(2-p)^{-\ell}}{(m-\ell)!}  \big[\log (1/h)\big]^{m-\ell} 
\end{equation}
 and then set $\bC_q := q^{2-p}$.
Again, by a change of variable it is sufficient to prove~\eqref{fortailzzz2} only in the case $q=1$.
We set, for all $h\in (0,1)$,
\begin{equation*}
\tilde p_k(h):= h^{p-2}\int_{(0,1)^k} 
 \ind_{\big\{\prod^k_{j=1}u_j \le  h\big\}} \prod_{i=1}^k u^{1-p}_i \dd u_i ,
\end{equation*}
and a direct calculation gives that $\tilde p_1(h)=(2-p)^{-1}$,
which yields the result for $m=1$. 
For the induction step, decomposing according  to whether $u_1 \leq h$ or $u_1>h$, we get that 
\begin{align*}
\tilde p_{k+1}(h) &= h^{p-2} \left[ \int^{h}_0 u_1^{1-p}\dd u_1 \left(\int_{0}^1 v^{1-p}\dd v\right)^k +  \int^1_{h} u_1^{1-p}\left(\frac{h}{u_1} \right)^{2-p} \tilde p_k(h/u_1) \dd u_1  \right] \\
& = \left(\frac{1}{2-p}\right)^{k+1}+  \int_0^{\log(1/h)}   \tilde p_k(e^{-v}) \dd v,
\end{align*}
where we also used a change of variable $v=\log(u_1/h)$ for the last identity.
From this we easily obtain by induction
that
\begin{equation*}
\tilde p_k(h) \le \sum_{\ell=1}^{k} \frac{(2-p)^{-\ell}}{(k-\ell)!} (\log (1/h))^{k-\ell} \, ,
\end{equation*}
which is the desired result.
Now, to obtain~\eqref{thirdline} from~\eqref{firstline2}, we use the inequality  $e^{\gep t} \geq \frac{(\gep t)^{m-\ell}}{(m-\ell)!}$ with $t=\log (1/h)$ to get that for any $0<\gep<2-p$
\begin{equation}
\sum_{\ell=1}^{m}\frac{(2-p)^{-\ell}}{(m-\ell)!}  \big[\log (1/h)\big]^{m-\ell} \le \sum_{\ell=1}^{ m} \left(\frac{\gep}{2-p}\right)^{\ell} \gep^{ -m} h^{-\gep}  \le  \frac{\gep}{ 2-p-\gep}\,  \gep^{ -m} h^{-\gep} .
\end{equation}
This concludes the proof of \eqref{thirdline}.
\end{proof}

\subsubsection{Proof of Proposition \ref{lesper}}

Note that as we have already proven \eqref{complea} and since $\cZ^{\go,[a,q)}_{\beta} \geq \hat \cZ^{\go,a}_{\beta,q}$, it is sufficient to prove that 
 \begin{equation}
 \lim_{q \to \infty} \sup_{a\in(0,1]} \bbE\Big[  \cZ^{\go,[a,q)}_{\beta} -\hat \cZ^{\go,a}_{\beta,q} \Big]=0
\end{equation}
 Decomposing over the cardinality of $\sigma$, using Proposition \ref{OKLMecke} as in \eqref{originn} and integrating over the space variable (recall \eqref{denzit}) we obtain
\begin{equation}\label{bbq}
  \bbE\Big[  \cZ^{\go,[a,q)}_{\beta} -\hat \cZ^{\go,a}_{\beta,q} \Big]
= e^{- \gb \kappa_a} \sum_{k\ge 0} \beta^k \int_{\mathfrak{X}^k \times [a,q)^k} \ind_{\cB^{\cc}_q}(\bt,\bu)  \prod_{i=1}^k    u_i \lambda( \dd u_i )  \dd t_i \, ,
\end{equation}
 where we recall that $\cB_q$ has been defined in~\eqref{def:BNcont}.
Here,  with some abuse of notation, we identified $\cB^{\cc}_q$ and its image by the projection $(\bt,\bx,\bu)\mapsto (\bt,\bu)$; 
note that $\cB_q$ does not involve the space variable.
To estimate the above integral, we use a union bound for  $\ind_{(\cB_q)^{\complement}}(\bt,\bu)$. When the value of $k$ is fixed we have
\begin{equation*}
 \ind_{(\cB_q)^{\complement}}(\bt,\bu)
\le  \sum_{m=1}^{k} \sum_{1\leq i_1<\dots <i_m \leq k } \ind_{\left\{  \prod_{\ell=1}^{m} u_{i_\ell} \ge q^{m} \prod_{\ell=1}^{m} (t_{i_\ell}-t_{i_{\ell-1}})^{\gamma}    \right\} }.
\end{equation*}
With this done, we can perform first the integral with respect to $u_j$ and $t_j$ for $j\notin \{i_1, \ldots, i_m\}$:
summing over the number of points $k_i$ that can be fitted between two points $i_\ell< i_{\ell+1}$, we obtain after factorization that the r.h.s.\ in \eqref{bbq}
is smaller than 
\begin{multline}
 e^{- \beta  \kappa_a }  \sum_{m\ge 1} \int_{0<t_1<\dots<t_m<1}  \int_{[a,q)^m}
 \beta^{m} \ind_{\big\{  \prod_{i=1}^{m} u_{i} \ge q^m \prod_{i=1}^{m} (t_{i}-t_{i-1})^{\gamma}\big\}}  \label{eq:resteZ2a}\\
  \times \prod_{i=0}^m \left(  \sum_{k_i=0}^\infty  \beta^{k_i}  \int_{ t_{i}< t^{(i)}_1<\dots<t^{(i)}_{k_i}< t_{i+1}   } \int_{[a,q)^{k_i}}
 \prod_{j=1}^{k_i} 
u^{(i)}_j \lambda( \dd u^{(i)}_j )\dd t^{(i)}_j \right)
\prod_{i=1}^m u_i \lambda(\dd u_i) \dd t_i  \,,
\end{multline}
where we used the convention $t_0=0$ and $t_{m+1}=1$.
Now we can compute explicitly each term of the product in the second line above  
(as in \eqref{kalkul}).
Replacing $q$ by $\infty$ in the domain of integration, which yields an upper bound and makes the computation simpler,
we get, recalling that $\int_{[a,\infty)}  u\lambda(\dd u) = \kappa_a+\mu$,
\begin{equation*}
\sum_{k_i=0}^\infty  \beta^{k_i} \int_{ t_{i}< t^{(i)}_1<\dots<t^{(i)}_{k_i}< t_{i+1}   }\int_{[a,\infty)^{k_i}}
 \prod_{j=1}^{k_i} 
u^{(i)}_j \lambda( \dd u^{(i)}_j ) \dd t^{(i)}_j 
= e^{ \beta (\kappa_a+\mu) (t_{i+1}-t_i)}.
\end{equation*}
 The product of these terms
gives a factor $e^{\gb(\kappa_a+\mu)}$
and we therefore get the inequality
\begin{multline}\label{trukkz}
  \bbE\Big[  \cZ^{\go,[a,q)}_{\beta} -\hat \cZ^{\go,a}_{\beta,q} \Big]\\
  \le e^{\gb \mu}  \sum_{m\ge 1} \int_{\mathfrak{X}^m\times [a,q)^m}
\beta^{m} \ind_{\big\{  \prod_{i=1}^{m} u_{i} \ge q^m \prod_{i=1}^{m} (t_{i}-t_{i-1})^{\gamma}\big\}} 
\prod_{i=1}^m u_i \lambda(\dd u_i) \dd t_i  \,,
\end{multline}
The r.h.s.\ in \eqref{trukkz}  can be bounded above using Lemma~\ref{lem:tailzzz1}. More specifically we use \eqref{secondline} to bound the integral over $u_i$, setting $a$ to $0$ to obtain an upper bound that does not depend on~$a$. We fix $\gep$ small enough so that $\gamma(p+\gep-1)<1$ (recall that we have $\gamma(p-1)<1$ by assumption) and by~\eqref{secondline} we obtain 
\begin{equation}
\label{avantgamma}
\int_{[a,q)^m} \ind_{\big\{  \prod_{i=1}^{m} u_{i} \ge q^m \prod_{i=1}^{m} (t_{i}-t_{i-1})^{\gamma}\big\}} 
\prod_{i=1}^m u_i \lambda (\dd u_i )  \leq \frac{ (\bc_q/\gep)^m}{p-1}    \prod_{i=1}^{m} (t_{i}-t_{i-1})^{\gamma(1-p-\gep)} \, .
\end{equation}
We therefore obtain that \eqref{trukkz} is smaller than

\begin{multline}
 \frac{e^{\gb\mu}}{p-1} \sum_{m\ge 1}  \big( \beta  \bc_q/\gep \big)^{m} 
\int_{\mathfrak{X}^m}  \prod_{i=1}^m (t_{i}-t_{i-1})^{\gamma(1-p-\gep)}\dd t_i \\
 =   \frac{e^{\gb\mu}}{p-1}  \sum_{m\ge 1} \big( \beta  \bc_q/\gep \big)^{m}    \frac{\Gamma(1-\gamma(p+\gep-1))^m}{ \gG( m(1-\gamma(p+\gep-1))+1)} \, ,
\label{lafonctiongamma}
\end{multline}
where in the last equality we used that $\gamma(p+\gep-1)<1$.
The sum in the r.h.s.\ of \eqref{lafonctiongamma} is finite for any value of $q\ge 1$, and can be made arbitrarily small by choosing $q$ large (with $\gep$ and $\gamma$ fixed),
since Lemma~\ref{lem:tailzzz1} ensures that the constant $\bc_q$
goes to $0$ as $q\to\infty$.
\qed

\subsubsection{Proof of Proposition \ref{llavar}}
 
 We have 
 \begin{equation}
 \label{eq:Zcarre}
    \bbE \Big[  \big(\hat \cZ^{\go,a}_{\beta,q} \big)^2 \Big]=\bbE \bigg[  \sum_{\sigma_1, \sigma_2 \in  \cP(\go)}  w_{a,\beta}(\sigma_1)w_{a,\beta}(\sigma_2) \ind_{\cB_q}(\sigma_1)\ind_{\cB_q}(\sigma_2) \bigg].
 \end{equation}
We  use again the notation $\varsigma= \sigma_1\cap \sigma_2$ and $\varsigma_i=\sigma_i\setminus \varsigma$. Let us relax the condition $\sigma_1,\sigma_2\in \cB_q$ to obtain something which is easier to handle in the computation.  Formally the divergence of the second moment of the (unrestricted) partition is obtained when integrating the contribution of the environment at the points in the replica intersection $\sigma_1\cap \sigma_2$, hence we should not be losing too much if we restrict our constraint to $\varsigma$.
With this in mind, we set
\[
\cD_q = \Big\{ \varsigma = (t_i,x_i,u_i)_{i=1}^{|\varsigma|}  \, \colon \max_{1\leq i\leq |\varsigma|}u_i <  q  \ , \   \prod_{i=1}^{|\varsigma|} u_i \leq  q^{|\varsigma|}\prod_{i=1}^{|\varsigma|} (t_{i}-t_{i-1})^{\gamma} \Big\} \, .
\]
and observe that 
\begin{equation}
\bbE \left[  (\hat \cZ^{\go,a}_{\beta,q})^2 \right]\le\bbE \bigg[  \sum_{\varsigma_1, \varsigma_2,\varsigma \in  \cP(\go) \text{ disjoints}}  w_{a,\beta}(\varsigma_1\cup \varsigma)w_{a,\beta}(\varsigma_2\cup\varsigma) \ind_{\cD_q}(\varsigma) \bigg].
\end{equation}
Using again Proposition \ref{OKLMecke} as in \eqref{zruc} and using \eqref{zruc2} 
to integrate over all the variables associated with $\xi_1$ and $\xi_2$ we obtain 
\begin{multline}
\label{avantlasdx}
\bbE \bigg[  \sum_{\varsigma_1, \varsigma_2,\varsigma \in  \cP(\go) \text{ disjoints}}  w_{a,\beta}(\varsigma_1\cup \varsigma)w_{a,\beta}(\varsigma_2\cup\varsigma) \ind_{\cD_q}(\varsigma) \bigg]\\
=   \sum_{m \ge 0}  (\beta e^{2\gb \mu} )^{m}  \int_{\mathfrak{X}^m \times (\bbR^d)^m \times [a,\infty)^m} \ind_{\cD_q}(\bt,\bx,\bu)
  \varrho(\bt,\bx)^2   \prod_{i=1}^m
    u_i^2 \lambda(\dd u_i)  \dd t_i  \dd x_i  \, .
 \end{multline}
Now integrating over $\bx$ and using \eqref{lekalek}, we obtain that the r.h.s.\ above is equal to 
\begin{equation}\label{lasdx}
  \sum_{m \ge 0} \left(  \frac{\gb e^{2\gb \mu}}{ 2^d\pi^{d/2}}  \right)^{m}  \int_{ \mathfrak{X}^m \times [a,q)^m}  \ind_{\big\{\prod_{i=1}^m u_i\le  q^m \prod_{i=1}^m (t_i-t_{i-1})^{\gamma}\big\}}\prod_{i=1}^m
     \frac{\dd t_i}{(t_i-t_{i-1})^{d/2}}  u^2_i \lambda(\dd u_i)\, .
\end{equation}

To estimate the integral over $u_1, \ldots, u_m$, we use  \eqref{thirdline}  in Lemma~\ref{lem:tailzzz2}. We integrate over $(0,q)$ to get an upper bound which is uniform in $a$.
We fix $\gep$ such that $\frac d2- (2-p-\gep)\gamma<1$ (recall that $\frac d2- (2-p)\gamma<1$ by assumption) and by~\eqref{thirdline} we obtain 
\[
\int_{[a,q)^m} \ind_{\big\{\prod_{i=1}^m u_i\le  q^m \prod_{i=1}^m (t_i-t_{i-1})^{\gamma}\big\}}\prod_{i=1}^m u^2_i \lambda( \dd u_i )  
\le  \frac{ (\bC_q/\gep)^m }{2- p-\gep}  \prod_{i=1}^m (t_i-t_{i-1})^{  \gamma  (2-p-\gep) } \, .
\]
Reinjected in \eqref{lasdx}, this yields
\begin{equation}
\label{secondmomentlabel}
    \bbE \Big[  \big( \hat \cZ^{\go,a}_{\beta,q} \big)^2 \Big]
\le \sum_{m \ge 0} \frac{ \big(  2^{-d} \pi^{-d/2} \beta  e^{\gb\mu}   \bC_q/\gep\big)^m }{2-p-\gep} \int_{\mathfrak{X}^m} \prod_{i=1}^m 
     \frac{\dd t_i}{(t_i-t_{i-1})^{ \frac d2- (2-p- \gep)\gamma} } \, .
     \end{equation}
To conclude we just need to show that the above sum is finite. 
To check this, we simply observe that, thanks to the fact that $\frac{d}{2} - \gamma(2-p -\gep) <1$, the integral in $\bt$  is equal to 
\[
\frac{\gG(1- (2-p-\gep)\gamma+d/2)^m}{\Gamma(m[1- (2-p-\gep) \gamma+d/2]+1)} \, ,
\]
and that the corresponding series in $m$ has an infinite radius of convergence.
\qed

\begin{rem}
For the proof of Propositions~\ref{lesper}-\ref{llavar}
in the case of the point-to-point partition function $\cZ_{\gb}^{\go,a} (t,x)$,
we need to slightly change the definition
of $\cB_q$ to take care of the end point, setting 
\[
\cB_q:=\Big\{ \sigma \in \cP(\go) \ : 
 \forall \sigma'\subset \sigma, \  \prod_{j=1}^{|\sigma'|} u'_j < q^{|\sigma'|}\prod_{j=1}^{|\sigma'|+1} (t'_j-t'_{j-1})^{\gamma} \Big\},
\]
with the convention that $t'_{|\sigma'|+1} =t$.
For the proof of Proposition~\ref{lesper},
there is an additional term $\rho_t(x)$
in~\eqref{bbq}, coming from the integration of
$\varrho(\bt,\bx) \rho_{t-t_k}(x-x_k)$ over the space variable, but the main
difference comes in~\eqref{avantgamma} when applying Lemma~\ref{lem:tailzzz1}.
The computation in~\eqref{lafonctiongamma} is different
(we have the integral of $\prod_{i=1}^{m+1} (t_i-t_{i-1})^{\gamma(1-p-\gep)}$ over~$\mathfrak{X}^m$
after scaling by $t$ if $t\neq 1$)
but the conclusion is identical.
For the proof of Proposition~\ref{llavar},
there is an additional term $\rho_{t-t_m}(x-x_m)^2$
in~\eqref{avantlasdx}. We proceed as in \eqref{inteinte} when integrating on $x_1,\dots,x_m$
and this yields only an extra multiplicative term $C e^{- \|x\|^2/t} (t-t_m)^{-d/2}$.
Then, the integral in~\eqref{secondmomentlabel}
is different
(we have the integral of $\prod_{i=1}^{m+1} (t_i-t_{i-1})^{ -\frac d2 +\gamma(p-2+\gep)}$ over~$\mathfrak{X}^m$) but
this does not change the conclusion.
\end{rem}

\subsection{Finiteness of partition functions}
\label{sec:generalnoiseproof}

We are going to prove here simultaneously Proposition \ref{prop:Zfinite}, the first part of Proposition \ref{th:superlem}, \eqref{lafinite} and Proposition \ref{lem:SHE}.
Note that the fact that $\cZ^{\go,a}_{\beta}$ and other partition functions are positive is a direct consequence of the rewriting given in  \eqref{altexp}. It remains to prove that under assumption~\eqref{hypolarge} they are almost surely finite. 
This is the following statement.
\begin{proposition}\label{polopopo}
If \eqref{hypolarge} is satisfied, then for every $u_0$ satisfying  \eqref{hypou0}  (with $T=1$) we have  for any $t\in [0,1]$, almost surely
 \begin{equation}
  \int_{\bbR^d} \cZ^{\go,a}_{\beta}(t,x)|u_0|(\dd x)<\infty \, .
 \end{equation}
In particular the cases $u_0=\delta_x$ and $u_0(\dd x)=\dd x$ respectively give
\begin{equation}\label{deuduncou}
 \cZ^{\go,a}_{\beta}(t,x)<\infty \quad \text{ and } \quad  \cZ^{\go,a}_{\beta}<\infty.
\end{equation}
\end{proposition}

\begin{rem}\label{proprop}
Proposition \ref{prop:Zfinite} and  \eqref{lafinite} are direct consequences of \eqref{deuduncou}.
Proposition~\ref{lem:SHE} also follows by observing that by time reversal
and translation invariance we have the following identity in distribution
 $$ \int_{\bbR^d} \cZ^{\go,a}_{\beta}[(0,y),(t,x)]u_0(\dd y) \stackrel{(d)}{=} \int_{\bbR^d} \cZ^{\go,a}_{\beta}(t,y-x)u_0(\dd y) $$
  and thus we just need to apply the result to $u_0$ translated by $x$.
\end{rem}

\begin{proof}[Proof of Proposition \ref{polopopo}]
For $\sigma = (t_i,x_i,u_i)_{i=1}^{|\sigma|} \in \cP(\go)$
 (with $t_k\leq 1$),
we define 
\begin{equation}
\label{ent-ener}
 H(\sigma):= \sum_{i=1}^{|\sigma|} \frac{\|x_i-x_{i-1}\|^2}{ t_i-t_{i-1} } 
 \quad \text{ and } \quad     G(\sigma):= \sum_{i=1}^{|\sigma|} \log (u_i) \, .
\end{equation}
The quantity $H(\sigma)$ and $G(\sigma)$ corresponds roughly to the cost and gains
at the exponential level to visit all the points in $\sigma$. We refer to $H(\sigma)$ as the entropy of the path.
Our statement is an almost direct consequence of the following lemma.
 \begin{lemma}\label{entross}
If \eqref{hypolarge} holds, for any fixed  $\gep>0$, we have almost surely
\begin{equation}
\cT(\go) := \sup_{\sigma\in \cP(\go)} \big\{  G(\sigma) - \tfrac \gep  2   H(\sigma) \big\} <\infty.
\end{equation}
 \end{lemma}
 \noindent
We postpone the proof of this lemma and first deduce the proposition from it.

 \smallskip
 
Using the representation \eqref{altexp} of the partition function (recall that $\go^{(a)}$ denote the set of points in the environment with  jump size larger than $a$)
and applying Lemma~\ref{entross},
we can write for $t\le 1$
\begin{align*}
e^{\beta \kappa_a}\cZ^{\go,a}_{\beta}&(t,x)=\sum_{\sigma\in \cP_{[0,t]}(\go^{(a)})}\beta^{|\sigma|} e^{G(\sigma)} \varrho(\bt,\bx)\rho_{t-t_k}(x-x_k)\\
&\le  \sum_{\cP_{[0,t]}(\go^{(a)})}\beta^{|\sigma|} e^{\cT+ \frac \gep 2 H(\sigma)} \prod_{i=1}^{|\sigma|} \varrho(\bt,\bx)\rho_{t-t_k}(x-x_k)\\
&= e^{\cT} \sum_{\sigma\in \cP_{[0,t]}(\go^{(a)})} \Big( \frac{\beta}{(1-\gep)^{d/2}} \Big)^{|\sigma|} \rho_{t-t_k}(x-x_k)  \prod_{i=1}^{ |\sigma|} \frac{e^{- (1-\gep)\frac{\|x_i-x_{i-1} \|^2}{2\left(t_i-t_{i-1}\right)}}}{ \big( 2\pi (t_i-t_{i-1}) /(1-\gep)\big)^{d/2}}  \, ,
 \end{align*}
 so that setting   $\vartheta =\vartheta(\gep): = (1-\gep)^{-1}$  and assuming that  $\gep<1/2$ 
 \begin{equation}
e^{-\cT} \cZ^{\go,a}_{\beta}(t,x)  \leq   e^{ -\gb \kappa_a} \sum_{\sigma\in \cP_{[0,t]}( \go^{(a)})}( 2^{d/2}\beta)^{|\sigma|} \varrho(  \vartheta  \bt,\bx) \vartheta^{d/2} \rho_{  \vartheta  (t-t_k)}(x-x_k)\,.
 \end{equation}
By Mecke's formula (Proposition~\ref{OKLMecke}) we conclude 
that
\begin{equation}
\label{lastfinite}
 \bbE\left[e^{-\cT(\go)}\int_{\bbR^d} \cZ^{\go, a}_{\beta}(t,x) u_0(\dd x)\right]\le   2^{d/2} e^{\beta [ 2^{d/2} \gl([a,\infty))-\kappa_a ]}\int_{\bbR^d } \rho_{  \vartheta t}(x) u_0(\dd x) \, .
\end{equation}
Using assumption \eqref{hypou0} on $u_0$ and
fixing $\gep$ sufficiently small  (\textit{i.e.}\ $\vartheta$ close enough to $1$) so that $\rho_{ \vartheta t}(x)$
is integrable w.r.t.\ $|u_0|$ (recall $t\leq T =1$),
we get that~\eqref{lastfinite} is finite.
This proves that  $\int_{\bbR^d} \cZ^{\go,a}_{\beta}(t,x)u_0(\dd x)<\infty$ almost surely  thanks to Lemma~\ref{entross}.
\end{proof}

\begin{proof}[Proof of Lemma \ref{entross}]
First of all, notice that if $\sigma = (t_i,x_i,u_i)_{i=1}^{|\sigma|}$
has a point with $u_i < 1$, 
then by removing this point from $\sigma$
we obtain a set $\sigma'$ with (strictly) smaller entropy $H(\sigma') < H(\sigma)$ and (strictly) higher energy $G(\sigma') > G(\sigma)$.
In the supremum, we therefore can restrict ourselves to points
 $(t,x,\ups) \in \go$ with $\ups\geq 1$. 
Let us now separate points according to  the size of their jump.
For each $k\ge 1$ we define
\[
\go_k := \big\{   (t,x, \ups)\in \go  \colon   t\in[0,1]  \text{ and } \log \ups \in [ e^{k-1} , e^{k}  )  \big\} \, ,
\]
and we let $\pi(\go_k)$ be its projection on the first two coordinates.
Note that the $\pi(\go_k)$'s are independent Poisson processes on  $[0,1]\times \bbR^d$ with respective intensity $\lambda_k \dd t\otimes \dd x$, where
\[
\gl_k:= \gl\big( \big[\exp(e^{k-1}), \exp(e^k) \big) \big) \, .
\]
One can then easily see that our assumption \eqref{hypolarge}
is equivalent to 
having
\begin{equation}
\label{eq:condition}
 \sum_{k\ge 1} \gl_k e^{ \frac{dk}{2}}<\infty \, .
\end{equation}
Our proof is based on the following statement, proven below. 
\begin{lemma}\label{bcantelli}
Fix $\theta\in \left(1, 1+\frac{1}{d}\right)$ and let $K>0$ be arbitrary.
Let $A_{k,n}$ be the event that there exists a path of $n$ points in $\go_k$ whose entropy is smaller than  $K n^{\theta} e^{k}$,
\textit{i.e.}
\[
A_{k,n}  = \bigcup_{  \sigma \in \cP(\go_k),\, |\sigma|=n }  \big\{  H(\sigma)   \leq K n^{\theta} e^k \big\} \, .
\]
Then, assuming~\eqref{hypolarge}, we have $\sum_{k,n} \bbP(A_{k,n})<\infty$.
\end{lemma}
\noindent Given $\sigma$, we let $n_k(\sigma)$ denote the number of points that the path displays in~$\go_k$.
We have 
\begin{equation}
 G(\sigma)\le \sum_{k=1}^{\infty} e^{k} n_k(\sigma). 
\end{equation}
We let $k_0(\go)$, $n_0(\go)$ be such that $A_{k,n}^{\cc}$ holds for every $k\ge k_0$ (for every $n$), and for every $n\ge n_0$ (for every $k$).
For every $k$, we have
\begin{equation}
  n_k(\sigma)\le (K^{-1} H(\sigma)e^{-k})^{ 1/\theta} +n_0\ind_{\{k\le k_0\}} \,,
\end{equation}
 where we have  used that $H(\sigma) \geq H( \sigma_k)$, with $\sigma_k \in \cP(\go_k)$ the set obtained by removing all points in $\sigma$ which are not in $\go_k$.
In particular, we have $n_k(\sigma)=0$ for $k\ge 1+ \max(k_0 ,  \log ( K^{-1} H(\sigma) ))$.
This yields 
\begin{align*}
  G(\sigma) &\le \sum_{k=1}^{k_0} e^k n_0
  + (K^{-1} H(\sigma))^{1/\theta} \sum_{k=0}^{\lfloor \log ( K^{-1}H(\sigma)) \rfloor}
  e^{k(\theta-1) /\theta} \\
  & \le C(\go)+  (1-e^{\frac{1-\theta}{\theta}})^{-1}  K^{-1} H(\sigma) \, .
\end{align*}
Since $K$  is arbitrary, fixing  $K \geq  \frac{2}{\gep} (1-e^{(1-\theta)/\theta})^{-1}$ yields that 
$G(\sigma) \leq C(\go) + \frac{\gep}{2} H(\sigma)$ almost surely, which concludes the proof.
\end{proof}

\begin{proof}[Proof of Lemma \ref{bcantelli}]
Let us start with the case $n=1$. 
Using that $H(\sigma) \geq  \|x\|^2$ if $\sigma =(t,x,\ups)$ is reduced to one point (recall $t\leq 1$), we get that 
$A_{k,1} \subset  \bigcup_{(t,x,\ups) \in \go_k}  \{ \|x\|^2 \leq  Ke^{k} \}$; in other words, if $A_{k,1}$ is verified then there
is a point in $\go_k$ within a distance  $\sqrt{K} e^{k/2}$ from the origin.
The probability
of $A_{k,1}$ is therefore smaller than  a constant times $ K^{d/2}  e^{ kd /2}  \gl_k$ and this is summable over $k$ thanks to~\eqref{eq:condition}. The case $n=2$ can be treated similarly.

\smallskip

When $n\ge 3$, let $\sigma = (t_i,x_i,u_i)_{i=1}^n$ be a path of points in
$\go_k$ staisfying our event, \textit{i.e.} $H(\sigma) \leq K n^{\theta} e^{k}$.
We make the two following claims. 

\smallskip
\noindent
\textit{{\bf Claim 1.} The path cannot venture too far:}
\begin{equation}\label{claim1}
\max_{i\in \lint 1, n\rint} \|x_i\| \le  \sqrt{K} n^{\theta/2} e^{k/2}\, .
\end{equation}

\smallskip
\noindent
\textit{{\bf Claim 2.} There are three consecutive points in our path in a relatively small cylinder:
there exists  $i \in \lint 1,n-2\rint$ satisfying} 
\begin{equation}\label{claim2}
 t_{i+2}-t_i\le \frac{4}{n-2} \quad \text{ and } \quad 
\|x_{i+1}-x_i \|^2+\|x_{i+2}-x_{i-1} \|^2
\le 
\frac{16 K n^{\theta}e^{k}}{(n-2)^2} \, .
\end{equation}

Before proving the two claims~\eqref{claim1}-\eqref{claim2},
let us use them to conclude the proof of Lemma~\ref{bcantelli}.
We can cover
 $[0,1] \times [ - \sqrt{K} n^{\theta/2} e^{k/2}, \sqrt{K} n^{\theta/2} e^{k/2}  ]^d$
with a collection of
$C_{K}   n^{1+d}$  overlapping cylinders 
(of the type $[t,t+\frac{8}{n-2}] \times [x,x+ \frac{8 }{n-2}\sqrt{K} n^{\theta/2} e^{k/2} ]^{d}$ for a collections of $t$'s distant by $\frac{4}{n-2}$ and of $x$'s distant by $\frac{4 }{n-2}\sqrt{K} n^{\theta/2} e^{k/2}$).
Thanks to the above claims,
 if $A_{k,n}$ is verified
 then there exists some path $\sigma \in \cP(\go)$
 of length $n$
 satisfying~\eqref{claim1}-\eqref{claim2}, meaning that
at least one of the constructed cylinders contains
three points in $\pi(\go_k)$.
By a union bound,  since 
each cylinder is of area bounded by $C'_{K} n^{\frac{d}{2} \theta -(d+1)} e^{kd/2}$, we therefore get that
 \begin{align*}
  \bbP(A_{k,n} ) &\leq C_{K} n^{1+d}  \times \Big(  C'_{K} n^{\frac{d}{2} \theta - (d+1) }  e^{kd/2}   \lambda_k\Big)^3 
 \leq  C''_{K}  \big( \gl_k e^{kd/2} \big)^3 n^{\frac{3d}{2}\theta-2(d+1)} \, .
 \end{align*}
We conclude simply by observing that this upper bound is summable over $k$ and $n$, thanks to~\eqref{eq:condition}
and since $\frac{3d}{2}\theta-2(d+1) < -\frac12 (d+1)$ thanks to our choice $\theta<1+\frac{1}{d}$.
 
 \smallskip

 The  first claim \eqref{claim1} just follows by observing that the entropy of the path is 
 larger than $\|x_i\|^2/t_i \ge \|x_i\|^2$ for every $i\in \lint 1, n\rint$.
 For the second claim \eqref{claim2}, we observe that 
 we have
$\sum_{i=1}^{n-2}(t_{i+2}-t_i)\le 2$,
which means that setting
\[
 I:= \Big\{ i\in \lint 1,n-2\rint, \ t_{i+2}-t_i\le \frac{4 }{n-2} \Big\} \, ,
\]
we have $|I|\ge \frac{n-2}{2}$. We also  have, by definition of $I$,
\begin{equation*}
 \sum_{i\in I} \|x_{i+1}-x_i \|^2+\|x_{i+2}-x_{i+1}\|^2\le   \frac{4 }{n-2}
 \sum_{i\in I}\frac{\|x_{i+1}-x_i\|^2}{t_{i+1}-t_i}+\frac{\|x_{i+2}-x_{i+1}\|^2}{t_{i+2} -t_{i+1}} 
  \le \frac{8  H(\sigma) }{n-2} \, .
\end{equation*}
Then \eqref{claim2} is simply a consequence of the fact that 
the smallest element of the sum is smaller than the average
and that $H(\sigma) \leq K n^{\theta} e^k$.
\end{proof}

\subsection{Proof of Lemma \ref{lemmaconv}}
\label{sec:lemmaconv}

We now adapt Proposition~\ref{thm:continuous2}
to prove the convergence of $\cZ_{\gb}^{\go,a}(f)$ and $\cZ_{\gb}^{\go,a}$
without the condition $\int_{[1,\infty)} \ups \lambda(\dd \ups) <\infty$, \textit{i.e.}~we prove Lemma~\ref{lemmaconv}.
Recall the definition  \eqref{splitnoise}
and set for $f\in \cB$
\begin{equation}\label{trunkaparti}
  \cZ^{\go,[a,b)}_{\beta}(f):=1+\sum_{k=1}^\infty \beta^k\int_{  \mathfrak{X}^k \times (\bbR^d)^k} \varrho( \bt , \bx , f) \prod_{i=1}^k \xi^{[a,b)}_{\go} ( \dd t_i, \dd x_i) \, .
\end{equation}
Note that Proposition \ref{thm:continuous2} (or rather Corollary \ref{toutlesf}), applied to the measure $\gl_b$ defined by 
\begin{equation}\label{glbdef}
\gl_b(A)=\gl(A\cap[0,b))
\end{equation}
automatically yields the following convergence of  $\cZ^{\go,[a,b)}_{\beta}(f)$.

\begin{cor}\label{lem:gene}
 Under the assumption \eqref{hyposmall}, for any $f\in \cB$ and any fixed $b\geq 1$ we have that $( \cZ^{\go,[a,b)}(f) )_{a\in(0,1]}$ is a uniformly integrable time-reversed martingale.
Hence the following convergence therefore holds almost surely and in $\bbL_1$:
\[
\lim_{a\to 0}\cZ^{\go,[a,b)}(f)  =:\cZ^{\go,[0,b)}_\beta(f)\, .
\]
Note also that we have $\bbE[ \cZ^{\go,[a,b)}(f) ] = e^{\gb \mu_b} \bQ(f)$ for all $a\in (0,1]$, with $\mu_b:= \int_{[1,b)} \ups \lambda(\dd \ups)$.
\end{cor}
We now prove the convergence of $\cZ^{\go,a}_\beta(f)$ for $f\in \cB_b$. Repeating the argument from the proof of Proposition \ref{prop:contiZfini}, there exists $b_0(\go,f)$ such that 
 for $b\ge b_0$, we have $\cZ^{\go,a}_\beta(f)=\cZ^{\go,[a,b)}_\beta(f)$ for every $a\in(0,1]$. Thus we have, from Corollary~\ref{lem:gene},
 \begin{equation*}
  \lim_{a\to 0}\cZ^{\go,a}_\beta(f)=\cZ^{\go,[0,b_0)}_\beta(f).
 \end{equation*}
It only remains to prove that  $\lim_{a\to 0}\cZ^{\go,a}_{\beta}(f)$  exists and is finite for $f\in \mathcal B$ when \eqref{hypolarge} also holds. We focus on the case $f\equiv 1$ for simplicity but the argument adapts immediately to the case of non-negative bounded $f$, so  there is no loss of generality.
The convergence of~$\cZ^{\go,a}_\beta$ is a consequence of the following statement, valid  
for any $\gep>0$, 
\begin{equation}\label{labello}
  \lim_{b\to \infty} \bbP\Big( \sup_{a\in(0,1]}\big(\cZ^{\go,a}_\beta - \cZ^{\go,[a,b)}_\beta\big) >\gep \Big)  =  0.
\end{equation}
 Indeed, by monotonicity in $b$ (cf.\ Lemma~\ref{eazy}) we have for every $b\in (1,\infty)$,
 $$\liminf_{a\to 0}\cZ^{\go,a}_\beta \ge 
 \cZ^{\go,[0,b)}_\beta.$$ Thus to ensure that $\cZ^{\go,a}_\beta$ converges to $\lim_{b\to \infty}\cZ^{\go,[0,b)}_\beta$ and that the latter is finite,
the only thing that needs  be proved is that for every $\gep>0$, there almost surely exists  $b_0(\go,\gep)$ such that
 \begin{equation}\label{zoukkk}
 \ \limsup_{a\to 0}\cZ^{\go,a}_\beta \le \cZ^{\go,[0,b_0)}_\beta+\gep.
 \end{equation}
Since $\sup_{a\in(0,1]}(\cZ^{\go,a}_\beta - \cZ^{\go,[a,b)}_\beta)$ is non-increasing in $b$,
\eqref{labello} implies that
\begin{equation}
\label{labellobis}
\lim_{b\to\infty} \sup_{a\in(0,1]}\big( \cZ^{\go,a}_\beta - \cZ^{\go,[a,b)}_\beta \big) =0 \qquad \bbP\text{-a.s.}
\end{equation}
and this readily implies that \eqref{zoukkk} holds for some large $b_0$.

To prove~\eqref{labello}, we observe that
the process $\big\{ \gep \wedge (\cZ^{\go,[a,b')}_\beta - \cZ^{\go,[a,b)}_\beta) ; a\in(0,1] \big\}$ is a time-reversed positive (for $b'>b>1$) super-martingale, thanks to Lemma~\ref{lem:martingale}.
Using Doob's  maximal inequality we thus obtain that for any fixed $b'>b>1$
\begin{equation}
 \bbP\Big(\sup_{a\in(0,1]}\big(\cZ^{\go,[a,b')}_\beta - \cZ^{\go,[a,b)}_\beta\big) >\gep \Big)\le  \frac{1}{\gep}\bbE \Big[ \big(\cZ^{\go,[1,b')}_\beta - \cZ^{\go,[1,b)}_\beta\big) \wedge \gep \Big].
\end{equation}
Sending $b'$ to infinity on both sides
 and noting that  $\cZ^{\go,a}_\beta = \lim_{b'\to\infty} \cZ^{\go,[a,b')}_\beta$ by monotone convergence (see Lemma~\ref{eazy}),
we obtain
\begin{equation}\label{thetruc}
  \bbP\Big(\sup_{a\in(0,1]}\big(\cZ^{\go,a}_\beta - \cZ^{\go,[a,b)}_\beta\big) > \gep \Big)\le  \frac{1}{\gep}\bbE \Big[ \big(\cZ^{\go,1}_\beta - \cZ^{\go,[1,b)}_\beta\big) \wedge \gep \Big].
\end{equation}
The right-hand side goes to zero by dominated convergence, thanks to Proposition \ref{prop:Zfinite},
 using again that $\cZ^{\go,1}_\beta = \lim_{b\to\infty} \cZ^{\go,[1,b)}_\beta$.
\qed

\subsection{Almost sure positivity of $\cZ_{\gb}^{\go}(f)$}
\label{sec:positivity}

 In this section, we prove Proposition~\ref{lapositivity}.

Let $f\in \cB_b$ be non-negative and such that $\bQ(f)> 0$.
Recall the definitions \eqref{splitnoise} and~\eqref{trunkaparti} of the truncated noise and of the corresponding partition function. We are going to show first that the positivity of $\cZ^{\go,[0,b)}_\beta (f)$ does not depend on the value of $b$,
for any $ 0<b< b'\le 1$.
\begin{equation}\label{ekwalzs}
 \bbP\Big(  \big\{ \cZ^{\go,[0,b)}_\beta (f)>0 \big\} \triangle  \big\{ \cZ^{\go,[0,b')}_\beta (f) >0 \big\} \Big)   =0.
\end{equation}
where $\triangle$ stands for the symmetric difference (in other words the events are equal in the $\bbL_1$ sense and in particular have the same probability).
 Applying Lemma \ref{eazy} to the measure~$\gl_b$ (recall \eqref{glbdef}) we obtain that
\begin{equation}
  \cZ^{\go,[a,b)}_\beta(f)= e^{-\gb (\kappa_a-\kappa_b)} \sum_{\sigma\in \cP(\go)} \beta^{|\sigma|}   \varrho(\bt,\bx, f)  \prod_{i=1}^{|\sigma|} u_i\ind_{\{ u_i\in[a,b) \}} \, .
\end{equation}
This last expression implies that for every $a<b' \leq b$ we have almost surely 
\begin{equation}
\label{eq:ZaZab}
   \cZ^{\go,a}_\beta (f)  \ge      e^{-\gb\kappa_b}\cZ^{\go,[a,b)}_\beta (f)  \ge     e^{-\gb\kappa_{b'}}  \cZ^{\go,[a,b')}_\beta (f) \,,
\end{equation}
and taking the limit when $a$ tends to zero we obtain 
\begin{equation}
  \cZ^{\go}_\beta (f)  \ge     e^{-\gb\kappa_b}  \cZ^{\go,[0,b)}_\beta (f)  \ge e^{-\gb\kappa_{b'}}  \cZ^{\go,[0,b')}_\beta (f) \, .
\end{equation}
This yields 
\[
\bbP\Big(  \big\{ \cZ^{\go,[0,b')}_\beta (f)>0 \big\} \setminus  \big\{ \cZ^{\go,[0,b)}_\beta (f) >0 \big\}  \Big) =0 \, .
\]
On the other hand, the same argument as in   Lemma \ref{lem:martingale}  yields that $
(\cZ^{\go,[a,b)}_\beta (f))_{b\in(a,1]}$ is a martingale (in $b$) for the filtration~$\mathcal{G}_b$ defined by
\begin{equation}
   \mathcal G_b:= \sigma \big(  \{ (t,x,\ups)\in \go \,\colon  \ups <b \}\big) \,.
\end{equation}
Taking $a$ to zero in the conditional expectation (using uniform integrability cf.\ Corollary~\ref{lem:gene}) we obtain that  for $1\geq b'>b$
\begin{equation}
\bbE \big[ \cZ^{\go,[0, b )}_\beta (f)  \, \big|  \, \mathcal G_{ b'} \big]= \cZ^{\go,[0, b')}_\beta (f)\,  ,
\end{equation}
which yields the second inclusion of \eqref{ekwalzs}
\[
\bbP\Big(  \big\{ \cZ^{\go,[0,b)}_\beta (f)>0 \big\} \setminus  \big\{ \cZ^{\go,[0,b')}_\beta (f) >0 \big\}  \Big) =0.
\]
Now let us fix a decreasing sequence $b_n\in (0,1]$ with $b_n \downarrow 0$ and consider the event $A(f):= \bigcap_{m\ge 0}\bigcup_{n\ge m}\big\{ \cZ^{\go,[0,b_n)}_\beta (f) >0 \big\} $.
An immediate consequence of \eqref{ekwalzs} is that 
 \begin{equation}
 \bbP\Big(  \cZ^{\go,[0,1)}_\beta (f)>0 \Big)= \bbP[ A(f) ]\,.
\end{equation}
Now  $A(f)$ is measurable  with respect to 
the $\sigma$-algebra 
$\mathcal G_0:=\bigcap_{ n\geq 0}\mathcal G_{ b_n} $. Therefore, by Kolmogorov's $0{-}1$ law, it has probability either $0$ or $1$;
 indeed, for any $n\geq 1$, $\cG_0$ is independent of $\cF_{b_n} = \sigma(\{(t,x,\ups) \in \go\, \colon \ups \geq b_n \})$: since $A\in \cG_0$ and $A \in \sigma( \bigcup_{n\geq 0} \cF_{b_n})$, we get that $A$ is independent of itself.
From Corollary~\ref{lem:gene} the martingale $(\cZ_{\gb}^{\go,[a,1)})_{a\in(0,1]}$
is uniformly integrable and thus 
\[
\bbE\big[ \cZ^{\go,[0,1)}_\beta (f) \big]= \bQ(f)>0 \, .
\]
Combining these two facts we obtain that necessarily $\bbP\big(  \cZ^{\go,[0,1)}_\beta (f)>0 \big)=1$, which concludes the proof using \eqref{eq:ZaZab}.
\qed

\begin{rem}
The proofs in Section~\ref{sec:lemmaconv}-\ref{sec:positivity} apply verbatim to the point-to-point partition function. They give that for any $(t,x)\in \bbR_+^* \times \bbR^d$
we have the a.s.\ convergence $\lim_{a\to\infty} \cZ_{\gb}^{\go,a} (t,x) = \cZ_{\gb}^{\go}(t,x)$,
with $\cZ_{\gb}^{\go}(t,x) $ positive and finite almost surely.
\end{rem}

\subsection{Proof of Proposition \ref{proptight}}
\label{sec:convplus}

 We now show that $(\bQ^{\go,a}_{\beta})_{a\in (0,1]}$ is tight.
 We need to find a sequence of
 compact sets $\cK_N$ such that almost surely
\[
\lim_{N\to \infty} \sup_{a\in(0,1]} \bQ^{\go,a}_{\beta}(\cK^{\cc}_N)=0\,.
\]
Since $\cZ^{\go,a}_{\beta}$ converges to a positive limit, this is of course equivalent to proving 
\begin{equation}\label{amontre}
 \lim_{N\to \infty} \sup_{a\in(0,1]} \cZ^{\go,a}_{\beta}  (\ind_{\cK^{\cc}_N}) =0 \, .
\end{equation}
In the case $\mu:=\int_{[1,\infty)} \ups \lambda(\dd \ups) <\infty$,  since $\bbE[\cZ^{\go,a}_{\beta}(\ind_A)]= e^{\gb\mu} \bQ(A)$, using Doob's maximal inequality we have 
\begin{equation}
 \bbP \Big( \sup_{a\in(0,1]} \cZ^{\go,a}_{\beta}(\ind_{\cK^{\cc}_N})\ge  \gep \Big)
 \le \frac{1}{\gep}\,  e^{\gb \mu}\bQ\left[\cK^{\cc}_N\right] . 
\end{equation}
It is then easy to show that \eqref{amontre} holds for an arbitrary increasing sequence of compacts verifying  $\lim_{N\to \infty} \bQ[ \cK_N]=1$
 (using the monotonicity of $f\mapsto \cZ^{\go,a}_{\beta}(f)$, cf.\ Lemma~\ref{eazy}).
For instance  one can take 
$\cK_N:= \big\{ \varphi\in C_0([0,1])\colon    |\varphi(t)-\varphi(s)|\le N |t-s|^{1/4} \ \forall s,t\in [0,1] \big\}$.
In the case $\int_{[1,\infty)} \ups \lambda(\dd \ups) =\infty$,
we proceed analogously with the truncated partition function
$\cZ_{\gb}^{\go,[a,b)}$.
We obtain that for any $b>0$,
\[
\lim_{N\to\infty} \sup_{a\in (0,1]} \cZ_{\gb}^{\go,[a,b)} (\ind_{\cK_N^\cc}) =0 \, .
\]
We then conclude using  \eqref{labellobis}, to get that a.s.
\[
\lim_{b\to\infty} \sup_{a\in (0,1]} \sup_{N\in \bbN} \Big( \cZ_{\gb}^{\go,a} (\ind_{\cK_N^\cc}) - \cZ_{\gb}^{\go,[a,b)} (\ind_{\cK_N^\cc})  \Big) 
\leq \lim_{b\to\infty}  \sup_{a\in (0,1]} \Big( \cZ_{\gb}^{\go,a}  - \cZ_{\gb}^{\go,[a,b)}   \Big) =0\,.
\]
\par\vspace{-\baselineskip}
\qed

\section{Degeneracy of the partition function: Propositions~\ref{prop:toinfinity} and~\ref{prop:tozero}}
\label{sec:trivial}

\subsection{Proof of Proposition~\ref{prop:toinfinity}}
\label{sec:infinity}

Let us assume that $\int_{[1,\infty)} (\log\ups)^{d/2} \lambda(\dd \ups) = \infty$ and show that the partition function $\cZ_{\gb}^{\go,a}$ is a.s.\ infinite.
We use the representation~\eqref{altexp}. Keeping only
paths $\sigma$ with cardinality one in the sum and keeping only those with $t\in[ 1/2,1]$ we have
\begin{equation}
 \cZ^{\go,a}_{\beta}\ge e^{-\beta \kappa_a}\sum_{(t,x,\ups)\in \go^{(a)}} \rho_t(x) \ups
 \ge  
   \frac{e^{-\beta \kappa_a}}{ \pi^{d/2}}\sum_{(t,x,\ups)\in \go^{(a)} \,, \, t\ge 1/2}  \ups\, e^{-\|x\|^2} \, .  
\end{equation}
Hence to conclude it is sufficient to show that almost surely
\begin{equation}
 \sup_{(t,x,\ups)\in \go \colon t\in [1/2,1]} \big\{  \log \ups-    \|x\|^2  \big\}=\infty.
\end{equation}
For this it is sufficient to check that almost surely, the event $A_j$ defined by 
\[
A_j := \big\{ \exists \, (t,x ,\ups)  \in \go\, ,  t\in [\tfrac12,1]  ,\,   \|x\|_{\infty} \in [2^{j-1},2^{j}),\,   \log  \ups \geq
  4^{j+1}  \big\} \, .
\]
 is satisfied for infinitely many $j$. By Borel--Cantelli, since the $A_j$ are by construction independent it suffices to show that $\sum_{j=1}^{\infty} \bbP(A_j) =\infty$.
The number of points $(t,x,\ups)\in \go$
such that $t\in [\tfrac12,1]$,   $\|x\|_{\infty} \in [2^{j-1},2^{j})$ and $ \log \ups \geq    4^{j+1}$
is a Poisson random variable with mean
\[
\lambda_j = \tfrac12\, 2^{d j}(1-2^{-d}) \,  \lambda\big( [\exp(  4^{j+1} ),\infty ) \big) \, .
\]
Hence we have $\bbP(A_j) =1 -e^{-\lambda_j}$ and we simply
need to show that $\sum_{j=1}^{\infty} \lambda_j = \infty$.
But this is a direct consequence of our assumption  $\int_{[1,\infty)} (\log \ups)^{d/2} \lambda(\dd \ups) =\infty$ since
\begin{align*}
 \int_{[1.\infty)} (\log \ups)^{d/2} \lambda(\dd \ups) \leq \sum_{j=1}^{\infty} \int_{\exp(4^{j+1})}^{\exp(4^{j+2})} (4^{j+2})^{d/2} \lambda(\dd \ups)
 \leq  \frac{2}{1-2^{-d}}  \, 4^d \sum_{j=1}^{\infty}  \lambda_j  \, .
\end{align*}
\par\vspace{-\baselineskip}
\qed

\subsection{Proof of Proposition~\ref{prop:tozero}}
\label{sec:zero}

Since the use of the size-biased measure is at the heart of our proof  we are going first to assume that $\mu:=\int_{[1,\infty)} \ups \gl(\ups)\dd \ups<\infty$, in order to be able to use Lemma \ref{spayne}. At the end of the proof we explain how to deal with the case $\mu=\infty$.

\smallskip

Note that $\cZ_{\gb}^{\go,a}$ converges almost surely, as a consequence
of the martingale property: we only need to prove that 
 it converges to zero in probability.
 Since~$\bar \cZ_{\gb}^{\go,a} := e^{-\gb \mu} \cZ_{\gb}^{\go,a}$ is a positive variable with mean $1$ it is sufficient to identify a sequence of events~$J_a$ such that 
\begin{equation}\label{lapremiere}
 \lim_{a\to 0}\bbE\big[\bar \cZ_{\gb}^{\go,a}\ind_{J_{a}}\big]=0 \quad \text{ and } \quad  \lim_{a\to 0}\bbP({J_{a}})=1,
\end{equation}
as it implies that $\bar \cZ_{\gb}^{\go,a}\ind_{J_{a}}$ and thus $\cZ_{\gb}^{\go,a}$ converge to zero in probability.
This is equivalent to proving that the total variation distance between the two measures $\bbP$ and $\tilde \bbP_{\gb}^a$ goes to~$1$.
That is, according to Lemma \ref{spayne}, we need to prove that
\begin{equation}\label{super}
 \lim_{a\to 0} \big\|  \bbP\left(\go \in \cdot\right)- \bQ\otimes \bbE\otimes \bbE'_a\left(\hat \go \in \cdot\right) \big\|_{TV}=1.
\end{equation}
Our proof's strategy relies on finding a statistic that helps to distinguish between $\go$ and~$\hat \go$ for most realizations of the Brownian trajectory.
More precisely, we use the second moment method. We are going to define a functional  $Y_a(\go)$ which verifies
\begin{equation}\label{troukz}
\lim_{a\to 0}
  \frac{\big(\bQ\otimes \bbE\otimes \bbE'_a\left[ Y_a(\hat \go)\right]-\bbE\left[ Y_a(\go)\right]\big)^2}{ \Var_{\bbP}(Y_a(\go))+\Var_{\bQ\otimes \bbE\otimes \bbE'_a}(Y_a(\go))}=\infty.
\end{equation}
The above implies that asymptotically $Y_a(\hat \go)$ and $Y_a(\go)$ concentrate around different values and thus that \eqref{super} holds (see \cite[Prop~7.12]{MCMT} for a quantitative statement and its proof).
We treat separately the cases $d\geq 3$, $d=1$  and $d=2$, in that order.

\subsubsection{The case $d\ge 3$}
We assume that
$\int_{(0,1)} \ups^{1+\frac{2}{d}} \gl(\dd \ups)=\infty$.
In order to find a statistic that allows us to distinguish between  $\bbP$ and $\tilde \bbP^{a}_{\gb}$, 
the idea here is to find a region of $\bbR\times \bbR^d\times \bbR_+$ where $\tilde \bbP^{a}_{\gb}$ 
displays significantly more points than $\bbP$.
We consider a sequence $R_a$ going (slowly) to infinity (we set its value later on)
and we set 
\begin{equation}
 Y_a(\go) := \#\Big\{ (t,x,u) \in \go \, \ : \  \|x\|_\infty \le R_a\sqrt{t}\,  ,\,  u\ge \big( a\vee t^{d/2} \big)\Big\}. 
\end{equation}
Under $\bbP$, $Y_a$ is a Poisson variable with mean given by
\begin{equation}
\label{eq:meanNa}
 \bbE\left[ Y_a(\go)\right]:= (2R_a)^d \int^1_0 t^{d/2} \gl\big( [a\vee t^{d/2}, \infty ) \big)\dd t.
\end{equation}
Note that our assumption on $\gl$ readily implies that $\int^1_0 t^{d/2} \gl([a\vee t^{d/2}, \infty ))\dd t$ and hence  $\bbE\left[ Y_a(\go)\right]$ go to infinity as $a\downarrow 0$.
On the other hand, conditionally on $(B_t)_{t\in[0,1]}$, under $\bbP\otimes \bbP'_a$, we have that $Y_a(\hat \go)$ is a Poisson random variable 
with mean
\begin{equation}\label{condipoisson}
 \bbE\otimes \bbE'_a\left[ Y_a(\hat \go)\right]= (2R_a)^d\int^1_0 t^{d/2} \gl \big( [a\vee t^{d/2}, \infty ) \big)\dd t+ \beta  \cX_a\, ,
\end{equation}
where we have set 
\begin{equation}
  \cX_a:=\int^1_0 \ind_{\{|B_t|\leq  R_a \sqrt{t}\}} \int_{[a\vee t^{d/2}, \infty )} \ups \gl(\dd \ups) \dd t.
\end{equation}
Before averaging with respect to the Brownian motion, since $R_a$ tends to infinity, notice that we can almost replace $\ind_{\{|B_t|\le R_a \sqrt{t}\}}$ by $1$, so $\cX_a$
is close to $m_a:=\int^1_0 \int_{[a\vee t^{d/2}, \infty )} \ups \gl(\dd \ups) \dd t$. With this in mind (and the fact that the variance of  a Poisson variable is equal to its expectation),
the important part that has to be checked for \eqref{troukz} to hold is that 
\begin{equation}\label{keything}
 \lim_{a\to 0}\frac{ m_a^2 }{(R_a)^d \int^1_0 t^{d/2} \gl([a\vee t^{d/2}, \infty ))\dd t}=\infty.
\end{equation}
Since $m_a =\int^1_0 \int_{[a\vee t^{d/2}, \infty )} \ups \gl(\dd \ups) \dd t\ge \int^1_0 t^{d/2} \gl([a\vee t^{d/2}, \infty ))\dd t$
  and 
$\int_{(0,1)}\int^1_0 t^{d/2} \gl([t^{d/2} , \infty ))\dd t =+\infty$ by assumption, the condition
 \eqref{keything} is satisfied as long as $R_a$ diverges slowly enough.
We can choose for instance 
 $$ R_a= \left( \int^1_0 t^{d/2} \gl([a\vee t^{d/2}, \infty ))\dd t\right)^{\frac{1}{2d}}.$$
Now that all the notation have been set, let us complete the proof of~\eqref{troukz}.
Setting $q_a:=\bQ\left(|B_1|\leq  R_a \right)$, we have by that Brownian scaling that
\begin{equation}
\begin{split}\label{largus}
 \bQ(\cX_a)&= q_a \int^1_0\int_{[a\vee t^{d/2}, \infty )} \ups \gl(\dd \ups) \dd t = q_a m_a \, ,\\
 \Var_{\bQ}\left(\cX_a\right)&= \bQ(\cX^2_a)-  \bQ(\cX_a)^2 \le ( 1- q_a^2 )  m_a ^2.
\end{split}
\end{equation}
As a consequence, recalling \eqref{condipoisson}, we have that
\begin{equation}
  \bQ\otimes\bbE\otimes \bbE'_a\left[ Y_a(\hat \go)\right]- \bbE \left[ Y_a( \go)\right]=  \beta q_a\,  m_a\, .
\end{equation}
We also have
\begin{multline}
 \Var_{\bQ\otimes \bbP\otimes \bbP'_a}\left[ Y_a(\hat \go)\right]= \bQ\left(\Var_{\bbP\otimes \bbP'_a}(Y_a(\hat \go) \right)+\beta^2  \Var_{\bQ}\left(\cX_a\right)\\
\le   (2R_a)^d\int^1_0 t^{d/2} \gl([a\vee t^{d/2}, \infty ))\dd t +\beta q_a m_a +
 \beta^2 ( 1- q_a^2) m_a^2.
\end{multline}
Now we can conclude that \eqref{troukz} holds, simply by using \eqref{keything} and the fact that $q_a$ tends to $1$
 as $a\downarrow 0$ (using also that $m_a =o(m_a^2)$ since $m_a$ goes to $\infty$).
\qed

\subsubsection{The case $d= 1$} 
We assume that $\int_{(0,1)} \ups^2 \lambda(\dd \ups) =\infty$.
In this case we set
\begin{equation}
Y_a(\go):= \sum_{(t,x,\ups)\in \go} \ups \ind_{\big\{ \ups \in [a,1),\,  t\in[0,1],\, \|x\|_{\infty} \le R_a \big\}} \, ,
\end{equation} 
where again $R_a$ is a sequence going to infinity sufficiently slowly (it is chosen below).
Then, we have
\begin{equation}
\label{fm1}
 \bbE[Y_a(\go)]= 2R_a \int_{[a,1)}  \ups \gl(\dd \ups) \, , \qquad \Var(Y_a(\go))=2R_a \int_{[a,1)}  \ups^2 \gl(\dd \ups).
\end{equation}
Additionally, 
setting this time  $q_a:=\int^1_0\bQ\left(|B_t|\leq  R_a \right) \dd t$ we have
\begin{equation}
\label{fm2}
 \bQ\otimes \bbE\otimes \bbE'_a \big[ Y_a(\hat\go) \big ]=2R_a \int_{[a,1)}  \ups \gl(\dd \ups)+\beta q_a  \int_{[a,1)}  \ups^2 \gl(\dd \ups) \,, 
\end{equation} 
 and using a variant of the argument used in \eqref{largus}
 \begin{multline}\label{fm3}
 \Var_{\bQ\otimes \bbP\otimes \bbP'_a} \big[ Y_a(\hat\go)\big ]\le 
2R_a \int_{[a,1)} \ups^2 \gl(\dd \ups)+\beta q_a \int_{[a,1)}\ups^3 \gl(\dd \ups) 
 + \beta^2(1-q^2_a)\Big(\int_{[a,1)}  \ups^3 \gl(\dd \ups)\Big)^2\\
 \le (2R_a+\beta q_a)  \int_{[a,1)}  \ups^2 \gl(\dd \ups) + \beta^2(1-q^2_a)\Big(\int_{[a,1)}   \ups^2 \gl(\dd \ups)\Big)^2 \,, 
\end{multline}
where we simply used that $\ups \leq 1$ for the second inequality.
Now, since $q_a$ goes to $1$ as $a\downarrow 0$, to conclude that \eqref{troukz} holds it is sufficient to have $R_a = o\big( \int_{[a,1)} \ups^2 \gl(\dd \ups) \big)$ which can be obtained by setting
$R_a = ( \int_{[a,1)} \ups^2 \gl(\dd \ups) )^{1/2}$.
\qed

\subsubsection{The case $d= 2$}
We assume that 
 $\int_{(0,1)} \ups^2 |\log \ups| \gl (\dd \ups)=\infty$.
 In this case we define
\begin{equation}
Y_a(\go):= \sum_{(t,x,\ups)\in \go} \frac{\ups}{t\vee \ups} \ind_{\big\{ \ups \in [a,1),\,  t\in[0,1],\, \|x\|_{\infty}\le  R_a \sqrt{t} \big\}} \, ,
\end{equation} 
where $R_a$ goes to infinity slowly enough (it is chosen below).
In that case, we have
\begin{equation}\begin{split}
\label{fm11}
 \bbE[Y_a(\go)]&= (2R_a)^2 \int_{[a,1)} \int^1_0\frac{\ups t}{t\vee \ups} \dd t\gl(\dd \ups) \\ \Var(Y_a(\go))&=(2R_a)^2 \int_{[a,1)}\int^1_0 \frac{\ups^2 t}{(t\vee \ups)^2} \dd t \gl(\dd \ups).
 \end{split}
\end{equation}
Now  with $q_a:=\bQ\left(|B_1|\leq  R_a \right)$ we have,
\begin{equation}\label{fm22}
 \bQ\otimes \bbE\otimes \bbE'_a[Y_a(\hat \go)]= (2R_a)^2 \int_{[a,1)} \int^1_0\frac{\ups t}{t\vee \ups} \dd t\gl(\dd \ups)+q_a\beta  \int_{[a,1)} \int^1_0\frac{\ups^2 }{t\vee \ups} \dd t\gl(\dd \ups).
\end{equation}
Using again a variant of \eqref{largus} to bound the variance from above we obtain
 \begin{multline}\label{fm33}
 \Var_{\bQ\otimes \bbP\otimes \bbP'_a} \big[ Y_a(\hat\go)\big ]\le  (2R_a)^2 \int_{[a,1)} \int^1_0\frac{\ups^2 t}{(t\vee \ups)^2} \dd t\gl(\dd \ups)\\+q_a\beta  \int_{[a,1)} \int^1_0\frac{\ups^3 }{(t\vee \ups)^2} \dd t\gl(\dd \ups)+
 \beta^2(1-q_a)^2\Big( \int_{[a,1)} \int^1_0\frac{\ups^2 }{t\vee \ups} \dd t\gl(\dd \ups) \Big)^2.
\end{multline}
To conclude we need to check that \eqref{troukz} holds. It is not difficult to show that the second and third term appearing in the variance of $Y_a(\hat\go)$ can be neglected (recall that $q_a$ goes to~$1$) and hence to conclude one only needs to ensure that
\begin{equation}
 \lim_{a\to 0} \frac{\left(\int_{[a,1)} \int^1_0\frac{\ups^2 }{t\vee \ups} \dd t\gl(\dd \ups)\right)^2}{(2R_a)^2 \int_{[a,1)} \int^1_0\frac{\ups^2 t}{(t\vee \ups)^2} \dd t\gl(\dd \ups)}=\infty \, .
\end{equation}
Now this can be done by setting $R_a= (\int_{[a,1)} \ups^2 |\log \ups| \dd \ups )^{1/4}$
since both integrals in the numerator and the denominator are comparable to 
$\int_{[a,1)} \ups^2 |\log \ups| \dd \ups$.
\qed

\subsubsection{Conclusion of the proof of Proposition~\ref{prop:tozero}}

When $\int_{[1,\infty)} \ups \gl(\ups)\dd \ups<\infty$,
we have shown that if~\eqref{hyposmall2} is not satisfied, then
$\lim_{a\to 0} \cZ_{\gb}^{\go,a}=0$ almost surely.
If $f\in \cB$, simply using that $|\cZ_{\gb}^{\go,a}(f)|\le \|f\|_{\infty}\cZ_{\gb}^{\go,a}$ gives us $\lim_{a\to 0} \cZ_{\gb}^{\go,a}(f)=0$.

\smallskip
Let us now turn to the case $\int_{[1,\infty)} \ups \gl(\ups)\dd \ups=\infty$,.
For $f\in \cB_b$
we can replace the noise $\xi^{(a)}_\go$ by a truncated one
  $\xi^{[a,b)}_\go$  (recall \eqref{splitnoise}), like in the proof of Proposition \ref{prop:contiZfini}, using that~$f$ has a bounded support.
We therefore conclude that we also have 
$\lim_{a\to 0} \cZ_{\gb}^{\go,a}(f)=0$.

It remains to show that if~\eqref{hypolarge} holds, 
then we also have $\lim_{a\to 0} \cZ_{\gb}^{\go,a}=0$ a.s.
We set $f_n(B):= \ind_{\{\max_{t\in[0,1]}|B_t|\le n\}}$ and $\bar f_n=1-f_n$.
We have for any~$n$
$\lim_{a\to 0}\cZ_{\gb}^{\go,a}(f_n)=0$ a.s.\ and 
we can thus conclude if we prove that 
\begin{equation}
\label{toconclude}
 \lim_{n\to \infty} \sup_{a\in (0,1]}\cZ_{\gb}^{\go,a}(\bar f_n)=0
\end{equation}
Using Doob's maximal inequality for the super-martingale $(\cZ^{\go,[a,b)}_{\gb}(\bar f_n)\wedge \gep)_{a\in (0,1]}$, we get
\begin{equation}
  \bbP\Big(\sup_{a\in(0,1]}\cZ^{\go,[a,b)}_\beta(\bar f_n) > \gep \Big) \le  \frac{1}{\gep}\bbE \left[ \cZ^{\go,[1,b)}_\beta(\bar f_n)  \wedge \gep \right].
\end{equation}
Sending $b$ to infinity on both sides 
we get $ \bbP(\sup_{a\in(0,1]}\cZ^{\go,a}_\beta(\bar f_n) > \gep )\le  \frac{1}{\gep}\bbE [ \cZ^{\go,1}_\beta(\bar f_n)  \wedge \gep]$,
which proves~\eqref{toconclude} by dominated convergence,
thanks to Proposition~\ref{prop:Zfinite}.
\qed

\section{Properties of the Continuum Directed Polymer in L\'evy Noise}
\label{sec:properties}

In this section, we prove various properties of the 
measure $\bQ_{\gb}^{\go}$ constructed in Theorem~\ref{thm:zalpha}.
We  always suppose that Assumptions~\eqref{hypolarge}-\eqref{hyposmall} are satisfied.

\subsection{Proof of Proposition \ref{abzolute}}
\label{sec:abzolute}

 The proof is based on the following technical result.
\begin{lemma}\label{mostsurelemma}
 For any fixed  $f\in \cB$ we have almost surely
\begin{equation}\label{mostsure}
 \bQ^{\go}_{T,\beta}(f)= \lim_{a\to 0} \bQ^{\go,a}_{T,\beta}(f)    \,.
\end{equation}
\end{lemma}
\noindent Let us briefly comment on the result above before we present its proof. Lemma~\ref{lemmaconv} and Proposition~\ref{lapositivity} ensure that $\bQ^{\go,a}_{T,\beta}(f)$ converges a.s.\ as a quotient:
for any $f\in \cB$, we have almost surely 
 \begin{equation}
 \label{convas}
 \lim_{a\to 0}\bQ^{\go,a}_{T,\beta}(f)= \frac{\cZ^{\go}_{T,\beta}(f) }{\cZ^{\go}_{T,\beta}}.
\end{equation}
What requires a proof is that for a fixed $f\in \mathcal B$ we almost surely have  
\begin{equation}\label{watacoincidence}
\cZ^{\go}_{T,\beta}(f)= \cZ^{\go}_{T,\beta}\bQ^{\go}_{T,\beta}(f)=:{\bZ}^{\go}_{T,\beta}(f).
\end{equation}
This is a non trivial statement since the two terms correspond to different definitions: $\bQ^{\go}_{T,\beta}(f)$ is the expectation of $f$ with respect to the weak limit of $\bQ^{\go,a}_{T,\beta}$.
The definition of weak convergence only implies that \eqref{watacoincidence} holds when $f\in \mathcal C$.

\smallskip

Note also here that the position of quantifiers is important. The convergence \eqref{mostsure} holds with probability one, simultaneously for every $f\in \mathcal C$, but  not necessarily simultaneously for all bounded measurable functions. The latter statement would correspond to convergence in total variation and does not hold when $\int_{(0,1)}v \gl(\dd v)=\infty$ (see Remark \ref{tvtrem}).

\begin{proof} 
To prove \eqref{mostsure}, 
we only need to verify that for any fixed measurable bounded set~$A$,
we almost surely have
\begin{equation}\label{mostsurevent}
 \lim_{a\to 0}\bQ^{\go,a}_{\gb}(A)=\bQ^{\go}_{\gb}(A)\,.
\end{equation}
Indeed, \eqref{mostsurevent} implies  that \eqref{mostsure} is satisfied for non-negative simple  functions with bounded support. Then given a non-negative $f\in \cB$,  consider a non-negative sequence of simple functions  with bounded support $(f_n)_{n\ge 1}$  which is such that $f_n\uparrow f$.
Using monotone convergence, we have for every $m\ge 1$
\begin{equation}
 \frac{\cZ^{\go}_{T,\beta}(f)}{\cZ^{\go}_{T,\beta}} =\lim_{a\to 0}\bQ^{\go,a}_{\gb}(f)= \lim_{a\to 0}\lim_{n\to \infty} \bQ^{\go,a}_{\gb}(f_n)\ge  \lim_{a\to 0}\bQ^{\go,a}_{\gb}(f_m)=\bQ^{\go}_{\gb}(f_m).
\end{equation}
Letting $m\to \infty$ and using monotone convergence again, we obtain   
$\cZ^{\go}_{T,\beta}(f)\ge  \cZ^{\go}_{T,\beta}\bQ^{\go}_{\gb}(f)$.
Since we also have (for the same reason) $\cZ^{\go}_{T,\beta}(g)\ge  \cZ^{\go}_{T,\beta}\bQ^{\go}_{\gb}(g)$ for  $g=(\|f\|_{\infty}-f)$, we obtain the equality~\eqref{mostsure}.

\smallskip

Let us now prove \eqref{mostsurevent}.
Since $\bQ$ is a regular measure, one can find an increasing sequence of closed set $(A^{(1)}_n)_{n\ge 1}$ and a decreasing
sequence of bounded open sets $(A^{(2)}_n)_{n\ge 1}$ such that 
\begin{equation}
\label{oxos}
 \forall n\ge 1,    A^{(1)}_n\subset A\subset  A^{(2)}_n, \quad \text{ and } \quad  \lim_{n\to \infty}
 \bQ (A^{(2)}_n\setminus A^{(1)}_n)=0.
\end{equation}
The inclusions implies that  
\begin{equation}
 \lim_{a\to 0}\bQ^{\go,a}_{\gb}(A^{(1)}_n) \le  \lim_{a\to 0}\bQ^{\go,a}_{\gb}(A)\le    \lim_{a\to 0}\bQ^{\go,a}_{\gb}(A^{(2)}_n).
\end{equation}
Now, by the Portmanteau theorem, we have for every $n\ge 1$
\begin{equation}\label{oxossi}
\lim_{a\to 0}\bQ^{\go,a}_{\gb}(A^{(1)}_n)\le \bQ^{\go}_{\gb}(A^{(1)}_n)\le \bQ^{\go}_{\gb}(A)\le  \bQ^{\go}_{\gb}(A^{(2)}_n)\le  
\lim_{a\to 0}\bQ^{\go,a}_{\gb}(A^{(2)}_n)\,.
\end{equation}
Hence to conclude  that $\lim_{a\to 0}\bQ^{\go,a}_{\gb}(A)= \bQ^{\go}_{\gb}(A)$, it is sufficient prove that
\begin{equation}\label{ogum}
  \lim_{n\to \infty}\lim_{a\to 0} \bQ^{\go,a}_{\beta} (A^{(2)}_n\setminus A^{(1)}_n)=(\cZ^{\go}_\beta)^{-1} \lim_{n\to \infty} \cZ^{\go}_\beta(\ind_{A^{(2)}_n\setminus A^{(1)}_n})=0\,,
\end{equation}
where the first equality is just a consequence of the definitions (see~\eqref{convas}).
If we assume that $\mu:=\int_{[1,\infty)} \ups \gl(\dd \ups)<\infty$, then using successively Fatou's lemma,  Corollary~\ref{toutlesf} and  \eqref{oxos}, we obtain that
\begin{equation}\label{okok}
\bbE \Big[ \lim_{n\to \infty}\cZ^{\go}_\beta \big( \ind_{A^{(2)}_n\setminus A^{(1)}_n} \big) \Big]\le \lim_{n\to \infty} \bbE \Big[ \cZ^{\go}_\beta \big( \ind_{A^{(2)}_n\setminus A^{(1)}_n} \big) \Big]=  e^{\beta \mu} \lim_{n\to \infty}\bQ(A^{(2)}_n\setminus A^{(1)}_n)=0, 
\end{equation}
proving \eqref{ogum}.
Now if $\int_{[1,\infty)} \ups \gl(\dd \ups)=\infty$, the previous argument yields that for any $b$,
\[
\lim_{n\to \infty}\cZ^{\go,[0,b)}_\beta \big(\ind_{A^{(2)}_n\setminus A^{(1)}_n}\big)=0 \, .
\]
Since $A^{(2)}_n\setminus A^{(1)}_n\subset A^{(2)}_1$ for all $n$ with $A_1^{(2)}$ bounded, we have that for $b$ sufficiently large, 
\[
\cZ^{\go}_\beta\big(\ind_{A^{(2)}_n\setminus A^{(1)}_n}\big)=\cZ^{\go,[0,b)}_\beta\big(\ind_{A^{(2)}_n\setminus A^{(1)}_n} \big)
\]
for every $n$.  This allows to conclude that $\cZ^{\go}_\beta(\ind_{A^{(2)}_n\setminus A^{(1)}_n})$ goes to $0$ also in that case.
\end{proof}

\begin{proof}[Proof of Proposition \ref{abzolute}]
Given $A$ such that $\bQ(A)=0$,  we have a.s.\ $\bQ^{\go,a}_{\gb}(A)=0$, for any $a\in (0,1]$. 
If $\mu<\infty$, this follows from~\eqref{lamine}. If $\mu=\infty$, one has  a.s.\ $\cZ^{\go,[a,b)}_{\gb} (\ind_A) =0$, for any $b>1$: by monotonicity $\cZ^{\go,a}_{\gb} (\ind_A) =\lim_{b\to\infty} \cZ^{\go,[a,b)}_{\gb} (\ind_A) =0$.
Thus Lemma \ref{mostsurelemma} applied to $\ind_A$ implies that $\bQ^{\go}_{\gb}(A)=0$ for almost every $\go$, giving  that $\bbP \rtimes\bQ^{\go}_{\beta}(A)=0$.
\end{proof}

\subsection{Proof of Proposition \ref{emptydance}}
\label{sec:emptydance} 
 
 Let us start with the item (i), 
$\bQ(  \cA_{\mathrm{empty}}(\go))=1$,
which is the easiest statement.
Note that for any $t>0$ and $x\in \bbR^d$ we have $\bQ(B_t=x)=0$.
Hence, 
\begin{equation}
 \bQ(  \cA^{\cc}_{\mathrm{empty}})\le \sum_{(t,x,\ups)\in \go }\bQ(B_t=x)=0.
\end{equation}

Now, to prove item (ii), we notice that when $\kappa_0 :=\int_{(0,1)} \ups \gl(\dd \ups)<\infty$ then almost surely 
the $\bQ^{\go}_\beta$-probability of an event $A$ is given by 
\begin{equation}\label{alter}
 \bQ^{\go}_\beta(A)=\frac{1}{\cZ^{\go}_{\beta}}\sum_{\sigma\in \cP(\go)}  w_{0,\gb}(\sigma,\ind_A)
\end{equation}
where $w_{0,\gb}(\sigma,\ind_A)$ is defined as in \eqref{weightsf}  with $a=0$. We have $w_{0,\gb}(\sigma,f) =\lim_{a\to 0} w_{a,\gb}(\sigma,f)$ provided that $\kappa_0<+\infty$ (otherwise $\lim_{a\to 0} w_{a,\gb}(\sigma,f)=0$).
 Notice that the sum in the r.h.s.\ of~\eqref{alter} is finite since, 
by monotone convergence (note that $e^{\gb \kappa_a} w_{a,\gb}(\sigma,\ind_A)$ is  non-increasing in $a$),
\[
e^{\gb\kappa_0} \sum_{\sigma\in \cP(\go)} w_{0,\gb}(\sigma,\ind_A) = \lim_{a\to 0} e^{\gb\kappa_a} \sum_{\sigma\in \cP(\go)} w_{a,\gb}(\sigma,\ind_A)
=\lim_{a\to 0}  e^{\gb\kappa_a} \cZ^{\go,a}_{\beta} (\ind_A)
= e^{\gb \kappa_0} \cZ^{\go}_{\beta} (\ind_A) \,,
\]
the last term being finite thanks to Lemma~\ref{lemmaconv}. We can then conclude that \eqref{alter} holds using Lemma \ref{mostsurelemma}.
In particular, from~\eqref{alter} we have
\begin{equation}
   \bQ^{\go}_\beta(  \cA_{\mathrm{empty}} )= \frac{e^{-\beta \kappa_0}}{\cZ^{\go}_{\beta}}\,  \in (0,1). 
\end{equation}
For $\cA_{\infty}$ the same argument as in item~(i) shows that  $w_{0,\gb}(\sigma,\ind_{\cA_{\infty}})=0$, since $\cA_{\infty}$ requires that the trajectory visits at least one point outside of $\sigma$ (recall that $\sigma$ is finite), and hence $\bQ^{\go}_\beta(  \cA_{\mathrm{\infty}} )=0$.

Let us now turn to the more delicate item~(iii).
Our idea is to find a sequence $A_{n}$ of sets in $C_0([0,1])$ which are such that
\begin{equation}\label{leclose}
 \lim_{n\to\infty} \bQ^{\go}_{\gb}(A_n^{\cc}) =0 \, \quad \text{a.s.}  \quad 
\text{ and } \quad 
 \limsup A_{n} := \bigcap_{n\ge 0} \bigcup_{m\ge n} A_m\subset \cA_{\mathrm{dense}} \, .
\end{equation}
We will then get that almost surely
\begin{equation}
  \bQ^{\go}_\beta (\cA_{\mathrm{dense}})\ge  \lim_{n\to \infty }  \bQ^{\go}_\beta \Big(       \bigcup_{m\ge n} A_{n}\Big)=1\,  .
\end{equation}
By \eqref{mostsure} the first requirement in~\eqref{leclose} is equivalent to  
$\lim_{n\to\infty} \lim_{a\to 0}  \bQ^{\go,a}_\beta (A^{\cc}_{n})  =0$
and thus to
\begin{equation}
\label{Zan}
\lim_{n\to\infty} \lim_{a\to 0} \cZ^{\go,a}_\beta (\ind_{A^{\cc}_{n}}) =0 \,,
\end{equation}
since $\cZ^{\go,a}_\beta$
converges to a positive limit and is thus bounded away from~$0$.
The obvious way to bound $\cZ^{\go,a}_\beta (A^{\cc}_{n})$ is via the computation of its expectation. For this reason we first assume
that $\mu=\int_{[1,\infty)}\ups \gl (\dd \ups)<\infty$. Let
us denote $z(a,\varphi,\go)$ the maximal spacing in the times of visit  to points in $\go^{(a)}$, \textit{i.e.}
\begin{equation*}
 z(a,\varphi,\go):= \sup \left\{ s \, \colon \exists t\in [0,1-s], \  \  \varDelta(\varphi,\go^{(a)})\cap[t,t+s]=\emptyset\right\}
\end{equation*}
 where we recall that $\varDelta (\varphi, \go) = \{t\in [0,1] \ \colon \exists \ups>0,\, (t,\varphi(t),\ups) \in \go\} $ (see~\eqref{def:Deltaphi}).
Let us also set
\begin{equation*}
 a_n:= \sup \{ a \colon \kappa_a \ge n \}\, ,
\end{equation*}
which goes to $0$ as $n$ tends to infinity.
 To lighten notation, let us set $ \kappa_n:=  \kappa_{a_n} \geq n$ and $\bar \kappa_n:= \mu+ \kappa_{a_n}$.
We then define 
\begin{equation}
 A_{n}:= \left\{  \varphi \, \colon  z(a_n,\varphi,\go)  \le   \kappa_n^{-1/2}   \right\} \, .
\end{equation}
and notice that $\cA_{\mathrm{dense}}$ is satisfied as soon as infinitely many $A_{n}$'s are satisfied. Now, from Lemma~\ref{lem:martingale} we have that $(\cZ^{\go, a }_{\beta}(A_{n}^{\cc}))_{a\in(0,a_n]}$ is a martingale which is uniformly integrable:
we can therefore extend it at $0$.
By Markov's inequality we have
\begin{equation}
\bbP\left(   \lim_{a\to 0}\cZ^{\go, a}_{\beta}(A^{\cc}_{n})\ge n^{-1} \right) \le n \bbE \big[ \cZ^{\go,a_n}_{\beta}(A^{\cc}_{n}) \big].
\end{equation}
The r.h.s.\ can be computed explicitly: we have (recall $t_0=0$, $t_{|\sigma|+1}=1$)
\begin{equation*}
\cZ^{\go,a_n}_{\beta}(A^{\cc}_{n})= \sum_{\sigma \in \cP(\go^{(a_n)}) \, ,\,  \exists  i\in \lint 0, |\sigma|\rint, \ t_{i+1}-t_{i}>  \kappa_n^{-1/2} }  w_{a_n,\beta}(\sigma).
\end{equation*}
Hence we have
\begin{equation*}
 \bbE \big[\cZ^{\go,a_n}_{\beta}(A^{\cc}_{n})\big]
 \le  e^{-\beta \kappa_{n}}\sum_{k=0}^{\infty} (\beta \bar \kappa_n)^k \int_{ \mathfrak{X}^k} \ind_{\{ \exists  i\in \lint 0,\, k\rint ,  t_{i+1}-t_{i}>  \kappa_n^{-1/2} \}}  \dd \bt \, .
\end{equation*}
Now, by symmetry (the roles of $t_{i+1}-t_i$ can be exchanged), we have for any $k\ge 1$
\begin{equation*}
 \int_{ \mathfrak{X}^k} \ind_{\{ \exists  i\in \lint 0, k\rint , t_{i+1}-t_{i}> (\log n)^2/n \}}  \dd \bt \le   (k+1)  \int_{ \mathfrak{X}^k} \ind_{\{  t_k\le  1- \kappa_n^{-1/2} \}}  \dd \bt =  \frac{  (k+1)  \big(1- \kappa_n^{-1/2} \big)^k}{k!} \,.
\end{equation*}
Thus,  bounding $k+1\leq 2 k$ for $k\geq 1$, we have,
\[
  \bbE \big[\cZ^{\go,a_n}_{\beta}(A^{\cc}_{n})\big]
  \le 
 e^{-\beta \kappa_{n}} \left( 1+   2\beta \bar \kappa_n e^{\gb  \bar \kappa_n (1- \kappa_n^{-1/2}) } \right)
 \le e^{-\beta n}+   2\beta \bar \kappa_n e^{\beta \mu} e^{  -\gb  \kappa_n^{1/2} } 
 \le n^{-3}\,,
\]
 where the last inequality holds for $n$ sufficiently large.
We therefore get that for $n$ sufficiently large
 \begin{equation}
 \bbP\left(   \lim_{a\to 0}\cZ^{\go}_{\beta}(A^{\cc}_{n})\ge n^{-1} \right) \le n^{-2} \, ,
 \end{equation}
which is summable, so we conclude that \eqref{Zan} holds a.s.\ by Borel--Cantelli lemma.
When $\int_{[1,\infty)}\ups \gl (\dd \ups) =\infty$, in order to prove~\eqref{Zan} (for the same $A_n$),
we observe that 
\begin{equation}
\cZ^{\go,a}_{\beta}(A^{\cc}_{n})\le \Big( \cZ^{\go,a}_{\beta}-  \cZ^{\go,[a,b)}_{\beta} \Big) + \cZ^{\go,[a,b)}_{\beta}(A^{\cc}_{n}).
\end{equation}
The second term goes to zero by the above proof and the first one can be made arbitrarily small by choosing $b$ large, thanks to Proposition~\ref{prop:Zfinite}  which ensures that $\cZ_{\gb}^{\go,a}<+\infty$.
\qed

\subsection{Proof of Proposition \ref{twomarg}}
\label{sec:twomarg}

Recall that (see~\eqref{p2p})
$$\cZ^{\go}_{\beta}[(t_1,x_1),(t_2,x_2)]:= \limsup_{a\to 0}\cZ^{\go,a}_{\beta}[(t_1,x_1),(t_2,x_2)].$$  
We use a $\limsup$ so that the quantity is defined everywhere (which is necessary for integration).
However,   as a consequence of Proposition~\ref{th:superlem}  (proven in Section~\ref{sec:UI}) and of translation invariance, 
 for a fixed  choice of end points, the $\limsup$ can almost surely be replaced by a limit, and hence
for any fixed $0<t_1<\dots<t_k=1$ and 
  $(x_1, \ldots, x_k)$ 
\begin{equation}
\bbP\left( \exists  i\in \lint 1,k\rint,\  \cZ^{\go}_{\beta}[(t_{i-1},x_{i-1}),(t_i,x_{i})] >  \liminf_{a\to 0} \cZ^{\go,a}_{\beta}[(t_{i-1},x_{i-1}),(t_i,x_{i})] \right)=0.
\end{equation}
Thus, as a consequence of Fubini's theorem, the set 
 \[
\left\{ (x_1, \ldots, x_k) \, \colon   \exists  i\in \lint 1,k\rint,\  \cZ^{\go}_{\beta}[(t_{i-1},x_{i-1}),(t_i,x_{i})] >  \liminf_{a\to 0} \cZ^{\go,a}_{\beta}[(t_{i-1},x_{i-1}),(t_i,x_{i})]  \  \right\}
\]
has almost surely zero Lebesgue measure.
Now, let $g$ be a bounded continuous function of~$k$ variables in $\bbR^d$, satisfying $0\le g \le 1$.
Applying~  \eqref{lescontis} with $\bar g(\varphi):= g(\varphi(t_1),\dots,\varphi(t_k) )$, we have 
\begin{equation}
 \bQ^{\go}_{\beta}( g(B_{t_1},\dots, B_{t_k}))=\lim_{a\to 0}   \bQ^{\go,a}_{\beta}(g(B_{t_1},\dots, B_{t_k})) \,, \qquad \text{a.s.}
\end{equation}
For any $a>0$ we have
\begin{equation}
  \bQ^{\go,a}_{\beta}(g(B_{t_1},\dots, B_{t_k}))= \frac{1}{\mathcal Z^{\go,a}_{\beta}} \int_{(\bbR^d)^k} g(\bx)\prod_{i=1}^k \cZ^{\go,a}_{\beta}[(t_{i-1},x_{i-1}),(t_i,x_{i})] \dd \bx \, .
\end{equation}
To conclude, we only need to show that almost surely
\begin{equation}\label{afaire}
 \lim_{a\to 0} \int_{(\bbR^d)^k} \!\!  g(\bx)\prod_{i=1}^k \cZ^{\go,a}_{\beta}[(t_{i-1},x_{i-1}),(t_i,x_{i})] \dd x_i  = \int_{(\bbR^d)^k} \!\!  g(\bx)\prod_{i=1}^k \cZ^{\go}_{\beta}[(t_{i-1},x_{i-1}),(t_i,x_{i})] \dd \bx \, .
\end{equation}
In particular, taking $g\equiv 1$, this will give that  (recall~\eqref{integreZ})
\begin{equation}
\label{integreZ2}
 \cZ_{\gb}^{\go} = \int_{(\bbR^d)^k} \prod_{i=1}^k \cZ^{\go}_{\beta}[(t_{i-1},x_{i-1}),(t_i,x_{i})] \dd x_i \,.
\end{equation}
Let us first treat the case $\int_{[1,\infty)} \ups \gl(\dd \ups)<\infty$. In that case we have
\begin{align*}
  \bbE  \bigg[ \Big|&\int_{(\bbR^d)^k}  \!\!  g(\bx)\prod_{i=1}^k \cZ^{\go,a}_{\beta}[(t_{i-1},x_{i-1}),(t_i,x_{i})] \dd \bx
  -\int_{(\bbR^d)^k} \!\! g(\bx)\prod_{i=1}^k \cZ^{\go}_{\beta}[(t_{i-1},x_{i-1}),(t_i,x_{i})] \dd \bx  \Big|\bigg]  \\
&\le \int_{(\bbR^d)^k} \!\! g(\bx)   \bbE \bigg[ \Big| \prod_{i=1}^k \cZ^{\go,a}_{\gb}[(t_{i-1},x_{i-1}),(t_i,x_{i})] 
  -  \prod_{i=1}^k \cZ^{\go}_{\beta}[(t_{i-1},x_{i-1}),(t_i,x_{i})]   \Big|\bigg] \dd \bx \, .
\end{align*}
Now, thanks to Proposition \ref{th:superlem} and the fact 
that the product of independent variables converging in~$\bbL_1$
also converges in $\bbL_1$, we have for every $x_1,\dots, x_k \in \R^d$
\begin{equation}
\label{eq:conv1}
 \lim_{a\to 0}\bbE \bigg[ \Big| \prod_{i=1}^k \cZ^{\go,a}_{\gb}[(t_{i-1},x_{i-1}),(t_i,x_{i})]
  -  \prod_{i=1}^k \cZ^{\go}_{\beta}[(t_{i-1},x_{i-1}),(t_i,x_{i})]   \Big|\bigg]=0.
\end{equation}
Morever, we have (recall \eqref{zamine})
\begin{equation}
\label{eq:domin1}
 \bbE \bigg[ \Big| \prod_{i=1}^k \cZ^{\go,a}_{\gb}[(t_{i-1},x_{i-1}),(t_i,x_{i})]
  -  \prod_{i=1}^k \cZ^{\go}_{\beta}[(t_{i-1},x_{i-1}),(t_i,x_{i})]   \Big|\bigg]
  \le 2 e^{\gb \mu } \varrho(\bt,\bx)
\end{equation}
and thus
the convergence~\eqref{afaire} holds in $\bbL_1$  by dominated convergence.
The fact that the convergence is also  almost sure
comes from the fact that the l.h.s.\ in~\eqref{afaire}
is a martingale.

In the case $\int_{[1,\infty)} \ups \gl(\dd \ups)=\infty$, we can apply~\eqref{afaire} for the truncated environment and obtain 
 that almost surely
\begin{multline}\label{afaire1}
 \lim_{a\to 0} \int_{(\bbR^d)^k} \!\! g(\bx)\prod_{i=1}^k \cZ^{\go,[a,b)}_{\beta}[(t_{i-1},x_{i-1}),(t_i,x_{i})] \dd \bx \\ = \int_{(\bbR^d)^k} g(\bx)\prod_{i=1}^k \cZ^{\go,[0,b)}_{\beta}[(t_{i-1},x_{i-1}),(t_i,x_{i})] \dd \bx.
\end{multline}
Bounding above $\cZ^{\go,[a,b)}_{\beta}$ by $\cZ^{\go,a}_{\gb}$
in the l.h.s.\ and using monotone convergence for the r.h.s., we obtain that
\begin{equation}\label{afaire2}
 \lim_{a\to 0} \int_{(\bbR^d)^k}  \!\!g(\bx)\prod_{i=1}^k \cZ^{\go,a}_{\beta}[(t_{i-1},x_{i-1}),(t_i,x_{i})] \dd \bx
   \ge \int_{(\bbR^d)^k} \!\! g(\bx)\prod_{i=1}^k \cZ^{\go}_{\beta}[(t_{i-1},x_{i-1}),(t_i,x_{i})] \dd \bx.
\end{equation}
Since the same inequality is also valid for $1-g$, to conclude it is sufficient to check that we have equality when $g\equiv 1$. This corresponds to checking that 
\begin{equation}
\label{intpointtopoint}
 \int_{(\bbR^d)^k} \prod_{i=1}^k \cZ^{\go}_{\beta}[(t_{i-1},x_{i-1}),(t_i,x_{i})] \dd x_i=
 \cZ^{\go}_\beta\, .
\end{equation}
But thanks to~\eqref{afaire}  (see in particular~\eqref{integreZ2}), we have
\begin{equation*}
  \int_{(\bbR^d)^k} \prod_{i=1}^k \cZ^{\go,[0,b)}_{\beta}[(t_{i-1},x_{i-1}),(t_i,x_{i})] \dd x_i=
 \cZ^{\go,[0,b)}_\beta \, ,
\end{equation*}
for all $b$, so \eqref{intpointtopoint} follows by monotone convergence.
\qed

\section{Stochastic Heat Equation with L\'evy noise: proof of Proposition~\ref{prop:SHE}}
\label{sec:propSHE}

 Recall that Proposition \ref{lem:SHE} has been proven in Section \ref{sec:UI} (see Remark \ref{proprop}).
It remains to show that for fixed  $t \in [0,1]$ and $x\in \R^d$,
we have
\begin{equation}
\label{gogoal}
\lim_{a\to 0} \int_{\bbR^d} 
\cZ^{\go,a}_{\beta}[(0,y),(t,x)] \, u_0(\dd y ) 
= \int_{\bbR^d} 
\cZ^{\go}_{\beta}[(0,y),(t,x)] \, u_0(\dd y )  
\end{equation} 
and that the right-hand-side is finite.
For simplicity, we assume that $u_0$ is a positive measure,
since otherwise 
we simply treat the positive and negative parts of $u_0$ separately.

In the case $\int_{[1,\infty)} \ups \lambda(\dd \ups) <\infty$,
we can repeat the proof of Proposition~\ref{twomarg}.
We have
\begin{equation*}
\bbE\big[ \big| u^{a}(t,x) -u(t,x) \big| \big]
  \le \int_{\bbR^d}   \bbE \Big[ \Big|  \cZ^{\go,a}_{\gb}[(0,y),(t,x)] 
  -   \cZ^{\go}_{\beta}[(0,y),(t,x)]   \Big|\Big] \,  u_0 (\dd y) \, .
\end{equation*}
Using \eqref{eq:conv1}-\eqref{eq:domin1} with $k=1$
together with the fact that $\int_{\bbR^d} \rho_t(y-x) u_0(\dd y) <\infty$ thanks to our assumption~\eqref{hypou0},
we conclude by dominated convergence that 
\[
\lim_{a\to 0} u^{a}(t,x)   
= u(t,x) \]
in $\bbL_1$ and almost surely (since $(u^a(t,x))_{a\in(0,1]}$ is a martingale).

\smallskip

Let us now turn to the case $\int_{[1,\infty)} \ups \lambda(\dd \ups)  = \infty$.
Recall the definition~\eqref{defuab} of $u^{[a,b)}$ and
notice that for all $a\in(0,1]$ and $b\ge 1$ we have $u^{[a,b)}(t,x) \leq u^{a}(t,x) <\infty$
almost surely.
Applying the $\bbL_1$ and a.s.\ convergence 
with the truncated environment,
we get that
\begin{equation}
\lim_{a\to 0}  u^{[a,b)}(t,x) = u^{[0,b)} (t,x) := 
 \int_{\bbR^d} 
\cZ^{\go,[0,b)}_{\beta}[(0,y),(t,x)] \, u_0(\dd y )\,, \qquad \text{a.s.}
\end{equation}
To conclude, we need to show that we can take the limit
$b\to\infty$ uniformly for $a\in (0,1]$.
More precisely,  similarly to~\eqref{labello}, we show that
for any $\gep>0$ we have
\begin{equation}
\label{labello2}
\lim_{b\to\infty} \bbP\Big(  \sup_{a\in (0,1]} \big( u^{a}(t,x)  - u^{[a,b)}(t,x) \big) >\gep \Big) =0 \, .
\end{equation}
Indeed,  for any $b'>b\geq 1$, considering the super-martingale $\gep \wedge ( u^{[a,b')}(t,x)  - u^{[a,b)}(t,x) )_{a\in(0,1]}$  
and applying Doob's inequality, we get
\[
\bbP\Big(  \sup_{a\in (0,1]} \big( u^{[a,b')}(t,x)  - u^{[a,b)}(t,x) \big) >\gep \Big) \leq \frac1\gep \bbE\Big[ \gep \wedge \big( u^{[1,b')}(t,x)  - u^{[1,b)}(t,x) \big)  \Big] \, .
\]
Sending $b'$ to infinity
we therefore get by monotone convergence (analogously to~\eqref{thetruc})
\[
\bbP\Big(  \sup_{a\in (0,1]} \big( u^{a}(t,x)  - u^{[a,b)}(t,x) \big) >\gep \Big) \leq \frac1\gep \bbE\Big[ \gep \wedge \big( u^{1}(t,x)  - u^{[1,b)}(t,x) \big)  \Big] \, .
\]
Then, since  $u^{1}(t,x)<\infty$,
the limit~\eqref{labello2} follows by dominated convergence.
As a by-product, this shows that $u(t,x)<\infty$ a.s.
\qed

\appendix

\section{On L\'evy noises}

Recall that we consider  a Poisson process  $\go$ on $\bbR\times \bbR^d \times \bbR_+$ with intensity 
$\dd t \otimes \dd x \otimes \gl( \dd \ups)$,
where $\gl$ is a positive measure on $(0,\infty)$ with $\lambda([1,\infty))<\infty$. 
Recall the definition~\eqref{petitnoise} of the truncated and centered measure $\xi_{\go}^{(a)}$ on $\bbR\times \bbR^d$.


\begin{proposition}\label{convergence2}
If $\gl$ satisfy
$\int_{(0,1)} \ups^2 \gl( \dd \ups)<\infty$,
then  $\xi^{(a)}_{\go}$ converges in  $H^{-s}_{\mathrm{loc}}(\bbR\times \bbR^d)$ with $s>(d+1)/2$, \textit{i.e.}\ there exists a distribution $\xi_{\go} \in H^{-s}_{\mathrm{loc}}(\bbR\times \bbR^d)$ such that
for  any non-negative smooth compactly supported $\psi$
\begin{equation}
 \lim_{a\to 0}  \| \psi(\xi_{\go}^{(a)}-\xi_{\go}) \|_{H^{-s}}=0.
\end{equation}
\end{proposition}

\begin{rem}
\label{rem:noiseL2}
The optimality of the criterion above on $\gl$ can be checked by computing the characteristic function  of the random variable $\int \theta \dd \xi_{\go}^{(a)}$ for a smooth $\theta$ with compact support. We have  
\begin{equation*}
 \chi^{(a)}(\theta):=\bbE[ e^{i \int \theta \dd \xi_{\go}^{(a)}}]=\exp\left( \int_{ \bbR\times \bbR^d \times [a,\infty)} \left(e^{i  v \theta(t,x)}-1-iv\theta(t,x)\ind_{[a,1)}(v)\right) \dd t \dd x \gl(\dd v)\right).
\end{equation*}
 When $\int_{(0,1)} \ups^2 \gl( \dd \ups)=\infty$, we have $\lim_{a\to 0} \chi^{(a)}(u\theta)=\ind_{\{0\}}(u)$ for any non-trivial $\theta$, and hence   $\int \theta \dd \xi_{\go}^{(a)}$ does not converge (even in distribution). 
\end{rem}

\begin{proof}
Let $\psi$ be non-negative,  $C^{\infty}$ and compactly supported. Let us show that $(\psi\, \xi_{\go}^{(a)})_{a\in(0,1]}$ is a Cauchy sequence in $H^{-s}(\bbR^{1+d})$.
In this proof only, we denote $\bar d = 1+d$, and $x,y$ will denote elements of $\bbR^{\bar d}$.

\smallskip

First, we observe that  $\psi\,  \xi_{\go}^{(a)}$ belongs to $H^{-s}(\bbR^{1+ d})$ almost surely because its Fourier transform is smaller than $|\xi_{\go}^{(a)}| (\psi) <+\infty$,
and $(1+|z|^2)^{-s}$ is integrable as $2s>\bar d$. 
Now, to show that $(\xi_{\go}^{(a)})$ is a Cauchy sequence, we need to prove that 
\[
\lim_{a\to 0} \sup_{b\in(0,a]}\|\psi\, (\xi_{\go}^{(a)}-\xi_{\go}^{(b)}) \|_{H^{-s}}=0.
\]
We have (recall the definition~\eqref{Hsnorm})
\begin{equation*}
 \|\psi (\xi_{\go}^{(a)}-\xi_{\go}^{(b)}) \|^2_{H^{-s}}
 =\int_{\bbR^{\bar d}} (1+|z|^2)^{-s} \left|\int_{\bbR^{\bar d}} e^{i z\cdot x} \psi(x)(\xi_{\go}^{(a)}-\xi_{\go}^{(b)})(\dd x)\right|^2 \dd z 
\end{equation*}
and hence by Fatou's lemma 
\begin{equation}\label{lefatoux}
  \sup_{b\in(0,a]}\|\psi( \xi_{\go}^{(a)}-\xi_{\go}^{(b)} )\|^2_{H^{-s}}
  \le \int_{\bbR^{\bar d}} (1+|z|^2)^{-s}  \sup_{b\in (0,a]} \left|\int_{\bbR^{\bar d}} e^{i z\cdot x} \psi(x)(\xi_{\go}^{(a)}-\xi_{\go}^{(b)})(\dd x) \right|^2 \dd z\,.
\end{equation}
Now, as $\big( \int(\xi_{\go}^{(a)}-\xi_{\go}^{(b)})(\dd x) e^{i z \cdot x} \psi(x) \dd x \big)_{b\in(0,a]}$
is a (complex valued) time reversed martingale with càdlàg path, Doob's maximal inequality yields
\begin{equation*}
 \bbE \left[  \sup_{b\in(0,a]} \left|\int_{\bbR^{\bar d}} e^{i z \cdot x}\psi(x)(\xi_{\go}^{(a)}-\xi_{\go}^{(b)})(\dd x)  \right|^2 \right]
 \le 4 \lim_{b\to 0}\bbE \left[ \left|\int_{\bbR^{\bar d}} e^{i z \cdot x}\psi(x)(\xi_{\go}^{(a)}-\xi_{\go}^{(b)})( \dd x)  \right|^2 \right].
\end{equation*}
The right-hand side is then straightforward to compute.
Expanding the square and using the fact that the off-diagonal integral
\begin{equation*}
\int_{\bbR^{\bar d} \times \bbR^{\bar d}} \psi(x) \psi(y) \ind_{\{x\ne y\}}e^{i z \cdot (x-y)}(\xi_{\go}^{(a)}-\xi_{\go}^{(b)})(\dd x) (\xi_{\go}^{(a)}-\xi_{\go}^{(b)})(\dd y)
\end{equation*}
has zero mean, we obtain that for any $a>b>0$
\begin{equation*}
\begin{split}
 \bbE \left[\left| \int_{\bbR^{\bar d}} e^{i z\cdot x} \psi(x)(\xi_{\go}^{(a)}-\xi_{\go}^{(b)})(\dd x)  \right|^2 \right]
& = \bbE \bigg[ \sum_{(x,\ups)\in \go} \ups^{2} \ind_{\{\ups\in (b,a]\}}\psi(x)^2 \bigg]  \,.
 \end{split} 
\end{equation*}
We can compute the last expression using the formula for Poisson point process (see Proposition~\ref{OKLMecke}): we obtain 
\begin{equation*}
\bbE \bigg[ \sum_{(x,\ups)\in \go} \ups^{2} \ind_{\{\ups\in (b,a]\}}\psi(x)^2 \bigg]  \le  \bbE \bigg[ \sum_{(x,\ups)\in \go} \ups^{2} \ind_{\{\ups\in (0,a]\}} \psi(x)^2 \bigg] = \int_{\bbR^{\bar d}} \psi(x)^2 \dd x \int_{(0,a)} \ups^2 \gl(\dd \ups) \, .
\end{equation*}
Hence we obtain from \eqref{lefatoux}
\begin{equation}
 \bbE \bigg[  \sup_{b\in(0,a]}\|\psi(\xi_{\go}^{(a)}-\xi_{\go}^{(b)} )\|^2_{H^{-s}}\bigg]\le \int_{\bbR^{\bar d}} (1+|z|^2)^{-s} \dd z \int_{\bbR^{\bar d}} \psi(x)^2 \dd x \int_{(0,a)} \ups^2 \gl(\dd \ups) \, ,
\end{equation}
which is sufficient to conclude that $\lim_{a\to \infty} \sup_{b\in(0,a]}\|\psi(\xi_{\go}^{(a)}-\xi_{\go}^{(b)}) \|^2_{H^{-s}}=0$ a.s., by Fatou's lemma.
\end{proof}

\section{Proofs of some properties of $\cZ_{\gb}^{\go,a}$ }
\label{app:conti}

\subsection{Alternative representation of $\cZ_{\gb}^{\go,a}$: proof of Lemma~\ref{eazy}}
\label{app:eazy}

For $\sigma\in \cP(\go)$, we let $\cE^k(\sigma)$ denote the collection of $k$ space-time points which include $(t_i,x_i)_{i=1}^{|\sigma|}$ and no other space-points of the Poisson process, that is
 \begin{equation}
  \cE^k(\sigma):= \{ (\bt,\bx)\in \mathfrak{X}^k\times (\bbR^d)^k \ : \ \{(t_i,x_i)\}_{i=1}^k\cap \pi(\go)= \pi(\sigma)  \},
\end{equation}
where $\pi$ denote the  projection on the first two coordinates.
The following technical result, which immediately implies Lemma~\ref{eazy},  establishes that $w_{a,\beta}(\sigma,f)$  (recall its definition~\eqref{weightsf}) corresponds to the contribution to the partition function of the integral over the disjoint unions of $ \cE^k(\sigma)$, $k\ge |\sigma|$. 
The idea is that the integration over $\cE^k(\sigma)$ decouples the Poisson and the Lebesgue parts of~$\xi_{\go}^{a}$,
\textit{i.e.}\ the contribution of points respectively inside and outside of $\sigma$.

\begin{lemma}
\label{lemma:claim}
For any  $f\in \cB$ and any given $\sigma\in \cP(\go)$ we have
 \begin{equation*}\label{leclaim}
  w_{a,\beta}(\sigma,f):= e^{- \beta  \kappa_a } \gb^{|\sigma|}   \varrho(\bt,\bx,f) \prod_{i=1}^{|\sigma|}  u_i\ind_{\{u_i\geq a\}} =\sum_{k\ge |\sigma|} \beta^k\int_{  \cE^k(\sigma)}  \varrho(\bt,\bx,f)  \prod_{i=1}^k \xi_{\go}^{(a)} (\dd t_i,\dd x_i).
 \end{equation*}
When $\sigma=\emptyset$ the term $k=0$ in the sum is by convention equal to $\bQ(f)$.
\end{lemma}

In order to lighten notations, we write the proof only 
in the case of a function $f\equiv 1$. We assume that
$\int_{[1,\infty)} \ups \lambda(\dd \ups) <\infty$ so that 
all the integrals below are well defined (recall Proposition~\ref{prop:contiZfini}).
The general case with a function $f\in \cB_b$ (for which all terms are well defined  assuming that $\lambda([a,\infty))<\infty$, thanks to Proposition~\ref{prop:contiZfini})
is a mere adaptation of notation.

\begin{proof}[Proof of Lemma~\ref{lemma:claim}]
Let $\sigma$ be a fixed set of points with $|\sigma|=\ell \geq 1$ (the case $\sigma=\emptyset$ can be checked separately) and let $(t_i,x_i,u_i)_{i=1}^{\ell}$ denote the (time ordered) points in $\sigma$.
Given $i\in \lint 0,\ell\rint$, and $(s,y)\in \cE^k(\sigma)$,
we let $s^{(i)}_j$, $y^{(i)}_j$, $j\le k_i$ denote the space time points of
$(s,y)$ in the time interval $(t_i,t_{i+1})$ (note that $k_i$ here is a function of $(s,y)$). Then by splitting the integrals according the the value of the~$k_i$'s, grouping the terms and factorizing,  we obtain that $ \sum_{k\ge \ell}  \beta^k    \int_{  \cE^k(\sigma)}    \varrho (\bt,\bx) \prod_{i=1}^k \xi_{\go}^{(a)}(\dd t_i,\dd x_i) $ is equal to
\begin{equation*}
 \label{horbltrk} 
 \begin{split}
& \beta^{\ell}\prod_{i=0}^{\ell-1} u_{i+1}  \ind_{\{u_{i+1} \geq a \}} \sum_{k_i=0}^\infty
 \left(- \gb   \kappa_a  \right)^{k_i} \\
&\times \int_{t_{i}<s^{(i)}_1<\dots<s^{(i)}_{ k_i}<t_{i+1}}\int_{(\bbR^d)^{k_i}}\left( \prod_{j=1}^{k_i+1}
 \rho_{\Delta_j s^{(i)}}(\Delta_j y^{(i)}) \right) \prod_{j=1}^{k_i}\dd s^{(i)}_j \dd y^{(i)}_j \\
 &  \times \sum_{k_{\ell}=0}^\infty  
 \left(-\gb\kappa_a \right)^{k_{\ell}} \int_{t_{\ell}<s^{ (\ell)}_1<\dots<s^{ (\ell)}_{k_{\ell}}<1}\int_{(\bbR^d)^{k_{\ell}}}\left( \prod_{j=1}^{k_{\ell}}
 \rho_{\Delta_j s^{(\ell)}}(\Delta_j y^{(\ell)}) \right) \prod_{j=1}^{k_{\ell}} \dd  s^{(\ell)}_j \dd  y^{(\ell)}_j
 \end{split}
\end{equation*}
with the convention  $\Delta_j s^{(i)}:= s^{(i)}_{j}-s^{(i)}_{j-1} $,
$s^{(i)}_0=t_{i}$, $s^{(i)}_{k_{i}+1}=t_{i+1}$ and analogously for $y^{(i)}_{j}$. When $k_i=0$ resp.\ $k_{\ell}=0$, the  value of the above integrals are by convention 
\[\int_{t_{i}<s^{(i)}_1<\dots<s^{(i)}_k<t_{i+1}}\int_{(\bbR^d)^{k_i}}\left( \prod_{j=1}^{k_i+1}
 \rho_{\Delta_j s^{(i)}}(\Delta_j y^{(i)}) \right) \prod_{j=1}^{k_i}\dd s^{(i)}_j \dd y^{(i)}_j=\rho_{t_{i+1}-t_i}(x_{i+1}-x_i)
 \]
 and 
 \[
 \int_{t_{\ell}<s^{(i)}_1<\dots<s^{(i)}_{k_{\ell}}<1}\int_{(\bbR^d)^{k_{\ell}}} \left( \prod_{j=1}^{k_{\ell}}
 \rho_{\Delta_j s^{(\ell)}}(\Delta_j y^{(\ell)}) \right) \prod_{j=1}^{k_{\ell}} \dd  s^{(\ell)}_j \dd  y^{(\ell)}_j =1.
 \]
Now, one can check that 
\begin{align*}
 \sum_{k_i=0}^\infty
 \left(-\gb\kappa_a \right)^{k_i}  &
 \int_{t_{i}<s^{(i)}_1<\dots<s^{(i)}_k<t_{i+1}}\int_{(\bbR^d)^{k_i}}\left( \prod_{j=1}^{k_i+1}
 \rho_{\Delta_j s^{(i)}}(\Delta_j y^{(i)}) \right) \prod_{j=1}^{k_i}\dd s^{(i)}_j \dd y^{(i)}_j
 \\ 
 &= \rho_{t_{i+1} -t_i}( x_{i+1}-x_i)\left(1+  \sum_{k_i=1}^\infty   \left(-\gb \kappa_a \right)^{k_i}  \int_{t_{i}<s^{(i)}_1<\dots<s^{(i)}_k<t_{i+1}} \prod_{j=1}^{k_i}\dd s^{(i)}_j\right)\\
& = \rho_{t_{i+1} -t_i}( x_{i+1}-x_i) e^{-\gb \kappa_a (t_{i+1}-t_i)},
\end{align*}
and similarly
\[
\sum_{k_{\ell}=0}^\infty   \left(-\gb \kappa_a\right)^{k_{\ell}} \int_{t_{\ell}<s^{ (\ell)}_1<\dots<s^{ (\ell)}_{k_{\ell}}<1}\int_{(\bbR^d)^{k_{\ell}}}
\left( \prod_{j=1}^{k_{\ell}}
 \rho_{\Delta_j s^{(\ell)}}(\Delta_j y^{(\ell)}) \right) \prod_{j=1}^{k_{\ell}} \dd  s^{(\ell)}_j \dd  y^{(\ell)}_j \\
 = e^{-\gb \kappa_a  (1-t_{\ell})},
\]
which concludes the proof.
\end{proof}
%
%
%
%

\begin{proof}[Proof of \eqref{altexpp2p}]
The proof of \eqref{altexpp2p} works exactly as above when  $\int_{[1,\infty)} \ups \lambda(\dd \ups) <\infty$. For the general case, one needs first to have an identity for a positive integral so that we have no trouble with our definition.
Truncating the environment and using monotone convergence we have
\begin{multline}
 \rho_t(x) +  \sum_{k=1}^\infty \beta^k\int_{0<t_1<\dots<t_k<t}\int_{(\bbR^d)^k}  \varrho(\bt,\bx) \rho_{t-t_k}(x-x_k)\prod_{i=1}^k  |\xi_{\go}^{(a)}| ( \dd t_i, \dd x_i) 
\\ = e^{2\kappa_a t\beta }\sum_{\cP_{[0,t]}(\go)}w_a(\sigma,(t,x)).
\end{multline}
This ensures that the sum and integrals in \eqref{def:pointtopoint} are convergent if and only if  we have
 $\sum_{\cP_{[0,t]}(\go)}\go(\sigma,(t,x))<\infty$. Then, repeating the proof above, we obtain that \eqref{altexpp2p} holds.
\end{proof}

\subsection{The size-biased measure: proof of Lemma~\ref{spayne}}\label{App:spayne}

Let us recall here, for the sake of clarity, the content of Lemma~\ref{spayne}.
The \emph{size-biased} measure $\tilde \bbP_{\gb}^{a}$ is defined as
$\tilde \bbP_{\gb}^{a}(J) = \bbE[\bar \cZ_{\gb}^{\go,a} \ind_{J}]$.
Then Lemma~\ref{spayne}
states that for all bounded measurable function $f$,
\[
\tilde \bP_{\gb}^{a}[f(\go)] = \bbP\otimes \bbP'_a\otimes \bQ\left[ f(\hat \go(\go,\go'_a,B))\right] \, ,
\]
where $\hat \go(\go,\go'_a,B) = \go \cup\{ (t,B_t,u) \colon (t,u)\in \go_a'\}$,
with $\bQ$ the distribution of a standard Brownian motion $B$, $\bbP_a'$ the distribution 
of a Poisson point process $\go_a'$ on $[0,1]\times \bbR^+$
with intensity $\dd t \otimes \gb \ga \ups \ind_{\{\ups\geq a\}}  \lambda( \dd u)$ and $\bbP$  the distribution of the Poisson point process $\go$ introduced in~\eqref{Poissondens}. Recall that we assume that $\mu:=\int_{[1,\infty)} \ups \lambda(\dd \ups) <\infty$

\begin{proof}[Proof of Lemma~\ref{spayne}]
It is sufficient to check that the distributions of the two point processes in Equation~\eqref{sides} coincide when restricted to
$[0,1]\times \bbR^d\times [a,\infty)$, since their distributions outside of this set are unaltered by the size-biasing and remain independent of the rest.

Given  a bounded measurable subset $A$ of $[0,1]\times \bbR^d\times [a,\infty)$, we define 
\begin{equation}
 \cN_A:= \#(\go \cap A) 
\end{equation}
Our proof starts with the observation that the distribution of simple point processes is completely characterized by
$\bbP \left( \cN_A=0 \right)$ for  all  bounded and measurable set $A$, see \cite[Theorem 6.11]{PoiBook}, and hence \textit{a fortiori} by the distribution of $\cN_A$.

Hence, setting $\hat \cN_A:= \#(\hat\go \cap A)$,  it is sufficient for us to prove that for every set $A$ and any  $k \geq 0$
\begin{equation}
\label{sidesB}
 \frac{1}{k!}\tilde \bbE^a_{\beta} \big[ \cN_A(\cN_A-1)\cdots (\cN_A-k) \big]=  \frac{1}{k!}\bbP\otimes \bbP'_a \otimes \bQ\left[ \hat \cN_A(\hat\cN_A-1)\cdots (\hat \cN_A-k) \right]
\end{equation}
and that the quantities above do not grow faster than exponentially, so that the distributions of $\cN_A$ and $\hat \cN_A$ are indeed characterized by their moments.

Let us define $f=f_{k,A} \colon \left([0,1]\times \bbR^d\times [a,\infty)\right)^k \to \bbR$ by
\[
f_{k,A} \big( (t_i,x_i,u_i)_{i=1}^k \big):= \ind_{\{t_1< t_2 < \dots < t_k\}} \prod_{i=1}^k \ind_{A}(t_i,x_i,u_i).
\]
Since almost surely, there are no two points in $\go$ with the same time coordinate we have almost-surely
\begin{equation}
 \sum_{(t_i,x_i,u_i)_{i=1}^k\in \, \go^k} f_{k,A}\big( (t_i,x_i,u_i)_{i=i}^k \big) = \frac{1}{k!} \cN_A(\cN_A-1)\cdots (\cN_A-k  +1)
\end{equation}
and the analogous identity is valid for $\hat \cN_A$.
Hence we can check that the identity \eqref{sides} holds simply by applying Mecke's formula (Proposition \ref{OKLMecke}) to each side in \eqref{sidesB}.

 Let us start with the (easier) case of $\hat \go$, \textit{i.e.} the right-hand side of~\eqref{sidesB}. 
We set 
\[
A_k:= \big\{ (t_i,x_i,u_i)_{i=i}^k\in A^k \ : \ t_1<t_2<\dots<t_k \big\}
\]
and we obtain (recalling the definition of $\bbP'_a$)
\begin{multline*}
 \bbP\otimes \bbP'_a \bigg[  \sum_{(t_i,x_i,u_i)_{i=1}^k\in \; \hat \go^k} f_{k,A} \big( (t_i,x_i,u_i)_{i=i}^k \big) \bigg]\\
 = \int_{A_k} \prod_{i=1}^k\left( \dd t_i \dd x_i  \lambda (\dd u_i )  + \beta \dd t_i \delta_{B_{t_i}}(\dd x_i)  u_i \lambda (\dd u_i )  \right)\, .
\end{multline*}
Now expanding the product and averaging with respect to  the Brownian Motion~$B$ we obtain that the 
right-hand side in \eqref{sidesB} is equal to
\begin{equation}
\label{rhsofsides}
 \int_{A_k} \sum_{I\subset\lint 1,k\rint}\beta^{|I|} u_I \rho_I( \bt,\bx)        \prod_{i=1}^k \dd t_i \dd x_i  \lambda (\dd u_i )
\end{equation}
where $u_I=\prod_{i\in I} u_i$ and  $\rho_I(\bt \bx):= \prod_{j=1}^{|I|} \rho_{t_{j_{i}}- t_{j_{i-1}}}(x_{j_{i}}-x_{j_{i-1}})$ with 
$(i_j)_{j=1}^{|I|}$ the ordered indices of $I$ (by convention $i_0=0$ and $t_0=0$, $x_0=0$).

Now let us move to the left-hand side of~\eqref{sidesB}, that is the expectation with respect to the size-biased measure.
Recalling the definition~\eqref{tildemeasure} of $\tilde \bbP_{\gb}^{a}$ 
and the representation~\eqref{altexp} of $\cZ_{\gb}^{\go,a}$,
it is equal to
\begin{multline}
\label{rhsofsides2}
\tilde \bbE_{\gb}^{a} \bigg[  \sum_{(t_i,x_i,u_i)_{i=i}^k\in \; \hat \go^k} f_{k,A} \big( (t_i,x_i,u_i)_{i=i}^k \big) \bigg] \\  = e^{-\beta \mu}\bbE\bigg[ \sum_{\sigma \in \cP(\go) } w_{a,\beta}(\sigma) \!\!\!\!\sum_{(t_i,x_i,u_i)_{i=i}^k\in\;  \go^k }\!\!\!\!  f_{k,A}\big((t_i,x_i,u_i)_{i=1}^k\big) \bigg]\, .
\end{multline}
We are going to decompose the sum above according to how $\sigma$ intersects with the points $(t_i,x_i,u_i)_{i=1}^k$ that are arguments of $f_{k,A}$.
For any given $I\subset \lint 1,k\rint$  with $|I| =m$ and  $\ell= (\ell_i)_{i=1}^{m+1}$, we set 
\begin{multline*}
 \cP_{I,\ell}=\Big\{ \sigma \in \cP(\go) \ : \ \big(\sigma \cap (t_i,x_i,u_i)_{i=1}^{k}\big)= (t_i,x_i,u_i)_{i\in I}  \\ \text{ and }
 \forall j\in \lint 1, m+1\rint,  \   \#\big\{ (t,x,u) \in \sigma \,, \;   t \in (t_{i_{j-1}},t_{i_{j}}) \big\} =\ell_j \Big\} \,,
\end{multline*}
where $(i_j)_{j=1}^m$ are the ordered elements of $I$,
with $i_0=0$ and $i_{m+1}=k+1$ (and $t_0=0$, $t_{k+1} =1$) by convention.
For this computation we introduce 
$ \bar \kappa_a= \int_{[a,\infty)} \ups \gl(\dd \ups) = \kappa_a+\mu$.
Using  again Mecke's formula and recalling the definition~\eqref{weights} of $w_{a,\gb}(\sigma)$,
we have,
for any such $I$ and $\ell$
\begin{align*}
 \bbE & \bigg[  \sum_{\sigma\in   \cP_{I,\ell} } w_{a,\beta}(\sigma)\sum_{(t_i,x_i,u_i)_{i=i}^k\in\;  \go^k} f_{k,A}((t_i,x_i,u_i)_{i=i}^k) \bigg] \\
& =  e^{-\gb \kappa_a}   \int_{A_k} H^{a}_{\beta}(t_{i_{m}},x_{i_m}, \ell_{m})    
 \prod_{j=1}^{m} 
 G^{a}_{\gb} (t_{i_{j-1}},t_{i_j},x_{i_{j-1}}, x_{i_j}, \ell_j)         \beta^{|I|} u_I  \prod_{i=1}^k \dd t_i \dd x_i  \lambda( \dd u_i ) \, ,
\end{align*}
where we set
\begin{multline*}
  G^{a}_{\beta}(t,t',x,x' ,\ell):=  
 \beta^{\ell} \int_{ t< s_1<\dots<s_{\ell}<t' }  \int_{(\bbR^d)^{\ell}} \int_{(a,\infty)^\ell}
 \left( \prod_{i=1}^{\ell+1} 
   \rho_{\Delta_i s}( \Delta_i y)   \right)
\prod_{i=1}^{\ell}  \dd s_i \dd y_i v_i \lambda(\dd v_i)
 \\ 
 =\left( \gb \bar \kappa_a\right)^{\ell}\frac{(t'-t)^{\ell}}{\ell!}   \rho_{t'-t}(x'-x),
\end{multline*}
with $\Delta_i s = s_{i}-s_{i-1}$ ($s_0=t$, $s_{\ell+1}=t'$) and $\Delta_i y  = y_{i}-y_{i-1}$ ($y_0=x$, $y_{\ell+1}=x'$) by convention;
and we also set
\begin{multline*}
  H^{a}_{\beta}(t,x, \ell):=  
 \beta^{\ell} \int_{ t< s_1<\dots<s_{\ell}<1}  \int_{(\bbR^d)^{\ell}} \int_{(a,\infty)^\ell}  \left( \prod_{i=1}^{\ell} 
   \rho_{\Delta_i s}( \Delta_i y)   \right)
 \prod_{i=1}^{\ell} \dd s_i \dd y_i v_i \lambda(\dd v_i)
 \\ 
 =\left( \gb \bar \kappa_a \right)^{\ell}\frac{(1-t)^\ell}{\ell!}.
\end{multline*}
Replacing $G_\gb^a$ and $H_{\gb}^a$ by their value we obtain (recall that $t_{m+1}=1$  by convention)
\begin{multline*}
 \bbE \bigg[  \sum_{\sigma\in   \cP_{I,\ell} } w_{a,\beta}(\sigma)\sum_{(t_i,x_i,u_i)_{i=i}^k\in  \go^k} f_{k,A}\big( (t_i,x_i,u_i)_{i=i}^k \big) \bigg] \\
 =  e^{- \gb \kappa_a }   \int_{A_k} \prod_{j=1}^{m+1}\left(\gb \bar \kappa_a \right)^{\ell_j}\frac{(t_{j}-t_{j-1})^{\ell_j}}{\ell_j! }         \beta^{|I|} u_I \rho_I( \bt ,\bx)   \prod_{i=1}^k \dd t_i \dd x_i  \lambda( \dd u_i )\, .
\end{multline*}
Summing over all the possible $\ell_j$ just results in a factor $e^{\beta \bar \kappa_a} =e^{\gb \mu} e^{\gb \kappa_a}$ and thus after summing over~$I$
we obtain
\[
\tilde \bbE_{\gb}^{a} \bigg[  \sum_{(t_i,x_i,u_i)_{i=i}^k\in \; \hat \go^k} f_{k,A} \big( (t_i,x_i,u_i)_{i=i}^k \big) \bigg]\\ = e^{\beta \mu} \sum_{I\subset \lint 1,k\rint} 
   \int_{A_k}        \beta^{|I|} u_I \rho_I( \bt,\bx)  \prod_{i=1}^k \dd t_i \dd x_i  \lambda(\dd u_i ) \, .
\]
All together, we find that \eqref{rhsofsides2} is equal to~\eqref{rhsofsides}. This proves~\eqref{sidesB},
which concludes the proof of Lemma~\ref{spayne}.
\end{proof}

\section{Stochastic comparisons}

We provide here
two results enabling us to compare some integrals 
with respect to the measure $\lambda$ to integrals with respect
to the Lebesgue measure.
In particular, they establish the two claims~\eqref{plusdidee}-\eqref{plusdidee2}.

\begin{proposition}
\label{prop:compare1}
Assume that $\mu:=\int_{[1,\infty)} \ups \lambda(\dd \ups) <\infty$ and also that  $\int_{(0,1)} \ups^p \lambda(\dd \ups) <\infty$
for some $p \in (1,2)$. Then for $q\geq 1$ there is
a constant $c_q$, verifying $\lim\limits_{q\to\infty} c_q q^{1-p} =0$,
such that for every $m\geq 1$ and any non-decreasing function $g:\bbR^m \to \bbR_+$,  \textit{i.e.}\ non-decreasing in every coordinate, with $\supp(g)\subset (\gep ,\infty)^m$ for some $\gep>0$, we have
\[
\int_{(0,q)^m} g(u_1 , \ldots, u_m) \prod_{i=1}^m u_i \lambda(\dd u_i) \leq  (c_q)^m \int_{(0,2q)^m} g(u_1 , \ldots, u_m) \prod_{i=1}^m u_i^{-p}  \dd u_i \, .
\]
\end{proposition}

\begin{proof}
Let us begin with a few observations.
First, we only need to treat the case $m=1$ since
applying the result successively to the functions $u_i\mapsto f(u_1, \ldots, u_m)$ concludes the proof.
Second, we can work with a differentiable and bounded
function $g$, the general case being obtained by monotone convergence.

By an integration by part, defining $\bar \mu_q(u) := \int_{(u,q)} \ups \lambda(\dd \ups)$,
we get that
\begin{equation}
\label{IPP1}
\int_{(0,q)} g(u) u \lambda(\dd u) = \int_{[\gep,q)} 
g(u) u \lambda(\dd u) = \int_{ [\gep,q)} g'(u) \bar \mu_q(u) \dd u \, .
\end{equation}

Now, an important observation is that under our assumptions we have that
\begin{equation}
\label{boundbarmu}
\bar \mu_q(u) := \int_{(u,q)} \ups \lambda(\dd \ups) \leq c_q u^{1-p} \, , \qquad  \forall \,  u  \in (0,q)  \, ,
\end{equation}
for a constant $c_q$ verifying $\lim_{q\to\infty} c_q q^{1-p} =0$.
Let us postpone the proof of~\eqref{boundbarmu},
but we can already see that plugged in~\eqref{IPP1} and using that $g'(u) \geq 0$, it implies that
\[
\int_{(0,q)} g(u) u \lambda(\dd u)  
\leq  c_q  \int_{ [\gep,q)} g'(u)  u^{1-p} \dd u = 
c_q g(q) q^{1-p} + (p-1) c_q \int_{[\gep,q)}  g(u)  u^{ -p} \dd u
\]
where we have used another integration by parts in the last identity.
Using again that $g$ is non-decreasing, we get that
$ \int_{[q,2q)} g(u) u^{-p} \dd u \geq \frac{1}{p-1} (1-2^{1-p}) g(q) q^{1-p}$,
and we therefore end up with
\[
\int_{(0,q)} g(u) u \lambda(\dd u)  \leq c_p \,  c_q \int_{(0,2q)}  g(u)  u^{ -p} \dd u\,,
\]
where the constant $c_p$ only depends on $p$.
This concludes the proof.

It remains to see why~\eqref{boundbarmu} is true.
We consider the cases $u<1$ and $u\geq 1$ separately.
If $u < 1$, we use the fact that $c'_p:=\int_{(0,1)} \ups^{p} \lambda(\dd \ups) <+\infty$ to get that
\[
\bar \mu_q(u)  
\leq \int_{(u,1)} \ups^{1-p} \ups^p \lambda(\dd \ups) +\int_{[1,\infty)} \ups \lambda(\dd \ups) 
\leq c'_p u^{1-p}  + \mu \leq c''_p   u^{1-p}\, ,
\]
since $p>1$.
If $u\geq 1$, we simply use that
\[
\bar \mu_q(u)  
\leq  c''_q u^{1-p} \, \quad \text{ with } c''_q := \sup_{u\in [1,q)} \bar \mu_q(u) /u^{1-p} 
\]
and notice that since $ \bar \mu_q(u)$ is non-increasing and goes to $0$ as $u\to\infty$
this implies that $\lim_{q\to\infty} q^{1-p} c''_q =0$.
Combining the above estimates gives~\eqref{boundbarmu}.
\end{proof}

\begin{proposition}
\label{prop:compare2}
Assume that  $\int_{(0,1)} \ups^p \lambda(\dd \ups) <\infty$
for some $p \in (1,2)$. Then for any $q\geq 1$ there is
a constant $C_q$, 
such that for every $m\geq 1$ and any non-increasing function $g:\bbR^m \to \bbR_+$ we have
\[
\int_{(0,q)^m} g(u_1 , \ldots, u_m) \prod_{i=1}^m u_i^2 \lambda(\dd u_i) \leq  (C_q)^m \int_{(0,q)^m} g(u_1 , \ldots, u_m) \prod_{i=1}^m u_i^{1-p}  \dd u_i \, .
\]
\end{proposition}

\begin{proof}
The proof is similar to that of Proposition~\ref{prop:compare1}
above.
Again, we only have to treat the case $m=1$ and of a bounded and differentiable function $g$, with $\|g\|_{\infty} \leq 1$ to simplify.

Setting $F(u):= \int_{(0,u]} \ups^2 \lambda(\dd \ups)$ (which is finite for any $u\geq 0$),
an integration by parts gives that
\begin{align*}
\int_{(0,q)} g(u)  u^2 \lambda(\dd u) =  g(q)  F(q) -  \int_{(0,q)} g'(u) F(u) \dd u\, .
\end{align*}
Now, notice that there is a constant $C_q := \int_{(0,q]} \ups^{p} \lambda(\dd\ups)<\infty$ such that
\begin{equation}
F(u)= \int_{(0,u]} \ups^{2-p} \ups^p \lambda(\dd \ups) \leq C_q u^{2-p}
\, , \qquad \forall\, u\in (0,q]\,.
\end{equation}
Using that $g'(u)\leq 0$, we therefore get that
\begin{align*}
\int_{(0,q)} g(u)  u^2 \lambda(\dd u) \leq  C_q  g(q) q^{2-p} -  C_q \int_{(0,q)} g'(u) u^{2-p} \dd u 
= (2-p)  C_q \int_{(0,q)} g(u) u^{1-p} \dd u \, ,
\end{align*}
where we used another integration by parts 
for the last identity.
\end{proof}

\noindent
{\bf Acknowledgement.}
The authors are gratefull to Carsten Chong for insightful comment on a first draft of the paper (see in particular Remark \ref{chongsremark}) as well as for indicating to us relevant references concerning stochastic PDEs with L\'evy noise. 
This work was written during H.L.\ extended stay in Aix-Marseille University funded by the European Union’s Horizon 2020 research and innovation programme under the Marie Skłodowska-Curie grant agreement No 837793. Q.B. acknowledges the support of ANR grant SWiWS (ANR-17-CE40-0032-0).

\bibliographystyle{plain}
\bibliography{biblio.bib}

\begin{thebibliography}{10}

\bibitem{AKQ10}
Tom Alberts, Konstantin Khanin, and Jeremy Quastel.
\newblock Intermediate disorder regime for directed polymers in dimension
  $1+1$.
\newblock {\em Phys. Rev. Letters}, 105(9):090603, 2010.

\bibitem{AKQ14b}
Tom Alberts, Konstantin Khanin, and Jeremy Quastel.
\newblock The continuum directed random polymer.
\newblock {\em J. Stat. Phys.}, 154:305--326, 2014.

\bibitem{AY15}
Kenneth Alexander and G{\"o}khan Yıldırım.
\newblock Directed polymers in a random environment with a defect line.
\newblock {\em Electron. J. Probab.}, 20:20 pp., 2015.

\bibitem{AL11}
Antonio Auffinger and Oren Louidor.
\newblock Directed polymers in a random environment with heavy tails.
\newblock {\em Commun. Pure Appl. Math.}, 64(2):183--204, 2011.

\bibitem{ber2017}
Nathana{\"e}l Berestycki.
\newblock An elementary approach to gaussian multiplicative chaos.
\newblock {\em Electron. Commun. Probab.}, 22:12 pp., 2017.

\bibitem{BCL21}
Quentin Berger, Carsten Chong, and Hubert Lacoin.
\newblock The stochastic heat equation with multiplicative {L}\'evy noise:
  Existence, moments, and intermittency.
\newblock {\em preprint arXiv:2111.07988 [math.PR]}, 2021.

\bibitem{BL17}
Quentin Berger and Hubert Lacoin.
\newblock The high-temperature behavior for the directed polymer in dimension
  $1+2$.
\newblock {\em Ann. Inst. Henri Poincar{\'e}, Probab. Stat.}, 53(1):430--450,
  02 2017.

\bibitem{BL18}
Quentin Berger and Hubert Lacoin.
\newblock Pinning on a defect line: characterization of marginal relevance and
  sharp asymptotics for the critical point shift.
\newblock {\em J. Inst. Math. Jussieu}, 17(2):305--346, 2018.

\bibitem{BL20_disc}
Quentin Berger and Hubert Lacoin.
\newblock The scaling limit of the directed polymer with power-law tail
  disorder.
\newblock {\em Comm. Math. Phys.}, 386(2):1051--1105, 2021.

\bibitem{BT19}
Quentin Berger and Niccol\`o Torri.
\newblock Directed polymers in heavy-tail random environment.
\newblock {\em Ann. Probab.}, 47(6):4024--4076, 2019.

\bibitem{CB95}
Lorenzo Bertini and Nicoletta Cancrini.
\newblock The {S}tochastic {H}eat {E}quation: {F}eynman-{K}ac formula and
  intermittence.
\newblock {\em J. Stat. Phys.}, 78(5-6):1377--1401, 1995.

\bibitem{Bir04}
Matthias Birkner.
\newblock A condition for weak disorder for directed polymers in random
  environment.
\newblock {\em Electron. Comm. Probab.}, 9:22--25, 2004.

\bibitem{Bol89}
Erwin Bolthausen.
\newblock A note on the diffusion of directed polymers in a random environment.
\newblock {\em Commun. Math. Phys.}, 123(4):529--534, 1989.

\bibitem{BdH97}
Erwin Bolthausen and Frank den Hollander.
\newblock Localization transition for a polymer near an interface.
\newblock {\em Ann. Probab.}, 25(3):1334--1366, 1997.

\bibitem{Bov06}
Anton Bovier.
\newblock {\em Statistical mechanics of disordered systems: a mathematical
  perspective}, volume~18.
\newblock Cambridge University Press, 2006.

\bibitem{CG10}
Francesco Caravenna and Giambattista Giacomin.
\newblock The weak coupling limit of disordered copolymer models.
\newblock {\em Ann. Probab.}, 38(6):2322--2378, 2010.

\bibitem{CSZ14}
Francesco Caravenna, Rongfeng Sun, and Nikos Zygouras.
\newblock The continuum disordered pinning model.
\newblock {\em Probab. Theory Relat. Fields}, 164:17--59, 2016.

\bibitem{CSZ13}
Francesco Caravenna, Rongfeng Sun, and Nikos Zygouras.
\newblock Polynomial chaos and scaling limits of disordered systems.
\newblock {\em J. {EMS}}, 19:1--65, 2017.

\bibitem{CSZ15}
Francesco Caravenna, Rongfeng Sun, and Nikos Zygouras.
\newblock Universality in marginally relevant disordered systems.
\newblock {\em Ann. Appl. Probab.}, 27(5):3050--3112, 2017.

\bibitem{CSZ18scaling}
Francesco Caravenna, Rongfeng Sun, and Nikos Zygouras.
\newblock On the moments of the $(2+1)$-dimensional directed polymer and
  stochastic heat equation in the critical window.
\newblock {\em Cummun. Math. Phys.}, 372(2):385--440, 2019.

\bibitem{CSZ21}
Francesco Caravenna, Rongfeng Sun, and Nikos Zygouras.
\newblock The critical 2d stochastic heat flow.
\newblock {\em arXiv:2109.03766}, 2021.

\bibitem{CSZreview}
Francesco Caravenna, Rongfeng Sun, and Nikos Zygouras.
\newblock Scaling limits of disordered systems and disorder relevance.
\newblock In {\em Proceedings for XVIII International Congress on Mathematical
  Physics, Santiago de Chile, 2015}, to appear.

\bibitem{Chong17}
Carsten Chong.
\newblock Stochastic {PDE}s with heavy-tailed noise.
\newblock {\em Stoch. Process. Appl.}, 127(7):2262--2280, 2017.

\bibitem{clark2019}
Jeremy Clark.
\newblock Weak-disorder limit at criticality for directed polymers on
  hierarchical graphs.
\newblock {\em Commun. Math. Phys}, 386:651--710, 2021.

\bibitem{Comets07}
Francis Comets.
\newblock Weak disorder for low dimensional polymers: the model of stable laws.
\newblock {\em Markov Process. Relat. Fields}, 13(4):681--696, 2007.

\bibitem{C17}
Francis Comets.
\newblock {\em Directed Polymers in Random Environments}, volume 2175 of {\em
  \'Ecole d'Et{\'e} de probabilit{\'e}s de {S}aint-{F}lour}.
\newblock Springer International Publishing, 2016.

\bibitem{CSY03}
Francis Comets, Tokuzo Shiga, and Nobuo Yoshida.
\newblock Directed polymers in a random environment: strong disorder and path
  localization.
\newblock {\em Bernoulli}, 9(4):705--723, 2003.

\bibitem{CV06}
Francis Comets and Vincent Vargas.
\newblock Majorizing multiplicative cascades for directed polymers in random
  media.
\newblock {\em ALEA, Lat. Am. J. Probab. Math. Stat.}, 2:267--277, 2006.

\bibitem{CSY06}
Francis Comets and Nobuo Yoshida.
\newblock Directed polymers in random environment are diffusive at weak
  disorder.
\newblock {\em Ann. Probab.}, 34:1746--1770, 2006.

\bibitem{DZ16}
Partha~S. Dey and Nikos Zygouras.
\newblock High temperature limits for $(1+ 1)$-dimensional directed polymer
  with heavy-tailed disorder.
\newblock {\em Ann. Probab.}, 44(6):4006--4048, 2016.

\bibitem{FFU17}
Julien Fageot, Michael Unser, and John~Paul Ward.
\newblock On the {B}esov regularity of periodic {L}{\'e}vy noises.
\newblock {\em Applied and Computational Harmonic Analysis}, 42(1):21 -- 36,
  2017.

\bibitem{GB07}
Giambattista Giacomin.
\newblock {\em Random Polymer Models}.
\newblock Imperial College Press, World Scientific, 2007.

\bibitem{GB10}
Giambattista Giacomin.
\newblock {\em Disorder and Critical Phenomena Through Basic Probability
  Models}, volume 2025 of {\em {\'E}cole d'\'Et{\'e} de probabilit{\'e}s de
  {S}aint-{F}lour}.
\newblock Springer-Verlag Berlin Heidelberg, 2010.

\bibitem{GLT11}
Giambattista Giacomin, Hubert Lacoin, and Fabio~L. Toninelli.
\newblock Disorder relevance at marginality and critical point shift.
\newblock {\em Ann. Inst. Henri Poincar{\'e}, Probab. Stat.}, 47(1):148--175,
  2011.

\bibitem{gu2019}
Yu~Gu, Jeremy Quastel, and Li-Cheng Tsai.
\newblock Moments of the 2d she at criticality.
\newblock {\em Probab. Math. Phys.}, 2(1):179--219, 2021.

\bibitem{HH85}
David~A. Huse and Christopher~L. Henley.
\newblock Pinning and roughening of domain walls in {I}sing systems due to
  random impurities.
\newblock {\em Phys. Rev. Letters}, 54:2708--2711, 1985.

\bibitem{IS88}
John~Z. Imbrie and Thomas Spencer.
\newblock Diffusion of directed polymers in a random environment.
\newblock {\em J. Stat. Phys.}, 52:608--626, 1988.

\bibitem{K70}
N~Kalinauskaite.
\newblock Certain expansions of multidimensional symmetric stable densities.
\newblock {\em Litov. Mat. Sb}, 10:727--731, 1970.
\newblock in Russian.

\bibitem{Lac10pol}
Hubert Lacoin.
\newblock New bounds for the free energy of directed polymer in dimension $1+1$
  and $1+2$.
\newblock {\em Commun. Math. Phys.}, 294:471--503, 2010.

\bibitem{LS17}
Hubert Lacoin and Julien Sohier.
\newblock Disorder relevance without {H}arris criterion: the case of pinning
  model with {$\gamma$}-stable environment.
\newblock {\em Electron. J. Probab.}, 22:26, 2017.

\bibitem{PoiBook}
G\"{u}nter Last and Mathew Penrose.
\newblock {\em Lectures on the {P}oisson process}, volume~7 of {\em Institute
  of Mathematical Statistics Textbooks}.
\newblock Cambridge University Press, Cambridge, 2018.

\bibitem{MCMT}
David~A. Levin and Yuval Peres.
\newblock {\em Markov Chains and Mixing Times (Second Edition)}.
\newblock MBK. American Mathematical Society, 2017.

\bibitem{Mueller98}
Carl Mueller.
\newblock The heat equation with {L}\'{e}vy noise.
\newblock {\em Stochastic Process. Appl.}, 74(1):67--82, 1998.

\bibitem{Nak16}
Makoto Nakashima.
\newblock Free energy of directed polymers in random environment in {$1+1$}
  dimension at high temperature.
\newblock {\em Electron. J. Probab.}, 24(50):43pp.

\bibitem{Loubert98}
Erwan Saint Loubert~Bi\'{e}.
\newblock \'{E}tude d'une {EDPS} conduite par un bruit poissonnien.
\newblock {\em Probab. Theory Relat. Fields}, 111(2):287--321, 1998.

\bibitem{ST94}
Gennady Samorodnitsky and Murad Taqqu.
\newblock {\em Stable Non-Gaussian Random Processes: Stochastic Models with
  Infinite Variance}.
\newblock Chapman and Hall/CRC, 1994.

\bibitem{ShiBRW}
Zhan Shi.
\newblock {\em Branching random walks}, volume 2151 of {\em Lecture Notes in
  Mathematics}.
\newblock Springer, 2015.
\newblock Lecture notes from the 42nd Probability Summer School held in Saint
  Flour, 2012.

\bibitem{Vi19}
Roberto Viveros.
\newblock Directed polymer in $\gamma$-stable random environments.
\newblock {\em Ann. Instit. Henri Poincar{\'e}, Probab. Stat.},
  57(2):1081--1102, 2021.

\bibitem{Wei16}
Ran Wei.
\newblock On the long-range directed polymer model.
\newblock {\em J. Stat. Phys.}, 165(2):320--350, 2016.

\end{thebibliography}

\end{document}